\documentclass{amsart}
\usepackage{verbatim}
\usepackage{amssymb}
\usepackage{amsmath}
\usepackage[usenames]{color}
\usepackage{graphicx}
\usepackage{amscd}
\usepackage[numbers,sort,square]{natbib}
\usepackage[T1]{fontenc}

\usepackage{enumerate}

\theoremstyle{plain}
\newtheorem{theorem}{Theorem}

\newtheorem{proposition}[theorem]{Proposition}
\newtheorem{lemma}[theorem]{Lemma}
\newtheorem{corollary}[theorem]{Corollary}

\newtheorem{definition}[theorem]{Definition}

\theoremstyle{definition}
\newtheorem{rem}[theorem]{Remark}

\newtheorem{remark}[theorem]{Remark}

\numberwithin{equation}{section}
\numberwithin{theorem}{section}
\allowdisplaybreaks

\usepackage[colorinlistoftodos,prependcaption,color=yellow,textsize=tiny]{todonotes}

\def\R{{\mathbb R}}
\def\C{{\mathbb C}}

\newcommand{\E}{{\mathbb E}}
\renewcommand{\P}{{\mathbb P}}

\newcommand{\F}{{\mathcal F}}

\newcommand{\eps}{\varepsilon}
\newcommand{\e}{\varepsilon}

\newcommand{\sign}{\mathrm{sign}\text{ }}

\newcommand{\dd}[0]{\mathrm{d}}
\newcommand{\ud}[0]{\,\mathrm{d}}

\newcommand{\vertiii}[1]{{\left\vert\kern-0.25ex\left\vert\kern-0.25ex\left\vert #1
    \right\vert\kern-0.25ex\right\vert\kern-0.25ex\right\vert}}

\begin{document}

\title[The Hilbert transform and orthogonal martingales]
{The Hilbert transform and orthogonal martingales in Banach spaces}

\author{Adam Os\k{e}kowski}
\address{Department of Mathematics, Informatics and Mechanics\\
University of Warsaw\\
Banacha 2 \\ 02-097 Warsaw\\
Poland}
\email{ados@mimuw.edu.pl}

\author{Ivan Yaroslavtsev}
\address{Delft Institute of Applied Mathematics\\
Delft University of Technology \\ P.O. Box 5031\\ 2600 GA Delft\\The
Netherlands}
\email{yaroslavtsev.i.s@yandex.ru}

\begin{abstract}
Let $X$ be a given Banach space and let $M$, $N$ be two orthogonal $X$-valued local martingales such that $N$ is weakly differentially subordinate to $M$. The paper contains the proof of the estimate
$$
\mathbb E \Psi(N_t) \leq C_{\Phi,\Psi,X} \mathbb E \Phi(M_t),\;\;\; t\geq 0,
$$
where $\Phi, \Psi:X \to \mathbb R_+$ are convex continuous functions and the least admissible constant $C_{\Phi,\Psi,X}$ coincides with the $\Phi,\Psi$-norm of the periodic Hilbert transform. As a corollary, it is shown that the $\Phi,\Psi$-norms of the periodic Hilbert transform, the Hilbert transform on the real line, and the discrete Hilbert transform are the same if $\Phi$ is symmetric.
We also prove that under certain natural assumptions on $\Phi$ and $\Psi$, the condition $C_{\Phi,\Psi,X}<\infty$ yields the UMD property of the space $X$.
As an application, we provide comparison of $L^p$-norms of the periodic Hilbert transform to Wiener and Paley-Walsh decoupling constants. We also study the norms of the periodic, nonperiodic and discrete Hilbert transforms, present the corresponding estimates in the context of differentially subordinate harmonic functions and more general singular integral operators.
\end{abstract}

\keywords{Weak differential subordination, orthogonal martingales, periodic Hilbert transform, UMD spaces, plurisubharmonic functions, decoupling constants, discrete Hilbert transform, harmonic functions, Riesz transform}

\subjclass[2010]{44A15, 60G44 Secondary: 60B11, 31C10, 31B05, 46B09}

\maketitle

\tableofcontents

\section{Introduction}
The purpose of this paper is to study a certain class of estimates for singular integral operators acting on Banach-space-valued functions. Let us start with a related classical problem which has served as a motivation for many mathematicians for almost a century. The question is: how does the size of a periodic function control the size of its conjugate? Formally, assume that $f$ is a trigonometric polynomial of the form
$$ f(\theta)=\frac{a_0}{2}+\sum_{k=1}^N \big(a_k\cos(k\theta)+b_k\sin(k\theta)\big),\qquad \theta\in \mathbb{T}\simeq [-\pi,\pi),$$
with real coefficients $a_0$, $a_1$, $a_2$, $\ldots$, $a_N$, $b_1$, $b_2$, $\ldots$, $b_N$, and define the conjugate to $f$ as
$$ g(\theta)=\sum_{k=1}^N \big(a_k\sin(k\theta)-b_k\cos(k\theta)\big),\qquad \theta\in [-\pi,\pi).$$
Alternatively, the conjugate function can be defined as $g=\mathcal{H}^\mathbb{T}_{\mathbb R}f$, where $\mathcal{H}^\mathbb{T}_\R$ is the periodic Hilbert transform given by
\begin{equation}\label{eq:periodic0}
\mathcal{H}^{\mathbb T}_\R f(\theta)=\frac{1}{2\pi}\mbox{p.v.}\int_{-\pi}^\pi f(s)\,{\cot\frac{\theta-s}{2}}\mbox{d}s,\qquad \theta\in [-\pi,\pi),
\end{equation}
and the symbol $\R$ in the lower index of $\mathcal{H}^\mathbb{T}$ indicates that the operator acts on real-valued functions.
We can state the problem as follows. For a given $1\leq p\leq \infty$, does there exist a universal constant $C_p$ (that is, not depending on the coefficients or the number $N$) such that
$$
 \left(\int_{[-\pi,\pi)}|g(\theta)|^p\mbox{d}\theta\right)^{1/p}\leq C_p\left(\int_{[-\pi,\pi)}|f(\theta)|^p\mbox{d}\theta\right)^{1/p}\,?
$$
Furthermore, if the answer is yes, what is the optimal value of $C_p$ (i.e., what is the $L^p$-norm of $\mathcal{H}^\mathbb{T}_\R$)? The first question was answered by M. Riesz in \cite{Riesz28}: the inequality does hold if and only if $1<p<\infty$. The best value of $C_p$ was determined by Pichorides \cite{Pic72} and Cole (unpublished): the constant $\cot(\pi/(2p^*))$ is the best possible, where $p^*=\max\{p,p/(p-1)\}$.
There is a natural further question concerning the version of the above result for Banach-space-valued functions (it is not difficult to see that the formula \eqref{eq:periodic0} makes perfect sense in the vector setting, at least for some special $f$, see Section 2 below). Few years after the results of Riesz, it was realized that not all spaces are well-behaved: Bochner and Taylor \cite{BoTa} showed that $||\mathcal{H}^\mathbb{T}_{\ell_1}||_{L_p\to L_p}=\infty$ for all $p$. The problem of characterizing the `good' Banach spaces was solved over forty years later: Burkholder \cite{Burk81} and Bourgain \cite{Bour83} showed that the so-called UMD\footnote{UMD stands for ``unconditional martingale differences''} spaces form a natural environment to the study of the $L^p$-boundedness ($1<p<\infty$) of the periodic Hilbert transform, and more generally, for the $L^p$-boundedness of a wider class of singular integral operators.

The above problems, though expressed in an analytic language, have a very strong connection with probability theory, especially with the theory of martingales (see e.g.\ \cite{Burk83,Bour83,BanWang95,BanWang96,Gar85,HNVW1,BanKw17,Os15a,Os17a,GM-SS}). Let us provide some necessary definitions. Suppose that $(\Omega,\F,\mathbb{P})$ is a complete probability space, equipped with a continuous-time filtration $(\F_t)_{t\geq 0}$. Let $M=(M_t)_{t\geq 0}$, $N=(N_t)_{t\geq 0}$ be two adapted real-valued local martingales, whose trajectories are right-continuous and have limits from the left. Let $[M]$, $[N]$ stand for the associated quadratic variation (square brackets) of $M$ and $N$, see \cite{DM82} and \eqref{eq:defquadvar} below. Furthermore, $M^*=\sup_{t\geq 0}|M_t|$, $N^*=\sup_{t\geq 0}|N_t|$ denote the corresponding maximal functions.  Following Ba\~nuelos and Wang \cite{BanWang95} and Wang \cite{Wang}, $N$ is {\em differentially subordinate} to $M$ (which we denote by $N \ll M$) if, with probability $1$, the process $t\mapsto [M]_t-[N]_t$ is a nondecreasing and nonnegative function of $t\geq 0$. Furthermore, we say that $M$ and $N$ are \emph{orthogonal}, if $[M,N]:= \frac{[M+N] - [M-N]}{4}=0$ 
almost surely.

One of the remarkable examples of the aforementioned connection between the theory of singular integral operators and martingale theory was provided by Ba\~nuelos and Wang in \cite{BanWang95}. They have shown that the $L^p$-norm of $\mathcal{H}^\mathbb{T}$ acting on real-valued functions is equal to the sharp constant in the corresponding $L^p$-inequality
\begin{equation}\label{eq:BanWang95DS+Orth}
 (\mathbb E|N_t|^p)^{\frac{1}{p}} \leq C_p (\mathbb E|M_t|^p)^{\frac{1}{p}},\;\;\; t\geq 0,
\end{equation}
where $N$ is assumed to be differentially subordinate and orthogonal to $M$. The goal of the current article is to show that this interplay between the norm of $\mathcal{H}^\mathbb{T}$ and the martingale inequality \eqref{eq:BanWang95DS+Orth} can be extended to i) more general $\Phi,\Psi$-norms (see the beginning of Section 3 for the definition) and ii) more general Banach spaces in which the functions and processes take values.

\smallskip

Let us say a few words about the structure of the paper. The next Section is devoted to the introduction of the background which is needed for our further study. In particular, we recall there the notion of UMD spaces, define appropriate analogues of Banach-space-valued differential subordination and orthogonality, formulate the vector extensions of stochastic calculus and provide some basic information about plurisubharmonic functions, fundamental objects in the complex analysis of several variables.  Section 3 contains the main result of the paper, connecting the best constants in certain $\Phi,\Psi$-estimates for the periodic Hilbert transform and their counterparts in martingale theory. Though the rough idea of the proof can be tracked back to the classical works \cite{Burk87,BanWang95,HKV03,Pic72} (the validity of a given estimate for the Hilbert transform / orthogonal differentially subordinate martingales is equivalent to the existence of a certain special plurisubharmonic function), there are several serious technical problems to be overcome, due to the fact that we work in the Banach-space-valued setting.
Section \ref{sec:applications} is devoted to some applications. The first and the most notable one connects together the $\Phi,\Psi$-norms of the periodic Hilbert transform $\mathcal H_X^{\mathbb T}$, the Hilbert transform $\mathcal H^{\mathbb R}_X$ defined on a real line, and the discrete Hilbert transform $\mathcal H_{X}^{\rm dis}$ (for the definition of the latter object, consult Definition \ref{def:defofRHT} and \ref{def:defofdisHT} below). It turns out that all these norms coincide for quite general class of $\Phi$ and $\Psi$. This in particular generalizes the recent result of Ba{\~n}uelos and Kwa{\'s}nicki \cite{BanKw17} on the discrete Hilbert transform $\mathcal H^{\rm dis}_{\mathbb R}$, which asserts that
\[
 \|\mathcal H^{\rm dis}_{\mathbb R}\|_{L^p(\mathbb Z)\to L^p(\mathbb{Z})} =  \|\mathcal H^{\mathbb T}_{\mathbb R}\|_{L^p(\mathbb T)\to L^p(\mathbb{T})} =  \cot\Bigl(\frac{\pi}{2p^*}\Bigr),\qquad 1<p<\infty.
\]
This used to be an open problem for $90$ years (see \cite{BanKw17,Lae07,Ti26}). Subsection \ref{sebsec:decconstants} is devoted to the comparison of $L^p$-norms of the periodic Hilbert transform to Wiener and Paley-Walsh decoupling constants.
Application in Subsection \ref{sebsec:NecofUMD} is concerned with UMD Banach spaces and can be regarded as an extension of Bourgain's result \cite{Bour83}: we show that under some mild assumption on $\Phi$ and $\Psi$, the validity of the corresponding $\Phi,\Psi$-estimate (with some finite constant) implies the UMD property of $X$. In Subsection \ref{subsec:sharperLpforWDS} we prove that the results obtained in this paper can be applied to obtain sharper estimates for weakly differentially subordinate martingales (not necessarily satisfying the orthogonality assumption). Subsection \ref{subsec:WDSharmfunc} contains the study of related estimates in the context of harmonic functions on Euclidean domains. In Subsection \ref{subsec:singint} we present the possibility of extending the estimates to the more general class of singular integral operators. Our final application, described in Subsection \ref{sec:Hilbert operators}, discusses the vector-valued extension of the classical results of Hardy concerning Hilbert operators.

\section{Preliminaries}
This section contains the definitions of some basic notions and facts used later. Here and below, the scalar field is assumed to be $\mathbb R$, unless stated otherwise. {In particular, all Banach spaces are real, unless stated otherwise.}

\subsection{Periodic Hilbert transform}\label{subsec:PHT}
In what follows, the symbol $\mathbb{T}$ will stand for the torus $(\{z\in \mathbb C:|z| = 1\},\cdot)$ equipped with the natural multiplication. Sometimes, passing to the argument of a complex number, we will identify $\mathbb{T}$ with the interval $[-\pi,\pi)$. Let $X$ be a Banach space. A function $f:\mathbb{T}\to X$ is called a {\em step function}, if it is of the form
$$
f=\sum_{k=1}^N x_k\mathbf 1_{A_k}(s),\;\;\; -\pi\leq s<\pi,
$$
where $N$ is finite, $x_k\in X$ and $A_k$ are intervals in $\mathbb T$. The periodic Hilbert transform $\mathcal{H}_X^{\mathbb{T}}$ of a step function $f:\mathbb{T} \to X$ is given by the singular integral
\begin{equation}\label{eq:periodic}
\mathcal{H}^{\mathbb T}_Xf(t)=\frac{1}{2\pi}\mbox{p.v.}\int_{-\pi}^\pi f(s)\,{\cot\frac{t-s}{2}}\mbox{d}s,\qquad -\pi\leq t< \pi.
\end{equation}

\subsection{UMD Banach spaces}\label{subsec:UMD} Suppose that $(\Omega,\F,\mathbb{P})$ is a nonatomic probability space.
A Banach space $X$ is called a {\em UMD space} if for some (or equivalently, for
all)
$p \in (1,\infty)$ there exists a finite constant $\beta$ such that the following holds. If $(d_n)_{n=1}^\infty$ is any {$X$-valued} martingale difference sequence (relative to some discrete-time filtration) contained in $L^p(\Omega; X)$ and
$(\varepsilon_n)_{n=1}^\infty$ is any deterministic sequence of signs, then
\[
\Bigl(\mathbb E \Bigl\| \sum^N_{n=1} \varepsilon_n d_n\Bigr\|^p\Bigr )^{\frac
1p}
\leq \beta \Bigl(\mathbb E \Bigl \| \sum^N_{n=1}d_n\Bigr\|^p\Bigr )^{\frac 1p}.
\]
The least admissible constant $\beta$ above is denoted by $\beta_{p,X}$ and is called
the {\em UMD constant} of $X$.
It is well-known that UMD spaces enjoy a large number
of useful properties, such as being reflexive. Examples of UMD
spaces include all finite dimensional spaces, {Hilbert spaces (then $\beta_{p,X} = p^*-1$ with $p^* = \max\{p, p/(p-1)\}$),} the reflexive range of
$L^q$-spaces, Sobolev spaces, Schatten class spaces, and Orlicz spaces. On the other hand, all nonreflexive Banach spaces, e.g.
$L^1(0,1)$ and $C([0,1])$, are not UMD. We refer the reader to
\cite{Burk01,HNVW1,Pis16} for further details.

\begin{remark}\label{rem:h_p,Xvsbeta_p,X}
As we have already mentioned in the introductory section, UMD Banach spaces form a natural environment for the $L^p$-boundedness of the periodic Hilbert transform. It follows from \cite{Burk83,Bour83} that for every $1<p<\infty$ we have
 \begin{equation}\label{eq:sqrtbetaleqhleqbeta^2}
   \sqrt{\beta_{p,X}} \leq \|\mathcal H_X^{\mathbb T}\|_{L^p(\mathbb T, X)\to L^p(\mathbb{T},X)} \leq \beta_{p, X}^2.
 \end{equation}
It is not known whether the quadratic dependence can be improved on either of the sides (see e.g.\ \cite{HNVW1,GM-SS,Burk01}). Notice that if $X = \mathbb R$, then the dependence becomes linear: indeed,
\[
 \frac{2}{\pi}\beta_{p,\mathbb R} = \frac{2}{\pi}(p^*-1) \leq \cot\Bigl(\frac{\pi}{2p^*}\Bigr) =  \|\mathcal H_X^{\mathbb T}\|_{\mathcal L(L^p(\mathbb T, X))} \leq p^*-1 = \beta_{p,\mathbb R},
\]
where, as above, $p^*:=\max\{p,p/(p-1)\}$.
\end{remark}

We will see later that the context of UMD is also natural from the viewpoint of more general $\Phi,\Psi$-estimates for the periodic Hilbert transform {(see Subsection \ref{sebsec:NecofUMD})}.

\subsection{Stochastic integration {and It\^o's formula}}
For given Banach spaces $X,\,Y$, the symbol $\mathcal{L}(X,Y)$ will denote the classes of all linear operators from $X$ to $Y$. We will also use the notation $\mathcal{L}(X)=\mathcal{L}(X,X)$. Suppose that $H$ is a Hilbert space. For each $h\in H$ and $x\in X$, we denote by $h\otimes x$ the associated linear operator given by $g\mapsto \langle g, h\rangle x$, $g\in H$. The process $\phi: \mathbb R_+ \times \Omega \to \mathcal L(H,X)$ is called  \textit{elementary progressive}
with respect to the filtration $\mathbb F = (\mathcal F_t)_{t \geq 0}$ if it is of the form
\begin{equation*}
 \phi(t,\omega) = \sum_{k=1}^K\sum_{m=1}^M \mathbf 1_{(t_{k-1},t_k]\times B_{mk}}(t,\omega)
\sum_{n=1}^N h_n \otimes x_{kmn},\;\;\; t\geq 0, \omega \in \Omega.
\end{equation*}
Here $0 \leq t_0 < \ldots < t_K <\infty$ is a finite increasing sequence of nonegative numbers, the sets
$B_{1k},\ldots,B_{Mk}$ belong to $\mathcal F_{t_{k-1}}$ for each  $k = 1,\,2,\,\ldots, K$, and the vectors $h_1,\ldots,h_N$ are assumed to be orthogonal.
Suppose further that $M$ is an adapted local martingale taking values in $H$. Then the {\em stochastic integral} $\phi \cdot M:\mathbb R_+ \times \Omega \to X$ of $\phi$ with respect to $M$ is defined by the formula
\begin{equation*}
 (\phi \cdot M)_t = \sum_{k=1}^K\sum_{m=1}^M \mathbf 1_{B_{mk}}
\sum_{n=1}^N \langle(M(t_k\wedge t)- M(t_{k-1}\wedge t)), h_n\rangle x_{kmn},\;\; t\geq 0.
\end{equation*}

In what follows, we will also need a version of It\^o formula, which is a variation of \cite[Theorem~26.7]{Kal} that does not use the Euclidean structure of a finite-dimensional Banach space. The proof can be found in \cite[Theorem 5.5]{Y17FourUMD}.

\begin{lemma}[It\^o formula]\label{lem:itoformula}
 Let $d\geq 1$, $X$ be a $d$-dimensional Banach space, $f\in C^2(X)$, $M:\mathbb R_+ \times \Omega \to X$ be a martingale. Let $(x_n)_{n=1}^d$ be a basis of $X$, $(x_n^*)_{n=1}^d$ be the corresponding dual basis. Then for each $t\geq 0$
 \begin{equation}\label{eq:itoformula}
  \begin{split}
   f(M_t) = f(M_0)&+ \int_0^t \langle \partial_xf(M_{s-}), \ud M_s\rangle\\
&+ \frac 12 \int_0^t \sum_{n,m=1}^d \tfrac{\partial^2 f(M_{s-})}{\partial x_n \partial x_m}\ud[\langle M, x_n^*\rangle,\langle M, x_m^*\rangle]_s^c\\
&+ \sum_{s\leq t}(\Delta f(M_s) - \langle  \partial_xf(M_{s-}), \Delta M_{s}\rangle).
  \end{split}
 \end{equation}
\end{lemma}

{
Here $\partial_x f(y)\in X^*$ is the Fr\'echet derivative of $f$ in point $y\in X$. Recall that $(x_n^*)_{n=1}^d\subset X^*$ is called the {\em corresponding dual basis} of $(x_n)_{n=1}^d$ if $\langle x_n, x_m^*\rangle = \delta_{nm}$ for each $m,n=1,\ldots, d$.
}

{

\subsection{Quadratic variation}
Let $(\Omega, \mathcal F, \mathbb P)$ be a probability space with a filtration $\mathbb F = (\mathcal F_t)_{t\geq 0}$ that
satisfies the usual conditions. Let $M:\mathbb R_+ \times \Omega \to \mathbb R$ be a local martingale. We define a {\em quadratic variation} of $M$ in the following way:
\begin{equation}\label{eq:defquadvar}
 [M]_t  := |M_0|^2 + \mathbb P-\lim_{{\rm mesh}\to 0}\sum_{n=1}^N \|M(t_n)-M(t_{n-1})\|^2,
\end{equation}
where the limit in probability is taken over partitions $0= t_0 < \ldots < t_N = t$. Note that $[M]$ exists and is nondecreasing a.s. The reader can find more on quadratic variations in \cite{Kal,Prot,DM82}.
For any martingales $M, N:\mathbb R_+ \times \Omega \to \mathbb R$ we can define a {\em covariation} $[M,N]:\mathbb R_+ \times \Omega \to \mathbb R$ as $[M,N] := \frac{1}{4}([M+N]-[M-N])$.
Since $M$ and $N$ have c\`adl\`ag versions, $[M,N]$ has a c\`adl\`ag version as well (see e.g. \cite[Theorem I.4.47]{JS}).

}

\subsection{Weak differential subordination and orthogonal martingales}\label{sec:orthodiff} We have defined the notions of differential subordination and orthogonality of real-valued local martingales in the introductory section. We turn our attention to their vector analogues.

\begin{definition}
 Let $M, N$ be local martingales taking values in a given Banach space $X$. Then $N$ is said to be {\em weakly differentially subordinate} to $M$ (which will be denoted by $N \stackrel{w}\ll M$) if $\langle N, x^*\rangle\ll\langle M, x^*\rangle$ for any functional $x^* \in X^*$.
\end{definition}

It is known (see \cite{Y17MartDec}) that if $N$ is weakly differentially subordinate to $M$, then
\begin{equation}\label{eq:WDSbeta^2(beta+1)est}
 (\mathbb E \|N_t\|^p)^{\frac{1}{p}} \leq \beta_{p,X}^2(\beta_{p,X}+1)(\mathbb E \|M_t\|^p)^{\frac{1}{p}},\;\;\; t\geq 0.
\end{equation}
This estimate can be improved under some additional assumptions on $M$ and $N$ (see \cite{Y17FourUMD,Y17MartDec}). Here we will show such an improvement for $M$ and $N$ being orthogonal (see Section \ref{sec:results}). Moreover, using this improvement we will strengthen \eqref{eq:WDSbeta^2(beta+1)est} (see Remark \ref{rem:wdsbeta(beta+1)est}).

\begin{definition}\label{def:orthmart}
Let $M, N$ be local martingales taking values in a given Banach space $X$. Then $M,\,N$ are said to be {\em orthogonal}, if 
$[\langle M, x^*\rangle,\langle N, x^*\rangle] = 0$ almost surely for all functionals $x^* \in X^*$.
\end{definition}

\begin{remark}\label{rem:orth+wds}
Assume that $M, N$ are local martingales taking values in some Banach space $X$. If $M,\,N$ are orthogonal and $N$ is weakly differentially subordinate to $M$, then $N_0=0$ almost surely (which follows immediately from the above definitions).  Moreover, under these assumptions, $N$ must have continuous trajectories with probability $1$. Indeed, in such a case for any fixed $x^*\in X^*$ the real-valued local martingales $\langle M, x^*\rangle$, $\langle N, x^*\rangle$ are orthogonal and we have $\langle N, x^*\rangle \ll \langle M, x^*\rangle$. Therefore, $\langle N, x^*\rangle$ has a continuous version for each $x^*\in X^*$ by \cite[Lemma 3.1]{Os09a} (see also \cite[Lemma 1]{BanWang96}), which in turn implies that $N$ is continuous: any $X$-valued local martingale has a c\`adl\`ag version (see \cite{Y17FourUMD} and \cite[Proposition 2.2.2]{VerPhD}).
\end{remark}

\begin{remark}
The requirement $\langle M_0, x^*\rangle\cdot \langle N_0, x^*\rangle  =0$ for all $x^*\in X^*$ in Definition \ref{def:orthmart} is usually omitted (see e.g.\ \cite{JS,BanWang96,BanWang95}). Nevertheless we need this requirement in order to simplify all the statements in the sequel concerning orthogonal martingales.
\end{remark}

Weakly differentially subordinate orthogonal martingales appear naturally while working with the periodic Hilbert transform, which can be seen by exploiting the classical argument of Doob (the composition of a harmonic function with a Brownian motion is a martingale). Indeed, suppose that $X$ is a given Banach space. Suppose that $f$ is a simple function and put $g=\mathcal{H}^\mathbb{T}_Xf$. Let $u_f$, $u_g$ denote the harmonic extensions of $f$ and $g$ to the unit disc, obtained by the convolution with the Poisson kernel. In particular, the equality  $g=\mathcal{H}^\mathbb{T}f$ implies that $u_g(0,0)=0$ and for any functional $x^*\in X^*$, the function $\langle u_f,x^*\rangle+i\langle u_g,x^*\rangle$ is holomorphic on the disc.

Next, suppose that $W=(W^1,W^2)$ is a planar Brownian motion started from $(0,0)$ and stopped upon leaving the unit disc. Then the processes $M=(M_t)_{t\geq 0}=(u_f(W_t))_{t\geq 0}$, $N=(N_t)_{t\geq 0}=(u_g(W_t))_{t\geq 0}$ are $X$-valued martingales such that $N_0=0$. For any functional $x^*\in X^*$, we apply the standard, one-dimensional It\^o's formula to obtain, for any $t\geq 0$,
$$ \langle M_t,x^*\rangle =\langle M_0,x^*\rangle +\int_0^t \nabla \langle u_f(W_s),x^*\rangle \mbox{d}W_s$$
and
$$ \langle N_t,x^*\rangle =\langle N_0,x^*\rangle +\int_0^t \nabla \langle u_g(W_s),x^*\rangle \mbox{d}W_s.$$
By the aforementioned connection to analytic functions, the gradients $\nabla \langle u_f,x^*\rangle$, $\nabla \langle u_g,x^*\rangle$ are orthogonal and of equal length, so
\begin{align*}
 [\langle M, x^* \rangle, \langle N, x^*\rangle]_t =\int_0^t \nabla \langle u_f(W_s),x^*\rangle\cdot \nabla \langle u_g(W_s),x^*\rangle  \ud s = 0,
\end{align*}
and
\begin{align*}
  [\langle M, x^* \rangle]_t-[\langle N,x^* \rangle]_t=|\langle M_0,x^*\rangle|^2 + \int_0^t \nabla \langle u_f(W_s),x^*\rangle^2-\nabla \langle u_g(W_s),x^*\rangle^2\ud s=|\langle M_0,x^*\rangle|^2\geq 0.
\end{align*}
Hence $M$, $N$ are orthogonal and satisfy the weak differential subordination $N\stackrel{w}\ll M$. Since the distribution of $W_\infty$ is uniform on the unit circle $\mathbb{T}$, essentially any estimate of the form
$$ \E V(M_t,N_t)\leq 0,\qquad t\geq 0,$$
 for weakly differentially subordinate orthogonal martingales leads to the analogous bound
$$ \int_\mathbb{T} V(f,\mathcal{H}^\mathbb{T}_Xf)\mbox{d}x\leq 0$$
for the periodic Hilbert transform, at least when restricted to the class of simple functions.  (Later in Theorem \ref{thm:orthmartPhiPsi} we will show that the reverse holds true).

For more information and examples concerning the differential subordination, weak differential subordination, and orthogonal martingales, we refer the reader to \cite{JS,HNVW1,Y17MartDec,Prot,BanWang96,Wang,Burk84,BanWang95}.

\subsection{Subharmonic and plurisubharmonic functions}

{An upper semiconitnuous} function $f:\mathbb R^d \to \mathbb R \cup \{-\infty\}$ is called {\em subharmonic} if for any ball $B\subset \mathbb R^d$ and any harmonic function $g:B\to \mathbb R$ such that $f\leq g$ on $\partial B$ one has the inequality $f\leq g$ on the whole $B$. The following lemma follows from \cite[Proposition I.9]{LG86}.

\begin{lemma}\label{lem:either-inftyorL^1loc}
 Let $d\geq 1$ and let $f:\mathbb R^d \to \mathbb R \cup \{-\infty\}$ be a subharmonic function. Then either $f \equiv -\infty$, or $f$ is locally integrable.
\end{lemma}

 Let $X$ be a Banach space. {An upper semiconitnuous} function $F:X + iX \to \mathbb R\cup \{-\infty\}$ is called {\em plurisubharmonic} if for any $x ,y\in X + iX$ the restriction $z\mapsto F(x + yz)$ is subharmonic in $z\in \mathbb C$.

 {
\begin{remark}\label{rem:complexifofX}
 Notice that $X+iX$ is a Banach space equipped with the norm
 \[
  \|x+iy\|_{X+iX}:= \sup_{x^* \in X^*, \|x^*\|\leq 1} (|\langle x, x^*\rangle|^2 + |\langle y, x^*\rangle|^2)^{\frac 12},\;\;\; x, y\in X
 \]
 (see \cite[Subsection B.4]{HNVW1}).
\end{remark}
}

\begin{remark}\label{rem:PSHeither-inftyorL^1loc}
 Let $X$ be finite-dimensional. Then any plurisubharmonic function defined on $X+iX$ is subharmonic (see \cite[Proposition I.9]{LG86} and \cite[Theorem 39]{G14}). Therefore, by Lemma \ref{lem:either-inftyorL^1loc}, a plurisubharmonic function either identically equals $-\infty$, or is locally integrable.
\end{remark}

{
Let $F:X+iX \to \mathbb R$ be $k$-times differentiable, $u_1,\ldots, u_k\in X+iX$. Then we denote
\[
 \frac{\partial^k F(v)}{\partial u_1\cdots \partial u_k} := \frac{\partial^k}{\partial t_1 \cdots \partial t_k} F(v + t_1u_1 + \cdots + t_ku_k)\Big|_{t_1,\ldots, t_k=0},\;\;\; v\in X+iX.
\]
In particular, for any $u\in X+iX$,
\[
 \frac{\partial^k F(v)}{\partial u^k} := \frac{\partial^k}{\partial t^k} F(v + tu)\Big|_{t=0},\;\;\; v\in X+iX.
\]
}

\begin{remark}\label{rem:xzplurisubhstaff}
 Note that if $X$ is finite-dimensional, $F$ is plurisubharmonic and twice differentiable, then for all $z_0\in X+iX$ and $x\in X$ we have
 \begin{align*}
  \frac{\partial^2 F(z_0)}{\partial x^2} +   \frac{\partial^2 F(z_0)}{\partial ix^2} = \Bigl(\frac{\partial^2 F(z_0 + zx)}{\partial \Re z^2} +   \frac{\partial^2 F(z_0 + zx)}{\partial \Im z^2}\Bigr)\Big|_{z=0}= \Delta_zF(z_0 + zx)|_{z=0} \geq 0.
 \end{align*}
\end{remark}

Later on we will need the following result.

\begin{proposition}\label{prop:PSH+concaveis2ndvar->convexin1st+cont}
 Let $X$ be a Banach space and let $F:X + iX \to \mathbb R\cup \{-\infty\}$ be plurisubharmonic. Assume further that $y\mapsto F(x+iy)$ is concave in $y\in X$ for any fixed $x\in X$. Then $x\mapsto F(x+iy)$ is convex in $x\in X$ for any $y\in X$, and $F$ is continuous.
\end{proposition}

For the proof we will need the following lemma.

\begin{lemma}\label{lem:plurisubharmonicconcaveinseconds-->convexinfirst}
 Let $X$ be a finite-dimensional Banach space and let $V:X+iX \to \mathbb R$ be a continuous twice differentiable plurisubharmonic function. Let $y\mapsto V(x+iy)$ be concave in $y\in X$ for all $x\in X$. Then $t\mapsto V(tx+z)$ is convex in $t\in \mathbb R$ for all $x\in X$ and $z\in X+iX$. In particular, {since} $t\mapsto V(tx+z)$ is differentiable, {we have that}
 \begin{equation}\label{eq:plurisubharmonicconcaveinseconds-->convexinfirst}
  V(tx+z)\geq V(sx+z) + \partial_{x} V(sx+z) (t-s),\;\;\; t,s\in \mathbb R.
 \end{equation}
\end{lemma}

\begin{proof}
The first part follows from the fact that $V$ is plurisubharmonic and twice differentiable. Indeed, we have
\begin{multline*}
  \frac{\partial^2 V(tx+z)}{\partial t^2}
  = \Bigl(\frac{\partial^2 V(tx+z+isx)}{\partial t^2}+  \frac{\partial^2 V(tx+z+isx)}{\partial s^2}\Bigr)\Big|_{s=0} - \frac{\partial^2 V(tx+z+isx)}{\partial s^2}\Big|_{s=0}\geq 0
\end{multline*}
since
$$
\Bigl(\frac{\partial^2 V(tx+z+isx)}{\partial t^2}+  \frac{\partial^2 V(tx+z+isx)}{\partial s^2}\Bigr)\Big|_{s=0} \geq 0
$$
by plurisubharmonicity and $ \frac{\partial^2 V(tx+z+isx)}{\partial s^2}\leq 0$ by concavity of $y\mapsto V(x + z +iy)$. The inequality \eqref{eq:plurisubharmonicconcaveinseconds-->convexinfirst} follows immediately from the convexity of $t\mapsto V(tx+iy)$ and twice differentiability of $V$.
\end{proof}

{To complete the proof of Proposition \ref{prop:PSH+concaveis2ndvar->convexin1st+cont}} we will need the following observation which will allow us to integrate over a Banach space.

\begin{remark}\label{rem:lebmeasuintoverX}
 Let $X$ be a finite dimensional Banach space. Then due to \cite[Theorem 2.20 and Proposition 2.21]{FolHarm} there exists a unique  translation-invariant measure $\lambda_X$ on $X$ such that $\lambda_X(\mathbb B_X) = 1$ for the unit ball $\mathbb B_X$ of $X$. We will call $\lambda_X$ the {\em Lebesgue measure}. In the sequel we will omit the Lebesgue measure notation while integrating over $X$ (i.e.\ we will write $\int_X F(s) \ud s$ instead of $\int_X F(s) \lambda_X(\dd s)$).
\end{remark}

\begin{proof}[Proof of Proposition \ref{prop:PSH+concaveis2ndvar->convexin1st+cont}]
 Without loss of generality we can assume that $X$ is finite-dimensional and that $f\not\equiv -\infty$. {As $X$ is finite dimensional, $X + iX \simeq \mathbb C^d$ for some $d$.} Let $\phi:X+iX\to \mathbb R_+$ be a $C^{\infty}$ function {radial with respect to $\mathbb C^d$ (i.e.\ depending only on $|z_1|^2 + \ldots + |z_d|^2$ with $z_1, \ldots, z_d$ being the basis of $\mathbb C^d$)} with bounded support such that
 \[
  \int_{X+iX}\phi(s)\ud s = 1.
 \]

 (This integral is well-defined due to Remark \ref{rem:complexifofX} and \ref{rem:lebmeasuintoverX}). 
For each $\eps>0$ we define $F_{\eps}:X+iX \to \mathbb R$ in the following way:
 \begin{equation}\label{eq:defofF_eps}
   F_{\eps}(s) = \int_{X+iX} F(s-\eps t)\phi(t)\ud t,\;\;\;\; s\in X+iX.
 \end{equation}
Then $F_{\eps}$ is plurisubharmonic due to \cite[Theorem 4.1.4]{Hor94}. Moreover, again by \cite[Theorem 4.1.4]{Hor94}, we have $F_{\eps}\searrow F$ as $\eps \searrow 0$. On the other hand, $F_{\eps}$ is well-defined and of class $C^{\infty}$. Furthermore, the function $y\mapsto F_{\eps}(x+iy)$ is concave in $y\in X$ for any $x\in X$ by \eqref{eq:defofF_eps}: here we use the fact that $F$ is locally integrable (see Remark \ref{rem:PSHeither-inftyorL^1loc}) and the concavity of $y\mapsto F(x+iy)$ for any fixed $x\in X$. Therefore by Lemma \ref{lem:plurisubharmonicconcaveinseconds-->convexinfirst}, the function $x\mapsto F_{\eps}(x+iy)$ is convex for any fixed $y\in X$; hence so is  $F$,  being the pointwise limit of $(F_{\eps})_{\eps>0}$ as $\eps \to 0$.

Let us now show that $F>-\infty$. Assume that there exists $x_0,y_0\in X$ such that $F(x_0+iy_0)=-\infty$. Since the function $y\mapsto F(x_0+iy)$ is concave, the set $A=\{y\in X:F(x_0+iy)>\infty\}\subset X$ is convex and open; moreover, $y_0\notin A$, so $X\setminus A$ is of positive measure. Now fix $(x, y)\in X\times (X\setminus A)$. Notice that $F(x_0+iy)=-\infty$. On the other hand $x\mapsto F(x+iy)$ is convex, so $F(x+iy) = -\infty$ as well (if a convex function equals $-\infty$ in one point, it equals $-\infty$ on the whole $X$). Therefore $F=-\infty$ in the set $X\times (X\setminus A)$ of positive measure; hence $F\equiv -\infty$ by Remark \ref{rem:PSHeither-inftyorL^1loc}, which leads to a contradiction.

Finally, note that $F<\infty$: we have $F\leq F_1$ with $F_1$ defined in \eqref{eq:defofF_eps}. Therefore $F$ is continuous as a finite concave-convex function (see \cite[Proposition 3.3]{Thib84} and \cite[Corollary 4.5]{JT}).
\end{proof}

For further material on subharmonic and plurisubharmonic functions, we recommend the works \cite{LG86,G14,Hor94,Ron71,Ron74}.

\subsection{Meyer-Yoeurp decomposition}\label{subsec:MYdec}

Let $X$ be a Banach space and let $M$ be a local martingale with values in $X$. Then {is called {\em purely discontinuous} if $[M]$ is a.s.\ a pure jump process (see \cite{Kal,JS} for details).} $M$ is said to have the {\em Meyer-Yoeurp decomposition} if there exist an $X$-valued continuous local martingale $M^c$ and an $X$-valued purely discontinuous local martingale $M^d$ such that $M^c_0=0$ and $M=M^c + M^d$ a.s. It was shown by Meyer in \cite{Mey76} and by Yoeurp in \cite{Yoe76} that any real-valued martingale has the Meyer-Yoeurp decomposition. Later in \cite{Y17GMY} it was shown that any $X$-valued local martingale has the Meyer-Yoeurp decomposition if and only if $X$ has the UMD property. See \cite{Kal,JS,Y17MartDec} for further details.

The following result shows the connection between the Meyer-Yoeurp decomposition and the weak differential subordination.

\begin{proposition}\label{prop:MYdecconttocontpdtopd}
 Let $X$ be a Banach space and let $M,\,N$ be local $X$-valued martingales possessing the Meyer-Yoeurp decompositions  $M=M^c+M^d$, $N=N^c+N^d$. Then $N \stackrel{w}\ll M$ if and only if $N^c \stackrel{w}\ll M^c$ and $N^d \stackrel{w}\ll M^d$. Moreover, if $M$ and $N$ are orthogonal, then $M^c$ and $N^c$, $M^d$ and $N^d$ are pairwise orthogonal.
\end{proposition}

\begin{proof}
 The first part can be found in \cite{Y17MartDec} (see also \cite[Lemma 1]{Wang}). Due to Remark~\ref{rem:orth+wds} we know that $N^d=0$, so it is sufficient to show that $M^c$ and $N^c$ are orthogonal. The latter is equivalent to the fact that $\langle M^c, x^*\rangle$ and $\langle N^c, x^*\rangle$ are orthogonal for any $x^*\in X^*$, which holds true by \cite[Lemma 1]{BanWang96}.
\end{proof}

\section{Main theorem}\label{sec:results}
Having introduced all the necessary notions, we turn to the study of our new results. For given two nonnegative and continuous functions $\Phi, \Psi:X\to \mathbb R_+$, we define the associated `$\Phi,\Psi$-norm' of $\mathcal{H}^\mathbb{T}_X$ by the formula
\begin{align*}
|\mathcal H_{X}^{\mathbb T}|_{\Phi,\Psi}:=\inf\Biggl\{c\in [0,\infty]: \int_{\mathbb T} \Psi(\mathcal H_X^{\mathbb T}f(s))\ud s\leq c\int_{\mathbb T} \Phi(f(s))\ud s\mbox{for all step functions }f:\mathbb T\to X\Biggr\}.
\end{align*}

Notice that if $\Psi \equiv 0$, then $|\mathcal H_{X}^{\mathbb T}|_{\Phi,\Psi} = 0$, and if $\Phi \equiv 0$, then $|\mathcal H_{X}^{\mathbb T}|_{\Phi,\Psi} \in \{0,+\infty\}$. Throughout the paper we exclude these trivial cases: we will assume that both $\Phi$ and $\Psi$ are not identically zero. Furthermore, for any $1<p<\infty$, we will denote the $L^p$-norm of $\mathcal{H}^\mathbb{T}_X$ by $\hbar_{p,X}$ (in the language of $\Phi,\Psi$-norms, we have $\hbar_{p,X}^p=|\mathcal H^{\mathbb T}_X|_{\Phi,\Psi}$ with $\Phi(x)=\Psi(x)=||x||^p$).

The following theorem is the main result of this section.

\begin{theorem}\label{thm:orthmartPhiPsi}
 Let $X$ be a separable Banach space and let $\Phi,\,\Psi:X\to \mathbb R_+$ be continuous convex functions such that $\Psi(0)=0$ and $|\mathcal H^{\mathbb T}_X|_{\Phi, \Psi}<\infty$. Let $M,\,N$ be two orthogonal $X$-valued local martingales such that $N \stackrel{w}\ll M$. Then
 \begin{equation}\label{eq:HTnormboudfororcontmart}
  \mathbb E \Psi(N_t) \leq  C_{\Phi, \Psi, X}\mathbb E \Phi(M_t),\qquad  t\geq 0,
 \end{equation}
 {and the least admissible $C_{\Phi,\Psi,X}$ equals $|\mathcal H^{\mathbb T}_X|_{\Phi, \Psi}$.}
\end{theorem}

The idea behind the proof of \eqref{eq:HTnormboudfororcontmart} can be roughly described as follows. First, we will show that the condition $|\mathcal H^{\mathbb T}_X|_{\Phi, \Psi}<\infty$ (i.e., the validity of a $\Phi,\Psi$-estimate for the periodic Hilbert transform) implies the existence of a certain special function on $X+iX$, enjoying appropriate size conditions and concavity. Next, we will compose this function with $M+iN$ and prove, using the concavity and It\^o's formula from the previous section, that the resulting process has nonnegative expectation. This in turn will give the desired bound, in the light of the size condition of the special function. Though this reasoning is typical for this kind of martingale inequalities, there are two essential differences. First, we will see that the special function will not have any explicit form: in particular, this makes the exploitation of its properties much harder, as one can get them only from some abstract (and restricted) reasoning. The second difference is related to the fact that we work {with} Banach-space-valued processes: this enforces us to study some additional, structural properties of the local martingales involved. {Moreover, since we will work in infinite-dimensional Banach spaces, the approximation to finite dimensions exploited in the proof should be especially delicate because we do not want to ruin weak differential subordination and orthogonality of the corresponding martingales.}

Having described our plan, we turn to its realization. We will need several intermediate facts. The following theorem links the quantity $|\mathcal{H}_X^\mathbb{T}|_{\Phi,\Psi}$ with a certain special plurisubharmonic function.

\begin{theorem}\label{thm:existenceofU^H_Phi,Psi}
 Let $X$ be a separable Banach space and let $\Phi,\,\Psi:X\to \mathbb R_+$ be continuous functions such that $\Psi(0)=0$ and $|\mathcal H^{\mathbb T}_X|_{\Phi, \Psi}<\infty$. Then there exists a plurisubharmonic function $U_{\Phi,\Psi}:X+i X \to \mathbb R$ such that $U_{\Phi,\Psi}(x)\geq 0$ for all $x\in X$ and
 \begin{equation*}
  U_{\Phi,\Psi}(x+iy)\leq  |\mathcal H^{\mathbb T}_X|_{\Phi, \Psi} \Phi(x) - \Psi(y),\;\;\; x,y\in X.
 \end{equation*}
 Moreover, if $\Psi$ is convex, then $y\mapsto U_{\Phi,\Psi}(x+iy)$ is concave in $y\in X$ for all $x\in X$.
\end{theorem}

\begin{proof}[Proof (sketch)]
We repeat the reasoning presented in \cite[Theorem 2.3]{HKV03} (the separability of $X$ is a key part of the construction  $U_{\Phi,\Psi}$). The last property follows from the construction of $U_{\Phi,\Psi}$, the fact that $y\mapsto |\mathcal H^{\mathbb T}_X|_{\Phi, \Psi}\Phi(x) - \Psi(y)$ is a concave function in $y\in X$, and the fact that a minimum of concave functions is a concave function as well.
\end{proof}

\begin{corollary}\label{thm:existenceofU^H_p,X}
 Let $X$ be a Banach space, $1<p<\infty$. Then $X$ is a UMD Banach space if and only if there exists a plurisubharmonic function $U_{p,X}:X+i X \to \mathbb R$ such that $U_{p,X}(x)\geq 0$ for all $x\in X$ and
 \begin{equation*}
  U_{p,X}(x+iy)\leq  \hbar_{p,X}^p \|x\|^p - \|y\|^p,\;\;\; x,y\in X.
 \end{equation*}
 Moreover, if this is the case, then $y\mapsto U_{p,X}(x+iy)$ is concave in $y\in X$ for all $x\in X$.
\end{corollary}

\begin{proof}
It is sufficient to take $\Phi(x) = \Psi(x) = \|x\|^p$, $x\in X$, and apply Theorem~\ref{thm:existenceofU^H_Phi,Psi} and the fact that $\hbar_{p,X}<\infty$ if and only if $X$ is a UMD Banach space (see \cite{Burk81,Bour83}).
\end{proof}

\begin{lemma}\label{lemma:sttimeargforconvexfuncandmart}
 Let $X$ be a Banach space, let $M$ be an $X$-valued local martingale and let $(\tau_n)_{n\geq 1}$ be a sequence of stopping times increasing to infinity almost surely. Let $\Phi:X\to \mathbb R_+$ be a convex function such that $\mathbb E \Phi(M_t)<\infty$ for some $t\geq 0$. Then $\mathbb E \Phi(M_{t\wedge \tau_n})\nearrow \mathbb E \Phi(M_t)$ as $n\to \infty$.
\end{lemma}

\begin{proof}
 Notice that $(\mathbb E\Phi(M_{t\wedge \tau_n}))_{n\geq 1}$ is an increasing sequence which is less then $\mathbb E \Phi(M_t)$ by the conditional Jensen's inequality, \cite[Theorem 7.12]{Kal}, and \cite[Lemma 7.1(iii)]{Kal}. On the other other hand $\Phi(M_{t\wedge \tau_n}) \to \Phi(M_t)$ a.s.\ since $\tau_n \to \infty$ as $n\to \infty$. It suffices to apply Fatou's lemma to get the assertion.
\end{proof}

The next statement contains the proof of a structural property of orthogonal martingales. We need an additional notion. A linear operator $T$ acting on a Hilbert space $H$ is called {\em skew-symmetric} (or {\em antisymmetric}) if $\langle Th,h\rangle =0$ for all $h\in H$.

\begin{proposition}\label{prop:existofaskewsymmopvalpr}
 Let $d\geq 1$, $W$ be a $d$-dimensional standard Brownian motion, let $X$~be a finite-dimensional Banach space and let $\phi, \psi:\mathbb R_+ \times \Omega \to \mathcal L(\mathbb R^d,X)$ be progressively measurable processes such that $M:= \phi\cdot W$ and $N:=\psi\cdot W$ are well-defined orthogonal martingales. Assume further that $N \stackrel{w}\ll M$. Then there exists a operator-valued progressively-measurable process $A:\mathbb R_+ \times \Omega \to \mathcal L(\mathbb R^d)$ such that $\|A\|\leq 1$, $\psi^* =A \phi^*$ a.s.\ on $\mathbb R_+ \times \Omega$, and $P_{\textnormal{Ran}(\phi^*)}(s,\omega)A(s,\omega)$ is skew-symmetric for all $s\geq 0$ and  $\omega\in \Omega$, where $P_{\textnormal{Ran}(\phi^*)} \in \mathcal L(\mathbb R^d)$ is the orthoprojection on $\textnormal{Ran}(\phi^*)$.
\end{proposition}

\begin{proof}
Let $(x_n^*)_{n\geq 1}$ be a dense sequence in $X^*$. Then by the orthogonality of $M$, $N$ and the condition $N\stackrel{w}\ll M$, we have
\begin{equation*}
 \|\psi^*(t,\omega)x_n^*\| \leq \|\phi^*(t,\omega)x_n^*\|,
\end{equation*}
\begin{equation*}
 \langle \psi^*(t,\omega)x_n^*,\phi^*(t,\omega)x_n^*\rangle=0
\end{equation*}
for almost all $\omega \in \Omega$, all $t\in \mathbb R_+$ and all $n\geq 1$.
Hence by {a} density argument, for any $x^*\in X^*$, almost all  $\omega \in \Omega$ and all $t\in \mathbb R_+$,
\begin{equation}\label{eq:ineqweakonphiandpsi}
 \|\psi^*(t,\omega)x^*\| \leq \|\phi^*(t,\omega)x^*\|,
\end{equation}
\begin{equation}\label{eq:weakorthofphiandpsi}
 \langle \psi^*(t,\omega)x^*,\phi^*(t,\omega)x^*\rangle=0.
\end{equation}
Fix $t\in \mathbb R_+$ and $\omega \in \Omega$ such that \eqref{eq:ineqweakonphiandpsi} and \eqref{eq:weakorthofphiandpsi} hold for any $x^*\in X^*$.
Define $A(t,\omega):H\to H$ in the following way (we omit $(t,\omega)$ for the convenience of the reader):
\begin{equation}\label{eq:defofAskewsymm}
  Ah:=
 \begin{cases}
\psi^*x^*, &\text{if}\;\; \exists x^*\in X^*\;\text{such that}\; h=\phi^*x^*;\\
0,&\text{if}\;\; h\perp\text{Ran}(\phi^*) .
\end{cases}
\end{equation}
Then $A$ is well-defined since if $h=\phi^*(t,\omega)x_1^*=\phi^*(t,\omega)x_2^*$ for some different $x_1^*,x_2^*\in X^*$, then by \eqref{eq:ineqweakonphiandpsi},
\begin{align*}
 \|\psi^*(t,\omega)x_1^* - \psi^*(t,\omega)x_2^*\| = \|\psi^*(t,\omega)(x_1^* -x_2^*)\|&\leq \|\phi^*(t,\omega)(x_1^* -x_2^*)\|\\
 &= \|\phi^*(t,\omega)x_1^* - \phi^*(t,\omega)x_2^*\| = \|h-h\|=0.
\end{align*}
Moreover, $A$ is linear on both $\text{Ran}(\phi^*)$ and $(\text{Ran}(\phi^*))^{\perp}$, so it can be extended to a linear operator $A \in \mathcal L(H)$. Notice that then we have $\psi^* = A\phi^*$. Furthermore, the conditions \eqref{eq:ineqweakonphiandpsi} and \eqref{eq:defofAskewsymm} imply that $\|A\|\leq 1$, while \eqref{eq:weakorthofphiandpsi} and \eqref{eq:defofAskewsymm} give that $P_{\textnormal{Ran}(\phi^*)}A$ is skew-symmetric  ($P_{\textnormal{Ran}(\phi^*)}$ being the orthoprojection on $\textnormal{Ran}(\phi^*)$).
\end{proof}

In our later considerations, we will also need the following technical result.

\begin{proposition}\label{prop:nicesecderofU_PhiPsi}
 Let $X$ be a finite-dimensional Banach space and let $\Phi,\Psi:X\to \mathbb R_+$ be continuous functions such that $\Psi$ is convex, $\Psi(0)=0$ and $|\mathcal H^{\mathbb T}_X|_{\Phi, \Psi}<\infty$. Let $U_{\Phi,\Psi}:X+iX\to \mathbb R$ be the special function from Theorem \ref{thm:existenceofU^H_Phi,Psi}. Assume additionally that $U_{\Phi,\Psi}$ is twice differentiable. Then for any $x, y\in X$, $z_0\in X+iX$ and any $\lambda\in [-1,1]$ we have
 \begin{equation}\label{eq:prop:nicesecderofU_PhiPsi}
  \begin{split}
      \frac{\partial^2 U_{\Phi, \Psi}(z_0)}{\partial x^2} + \frac{\partial^2 U_{\Phi, \Psi}(z_0)}{\partial y^2}&+2 \lambda \Bigl(\frac{\partial^2 U_{\Phi, \Psi}(z_0)}{\partial x\partial iy} - \frac{\partial^2 U_{\Phi, \Psi}(z_0)}{\partial y\partial ix}\Bigr)  + \lambda^2\Bigl(\frac{\partial^2 U_{\Phi, \Psi}(z_0)}{\partial ix^2}+ \frac{\partial^2 U_{\Phi, \Psi}(z_0)}{\partial iy^2}\Bigr)\geq 0.
  \end{split}
 \end{equation}
\end{proposition}

\begin{proof}
 Notice that the function
  \begin{align*}
   \lambda \mapsto \frac{\partial^2 U_{\Phi, \Psi}(z_0)}{\partial x^2} + \frac{\partial^2 U_{\Phi, \Psi}(z_0)}{\partial y^2} +2 \lambda \Bigl(\frac{\partial^2 U_{\Phi, \Psi}(z_0)}{\partial x\partial iy} - \frac{\partial^2 U_{\Phi, \Psi}(z_0)}{\partial y\partial ix}\Bigr) + \lambda^2\Bigl(\frac{\partial^2 U_{\Phi, \Psi}(z_0)}{\partial ix^2} + \frac{\partial^2 U_{\Phi, \Psi}(z_0)}{\partial iy^2}\Bigr)
 \end{align*}
 is concave due to the fact that $\frac{\partial^2 U_{\Phi, \Psi}(z_0)}{\partial ix^2}$, $\frac{\partial^2 U_{\Phi, \Psi}(z_0)}{\partial iy^2}\leq 0$ by the last part of Theorem~\ref{thm:existenceofU^H_Phi,Psi}. Therefore it is sufficient to show \eqref{eq:prop:nicesecderofU_PhiPsi} for $\lambda = 1$ and $\lambda = -1$. We will consider the first possibility only,  the second can be handled analogously. We have
 \begin{align*}
  \frac{\partial^2 U_{\Phi, \Psi}(z_0)}{\partial x^2} + \frac{\partial^2 U_{\Phi, \Psi}(z_0)}{\partial y^2}&+2 \Bigl(\frac{\partial^2 U_{\Phi, \Psi}(z_0)}{\partial x\partial iy} - \frac{\partial^2 U_{\Phi, \Psi}(z_0)}{\partial y\partial ix}\Bigr)  + \Bigl(\frac{\partial^2 U_{\Phi, \Psi}(z_0)}{\partial ix^2} + \frac{\partial^2 U_{\Phi, \Psi}(z_0)}{\partial iy^2}\Bigr)\\
 &= \frac{\partial^2 U_{\Phi,\Psi}(z_0+t(x+iy))}{\partial t^2} +\frac{\partial^2 U_{\Phi,\Psi}(z_0+t(y-ix))}{\partial t^2}\\
 &=\Delta_z U_{\Phi,\Psi}(z_0+z(y-ix))\geq 0,
 \end{align*}
since $U_{\Phi,\Psi}$ is plurisubharmonic (here $\Delta_z$ is the Laplace operator acting with respect to the $z$-variable).
\end{proof}

\begin{corollary}\label{cor:propnicesecderofU_PhiPsi}
Under the assumptions of the previous Proposition,  for any $x, y\in X$, $z_0\in X+iX$, $\lambda\in [-1,1]$ and any $\mu\in [-|\lambda|, |\lambda|]$ we have
 \begin{equation}\label{eq:corprop:nicesecderofU_PhiPsi}
  \begin{split}
      \frac{\partial^2 U_{\Phi, \Psi}(z_0)}{\partial x^2} + \frac{\partial^2 U_{\Phi, \Psi}(z_0)}{\partial y^2}&+2 \mu \Bigl(\frac{\partial^2 U_{\Phi, \Psi}(z_0)}{\partial x\partial iy} - \frac{\partial^2 U_{\Phi, \Psi}(z_0)}{\partial y\partial ix}\Bigr) + \lambda^2\Bigl(\frac{\partial^2 U_{\Phi, \Psi}(z_0)}{\partial ix^2} + \frac{\partial^2 U_{\Phi, \Psi}(z_0)}{\partial iy^2}\Bigr)\geq 0.
  \end{split}
 \end{equation}
\end{corollary}

\begin{proof}
The left-hand side of \eqref{eq:corprop:nicesecderofU_PhiPsi} is linear in $\mu$, so it is sufficient to check the estimate for $\mu = \pm \lambda$.
\end{proof}

The following lemma {was proven in the supplement to \cite[Lemma 3.7]{Y17MartDec}}.

\begin{lemma}\label{lemma:traceindep}
 Let $d$ be a natural number, $E$ be a $d$-dimensional linear space. Let $V: E\times E \to \mathbb R$ and $W:E^* \times E^* \to \mathbb R$ be two bilinear functions. Then the expression
 \begin{equation*}
  \sum_{n, m=1}^d V(e_n, e_m) W(e_n^*, e_m^*)
 \end{equation*}
does not depend on the choice of basis $(e_n)_{n=1}^d$ of $E$ (here $(e_n^*)_{n=1}^d$ is the corresponding dual basis of $(e_n)_{n=1}^d$).
\end{lemma}

\begin{corollary}\label{cor:lemmatraceindep}
 Let $d$ be a natural number, $E$ be a $d$-dimensional linear space. Let $V: E\times E \to \mathbb R$ and $W_1, W_2:E^* \times E^* \to \mathbb R$ be bilinear functions. Assume additionally that $V$ is symmetric nonnegative (i.e.\ $V(x,x)\geq 0$ for all $x\in E$) and that $W_1(x^*, x^*) \leq W_2(x^*, x^*)$ for all $x^*\in X^*$. Then
 \begin{equation*}
  \sum_{n, m=1}^d V(e_n, e_m) W_1(e_n^*, e_m^*) \leq  \sum_{n, m=1}^d V(e_n, e_m) W_2(e_n^*, e_m^*)
 \end{equation*}
for any basis $(e_n)_{n=1}^d$ of $E$ (here $(e_n^*)_{n=1}^d$ is the corresponding dual basis of $(e_n)_{n=1}^d$).
\end{corollary}

\begin{proof}
 Since $V$ is symmetric and nonnegative it defines an inner product on $E\times E$. Let $(\tilde e_n)_{n=1}^d$ be an orthogonal basis of $E$ under the inner product $V$ (i.e.\ $V(\tilde e_n,\tilde e_m)=0$ for all $n\neq m$, and $V(\tilde e_n,\tilde e_n)\geq 0$ for all $n=1,\ldots,d$). Then we have that
 \begin{equation}\label{eq:VW_1VW_2}
  \begin{split}
        \sum_{n, m=1}^d V(\tilde e_n, \tilde e_m) &W_1(\tilde e_n^*, \tilde e_m^*)= \sum_{n=1}^d V(\tilde e_n, \tilde e_n) W_1(\tilde e_n^*, \tilde e_n^*)\\
     &\leq \sum_{n=1}^d V(\tilde e_n, \tilde e_n) W_2(\tilde e_n^*, \tilde e_n^*) = \sum_{n, m=1}^d V(\tilde e_n, \tilde e_m) W_2(\tilde e_n^*, \tilde e_m^*),
  \end{split}
 \end{equation}
where $(\tilde e_n^*)_{n=1}^d$ is the corresponding dual basis of $(\tilde e_n)_{n=1}^d$. Consequently, the desired follows from \eqref{eq:VW_1VW_2} and Lemma \ref{lemma:traceindep}.
\end{proof}

The next few statements aim at establishing an appropriate ``localization'' procedure: we will prove how to deduce the general, possibly infinite-dimensional context from its finite-dimensional counterpart. We need some additional notation.
Let $X$ be a Banach space with a dual $X^*$, $Y\subset X^*$ be a linear subspace. Let $P:Y\hookrightarrow X^*$ be the continuous embedding operator. Then $P^*$ is a well-defined bounded linear operator from $X^{**}$ to $X_Y:=Y^*$ such that $\text{Ran}(P^*) = X_Y$. Moreover, if $Y$ is finite-dimensional, then $\text{Ran}(P^*|_{X}) = X_Y$, where $P^*|_{X}:X\to X_{Y}$ is a well-defined restriction of $P^*$ on $X$ due to the natural embedding $X\hookrightarrow X^{**}$. For any function $\phi:X\to \mathbb R_+$, we can define $\phi_Y:X_Y\to\mathbb R_+$ by the formula
\begin{equation}\label{eq:defofphi_Y}
 \phi_Y(\tilde x) = \inf\{\phi(x)\,:\,x\in X, P^*x=\tilde x\},\;\;\; \tilde x \in X_Y.
\end{equation}
{In other words, $\phi_Y(\tilde x)$ denotes the infimum of $\phi(x)$ over all $x\in X$ satisfying $x|_{Y} = \tilde x$ where we consider $x$ as an element of $Y^*$.}

\begin{lemma}\label{lem:Phi_YPsi_Yareconvex}
 Let $X$ be a Banach space with a dual $X^*$ and let $Y\subset X^*$ be a {closed} linear subspace. Let $\phi:X\to \mathbb R_+$ be a convex function. Then $\phi_Y:X_Y\to \mathbb R_+$ defined by \eqref{eq:defofphi_Y}
is convex and we have $\phi_Y(P^* x)\leq \phi(x)$ for all $x\in X$.
\end{lemma}

\begin{proof}
 Fix $\tilde x_1,\tilde x_2\in X_Y$, $\lambda \in[0,1]$ and set $\tilde x = \lambda\tilde  x_1 + (1-\lambda)\tilde  x_2$. Then
 \begin{align*}
  \phi_Y(\tilde x) &= \inf_{\substack{x\in X\\P^*x=\tilde x}}\phi(x) = \inf_{\substack{x_1\in X, P^*x_1=\tilde x_1\\x_2\in X, P^*x_2=\tilde x_2}}\phi(\lambda x_1 + (1-\lambda) x_2)\leq\inf_{\substack{x_1\in X, P^*x_1=\tilde x_1\\x_2\in X, P^*x_2=\tilde x_2}}\lambda\phi( x_1) + (1-\lambda)\phi(x_2)\\
   &=\lambda\inf_{x_1\in X, P^*x_1=\tilde x_1}\phi(x_1) + (1-\lambda)\inf_{x_2\in X, P^*x_2=\tilde x_2}\phi( x_2)=\lambda \phi_{Y}(\tilde x_1) + (1-\lambda)\phi_Y(\tilde x_2),
 \end{align*}
 so $\phi_Y$ is convex. The last part of the lemma follows from the definition of $\phi_Y$.
\end{proof}

\begin{lemma}\label{lem:Phi_nPsi_nconvergestoPhiPsi}
 Let $X$ be a separable Banach space, $\phi:X\to\mathbb R_+$ be convex lower semi-continuous. Then there exists an increasing sequence of finite-dimensional subspaces $(Y_n)_{n\geq 1}$ of $X^*$ such that the following holds. If $P_n:Y_n\hookrightarrow X^*$ is the corresponding embedding for each $n\geq 1$ and $\phi_n:Y_n^*\to \mathbb R_+$ satisfies
 \begin{equation}\label{eq:defofphi_n}
 \phi_n(\tilde x) = \inf\{\phi(x)\,:\,x\in X, P_n^*x=\tilde x\},\;\;\; \tilde x \in Y_n^*,
\end{equation}
then for each $x\in X$ the sequence $(\phi_n(P_n^*x))_{n\geq 1}$ increases to $\phi(x)$ as $n\to \infty$.
\end{lemma}

\begin{proof}
By \cite[Lemma 1.2.10]{HNVW1} there exist a sequence $(x_n^*)_{n\geq 1}$ in $X^*$ and a sequence $(a_n)_{n\geq 1}$ of real numbers such that
\begin{equation}\label{eq:convassupoflin}
  \phi(x) = \sup_{n}\langle x,x_n^* \rangle + a_n,\;\;\; x\in X.
\end{equation}
Let $Y_n := \text{span}(x_1^*,\ldots,x_n^*)$ for each $n\geq 1$. Fix $x\in X$. First notice that $\phi_n(P_n^*x)\leq \phi(x)$ by Lemma \ref{lem:Phi_YPsi_Yareconvex}. Moreover, $\phi_n(P_n^*x)\leq \phi_{n+1}(P_{n+1}^*x)$ for each $n\geq 1$ since $Y_n\subset Y_{n+1}$ (see \eqref{eq:defofphi_n}). Fix $n\geq 1$. Then for any $y\in X$ such that $P_n^*x=P_n^*y$ we have  $\langle x,x_k^* \rangle = \langle y,x_k^* \rangle$ for any $k=1,\ldots, n$, so by \eqref{eq:convassupoflin},
\begin{align*}
 \phi_n(P_n^* x) &= \inf\{\phi(y):y\in X, P_n^*y=P_n^* x\}\geq \inf\{\sup_{1\leq k\leq n}\langle y,x_k^* \rangle + a_k:y\in X, P_n^*y=P_n^* x\}\\
 &= \inf\{\sup_{1\leq k\leq n}\langle x,x_k^* \rangle + a_k:y\in X, P_n^*y=P_n^* x\}=\sup_{1\leq k\leq n}\langle x,x_k^* \rangle + a_k.
\end{align*}
Since the latter expression tends to $\phi(x)$ as $n\to\infty$, we obtain the desired monotone convergence $\phi_n(P_n^*x)\nearrow \phi(x)$.
\end{proof}

\begin{proposition}\label{prop:normPhi_YPsi_YlessnormPhiPsi}
 Let $X$ be a Banach space with a dual $X^*$ and let $Y\subset X^*$ be a finite-dimensional linear subspace. Assume further that $\Phi, \Psi:X\to\mathbb R_+$ are convex continuous functions and let $\Phi_Y, \Psi_Y:X_Y\to \mathbb R_+$ be defined by \eqref{eq:defofphi_Y}. Then
 \[
  |\mathcal H^{\mathbb T}_{X_Y}|_{\Phi_Y, \Psi_Y} \leq |\mathcal H^{\mathbb T}_X|_{\Phi,\Psi}.
 \]
\end{proposition}

\begin{proof}
Recall that
\[
 |\mathcal H^{\mathbb T}_{X_Y}|_{\Phi_Y, \Psi_Y} = \sup_{f\in F^{\rm step}_{X_Y}}\frac{\int_{\mathbb T}\Psi_Y(\mathcal H^{\mathbb T}_{X_Y} f)\ud \mu}{\int_{\mathbb T}\Phi_Y(f)\ud \mu},
\]
where $\mu$ is the Lebesgue measure on $\mathbb T$. Fix $f\in F^{\rm step}_{X_Y}$ and $\eps>0$. Let $(\tilde x_n)_{n=1}^N\subset X_Y$ be the range of $f$. For each $n=1,\ldots,N$ we define $x_n\in X$ to be such that $P^* x_n = \tilde x_n$ and $\Phi(x_n) \leq (1+\eps)\Phi_Y(\tilde x_n)$ (existence of such $x_n$ follows from the fact that $\text{Ran}(P^*) = X_Y$); we define $g:\mathbb T \to X$ to be such that $f(s) = \tilde x_n$ if and only if $g(s) = x_n$, $s\in \mathbb T$. Then $\Phi_{Y}(f) = \Phi_Y(P^* g)$ and $\Psi_Y(\mathcal H^{\mathbb T}_{X_Y} f) = \Psi_Y(\mathcal H^{\mathbb T}_{X_Y} P^*g)=\Psi_Y(P^*\mathcal H^{\mathbb T}_X g)$ for any $s\in \mathbb T$ by the definition of the Hilbert transform on the torus. Therefore
\begin{align*}
 \frac{\int_{\mathbb T}\Psi_Y(\mathcal H^{\mathbb T}_{X_Y} f)\ud \mu}{\int_{\mathbb T}\Phi_Y(f)\ud \mu}
 = \frac{\int_{\mathbb T}\Psi_Y(P^*\mathcal H^{\mathbb T}_X g)\ud \mu}{\int_{\mathbb T}\Phi_Y(P^*g)\ud \mu}
 \stackrel{(*)}\leq (1+\eps)\frac{\int_{\mathbb T}\Psi(\mathcal H^{\mathbb T}_X g)\ud \mu}{\int_{\mathbb T}\Phi(g)\ud \mu} \stackrel{(**)}\leq (1+\eps)|\mathcal H^{\mathbb T}_X|_{\Phi,\Psi},
\end{align*}
where $(*)$ follows from the fact that $\Phi(g(s))\leq (1+\eps)\Phi_Y(f(s))$ for any $s\in \mathbb T$ and from the fact that $\Psi_Y(P^*\cdot) \leq \Psi(\cdot)$ on $X$, while $(**)$ follows from the definition of $|\mathcal H^{\mathbb T}_X|_{\Phi,\Psi}$. Since $f\in F^{\rm step}_{X_Y}$ and $\eps>0$ were arbitrary, the claim follows.
\end{proof}

The final ingredient is the following well-known statement from the theory of stochastic integration.

\begin{lemma}\label{lem:stochintfor*bddmart}
 Let $d\geq 1$ and let $M$ be a martingale with values in $\R^d$ satisfying the condition $\mathbb E M^*_{\infty}<\infty$. Let $V:\mathbb R_+ \times \Omega \to \mathbb R^d$ be a predictable and bounded process. Then $V\cdot M := \int \langle V,\ud M\rangle$ is a well-defined martingale and $\mathbb E (V\cdot M)^*_{\infty}<\infty$.
\end{lemma}

Equipped with the above statements, we are ready for the study of our main result. We should point out that the main difficulty lies in proving the inequality \eqref{eq:HTnormboudfororcontmart} for finite-dimensional Banach spaces. The novelty in comparison to other results from the literature is that we work under slightly different condition of weak differential subordination and orthogonality; therefore, though at some places the arguments might look similar to, for instance, those appearing in \cite{BanWang95}, there is no apparent connection between them.

\begin{proof}[Proof of \eqref{eq:HTnormboudfororcontmart} for finite-dimensional $X$]
 We split the reasoning into several intermediate parts.

\smallskip

{\em Step 1. Some reductions.} First assume that the function $U_{\Phi, \Psi}$ (defined in Theorem~\ref{thm:existenceofU^H_Phi,Psi}) is continuous and twice differentiable. Since $N$ has continuous paths almost surely, we may assume that $N$ is a bounded martingale: this is due to a simple stopping time argument combined with Lemma \ref{lemma:sttimeargforconvexfuncandmart}. Moreover, we may assume that $\mathbb E \Phi(M_t)<\infty$, since otherwise there is nothing to prove. Let $d$ be the dimension of $X$.
 Then analogously to \cite[Section 4]{Y17MartDec} we can find a continuous time-change $\tau = (\tau_s)_{s\geq 0}$ and redefine $M:= M\circ \tau$ and $N:= N\circ\tau$, so that the following holds. For some $2d$-dimensional standard Brownian motion $W$ on an extended probability space $(\widetilde {\Omega}, \widetilde {\mathcal F}, \widetilde {\mathbb P})$ equipped with an extended filtration $\widetilde {\mathbb F}=(\widetilde {\mathcal F}_t)_{t\geq 0}$, there exist progressively measurable processes $\phi, \psi:\mathbb R_+ \times\Omega \to \mathcal L(\mathbb R^{2d},X)$ such that $M^c = \phi \cdot W$ and $N=\psi \cdot W$, where $M=M^c + M^d$ is the Meyer-Yoeurp decomposition of $M$ (see \cite{Kal,Y17MartDec,Y17GMY}). In addition, the arguments in  \cite[Section 4]{Y17MartDec} also yield the identities
  $[M\circ \tau] = [M]\circ \tau$, $[N\circ\tau] = [N]\circ\tau$ and $[M\circ \tau, N\circ \tau] = [M,N]\circ \tau$,
so the weak differential subordination and orthogonality are not ruined under the time-change.

Now, for each $n\geq 1$, introduce the stopping time
\begin{equation}\label{eq:defofsigma_n}
 \sigma_n:= \inf\{t\geq 0:M_t>n\}.
\end{equation}
By Lemma \ref{lemma:sttimeargforconvexfuncandmart} it is sufficient to show that
\begin{equation}\label{eq:inmain}
   \mathbb E \Psi(N_{t\wedge \sigma_n}) \leq |\mathcal H^{\mathbb T}_X|_{\Phi, \Psi}\mathbb E \Phi(M_{t\wedge \sigma_n})
\end{equation}
for any $n\geq 1$. Actually, passing to $M/n$ and $N/n$ {(and by modifying $\Phi$ and $\Psi$ accordingly)}, we see that it is enough to show the above estimate for $n=1$. For the sake of notational convenience, we redefine $M:= M^{\sigma_1}$ and $N:= N^{\sigma_1}$ and observe that it suffices to show $\mathbb E U_{\Phi, \Psi}(M_t + iN_t) \geq 0$, since then \eqref{eq:inmain} follows at once from the majorization property of $U_{\Phi,\Psi}$.

\smallskip

{\em Step 2. Application of It\^o's formula.} Let $(e_n)_{n=1}^d$ be a basis of $X$, and $(e_n^*)_{n=1}^d$ be the corresponding dual basis. Then by the It\^o formula \eqref{eq:itoformula}, we get
\begin{equation}\label{eq:U_PhiPsi(M_t+iN_t)}
\begin{split}
   \mathbb E U_{\Phi, \Psi}(M_t + iN_t) = \mathbb E U_{\Phi, \Psi}(M_0+iN_0) &+ \mathbb E \int_{0}^t \langle \partial_x U_{\Phi, \Psi}(M_{s-}+iN_{s-}),\ud M_s\rangle\\
  &+\mathbb E \int_{0}^t \langle \partial_{ix} U_{\Phi, \Psi}(M_{s-}+iN_{s-}),\ud N_s\rangle+ \mathbb E I_1 + \mathbb E I_2,
  \end{split}
\end{equation}
where {$\partial_x U_{\Phi, \Psi}(\cdot),\partial_{ix} U_{\Phi, \Psi}(\cdot)\in X^*$} are the corresponding {Fr\'echet} derivatives of $U_{\Phi, \Psi}$ in the real and the imaginary subspaces of $X+iX$ respectively,
\begin{align*}
 I_1 =\sum_{0\leq s\leq t}(\Delta U_{\Phi,\Psi}(M_s + iN_s) - \langle \partial_x U_{\Phi, \Psi}(M_{s-} + iN_{s-}),\Delta M_s\rangle),
\end{align*}
and
\begin{align*}
 I_2 &= \frac 12 \int_0^t\sum_{i,j=1}^d \frac{\partial^2 U_{\Phi, \Psi}(M_{s-}+iN_{s-})}{\partial e_i\partial e_j}\ud[\langle M^c, e_i^*\rangle, \langle M^c, e_j^*\rangle]_s\\
 &\quad +\frac 12 \int_0^t\sum_{i,j=1}^d \frac{\partial^2 U_{\Phi, \Psi}(M_{s-}+iN_{s-})}{\partial ie_i\partial ie_j} \ud[\langle N, e_i^*\rangle, \langle N, e_j^*\rangle]_s\\
 &\quad + \int_0^t\sum_{i,j=1}^d\frac{\partial^2 U_{\Phi, \Psi}(M_{s-}+iN_{s-})}{\partial e_i\partial ie_j} \ud[\langle M^c, e_i^*\rangle, \langle N, e_j^*\rangle]_s\\
 & =\frac 12 \int_0^t\sum_{i,j=1}^d \frac{\partial^2 U_{\Phi, \Psi}(M_{s-}+iN_{s-})}{\partial e_i\partial e_j}\langle \phi^*(s) e_i^*, \phi^*(s) e_j^*\rangle\ud s\\
  &\quad +\frac 12 \int_0^t\sum_{i,j=1}^d \frac{\partial^2 U_{\Phi, \Psi}(M_{s-}+iN_{s-})}{\partial ie_i\partial ie_j}\langle \psi^*(s) e_i^*, \psi^*(s) e_j^*\rangle\ud s\\
 &\quad + \int_0^t\sum_{i,j=1}^d \frac{\partial^2 U_{\Phi, \Psi}(M_{s-}+iN_{s-})}{\partial e_i\partial ie_j}\langle \phi^*(s) e_i^*, \psi^*(s) e_j^*\rangle\ud s.
\end{align*}

{\em Step 3. Analysis of the terms on the right of \eqref{eq:U_PhiPsi(M_t+iN_t)}.}
Let us first show that
\[
  \mathbb E \int_{0}^t \langle \partial_x U_{\Phi, \Psi}(M_{s-}+iN_{s-}),\ud M_s\rangle+\mathbb E \int_{0}^t \langle \partial_{ix} U_{\Phi, \Psi}(M_{s-}+iN_{s-}),\ud N_s\rangle
\]
exists and equals zero.
First notice that since $M = M^{\sigma_1}$, the variable $M_{s-}$ is bounded by $1$ for any $0\leq s\leq \sigma_1$. Furthermore, as we have assumed above, the process $N$ is also bounded. Since $U_{\Phi,\Psi}$ is twice differentiable, both $\partial_x U_{\Phi, \Psi}(\cdot)$ and $\partial_{ix} U_{\Phi, \Psi}(\cdot)$ are continuous functions, so $s\mapsto \partial_x U_{\Phi, \Psi}(M_{s-}+iN_{s-})$ and $s\mapsto\partial_{ix} U_{\Phi, \Psi}(M_{s-}+iN_{s-})$ define bounded processes on $0\leq s\leq \sigma_1$. Furthermore, it is easy to see that
\[
 \mathbb E M_t^{*} = \mathbb E M_{t\wedge \sigma_1}^*\leq \mathbb E \|M_{t\wedge \sigma_1}\| + 1  \leq  \mathbb E \|M_t\|+1 <\infty,
\]
and hence by Lemma \ref{lem:stochintfor*bddmart},
\begin{equation}\label{eq:stochintstoppedatsigma_1}
 \begin{split}
   t\mapsto \int_{0}^t \langle \partial_x U_{\Phi, \Psi}(M_{s-}+iN_{s-})\mathbf 1_{s\in[0,\sigma_1]},\ud M_s\rangle,\;\;\; t\geq 0,\\
 t\mapsto \int_{0}^t \langle \partial_{ix} U_{\Phi, \Psi}(M_{s-}+iN_{s-})\mathbf 1_{s\in[0,\sigma_1]},\ud N_s\rangle,\;\;\; t\geq 0,
 \end{split}
\end{equation}
define martingales. Moreover, with probability $1$,
\begin{align*}
\int_{0}^t \langle \partial_x U_{\Phi, \Psi}(M_{s-}+iN_{s-})\mathbf 1_{s\in[0,\sigma_1]},\ud M_s\rangle = \int_{0}^t \langle \partial_x U_{\Phi, \Psi}(M_{s-}+iN_{s-}),\ud M_s\rangle,\\
 \int_{0}^t \langle \partial_{ix} U_{\Phi, \Psi}(M_{s-}+iN_{s-})\mathbf 1_{s\in[0,\sigma_1]},\ud N_s\rangle=\int_{0}^t \langle \partial_{ix} U_{\Phi, \Psi}(M_{s-}+iN_{s-}),\ud N_s\rangle,
\end{align*}
since $M=M^{\sigma_1}$ and $N=N^{\sigma_1}$, and consequently the expectations of the above integrals vanish.
Let us now show that $I_1, I_2\geq 0$ almost surely. For the first term, the argument is simple: by \eqref{eq:plurisubharmonicconcaveinseconds-->convexinfirst}, each summand in $I_1$ is nonnegative. The analysis of $I_2$ is slightly more complex.
By Proposition \ref{prop:MYdecconttocontpdtopd}, we get that $N \stackrel{w}\ll M^c$ and $M^c$, $N$ are orthogonal, so Proposition \ref{prop:existofaskewsymmopvalpr} implies the existence of a progressively measurable operator-valued process $A:\mathbb R_+ \times \Omega \to \mathcal L(\mathbb R^d)$ such that $\|A\|\leq 1$, $\psi^* = A\phi^*$, and $P_{\text{Ran}(\phi^*)}A$ is skew-symmetric on $\mathbb R_+ \times \Omega$ (here $P_{\text{Ran}(\phi^*)}$ is an orthoprojection on $\text{Ran}(\phi^*)$). Thus it is enough to show that
\begin{equation}\label{eq:derofIintimetimes2}
 \begin{split}
 &\sum_{i,j=1}^d \frac{\partial^2 U_{\Phi, \Psi}(M_{s-}+iN_{s-})}{\partial e_i\partial e_j}\langle \phi^*(s) e_i^*, \phi^*(s) e_j^*\rangle\ud s\\
  &\quad +\sum_{i,j=1}^d \frac{\partial^2 U_{\Phi, \Psi}(M_{s-}+iN_{s-})}{\partial ie_i\partial ie_j}\langle \psi^*(s) e_i^*, \psi^*(s) e_j^*\rangle\ud s\\
 &\quad + 2\sum_{i,j=1}^d \frac{\partial^2 U_{\Phi, \Psi}(M_{s-}+iN_{s-})}{\partial e_i\partial ie_j}\langle \phi^*(s) e_i^*, P_{\text{Ran}(\phi^*)}A\phi^*(s) e_j^*\rangle\ud s \geq 0.
 \end{split}
\end{equation}
By the spectral theory of skew-symmetric matrices (see e.g.\ \cite[Corollary 2]{You61}) there exist $L\geq 0$, positive numbers $(\lambda_n)_{n=1}^L$ and an orthonormal basis $(h_n)_{n=1}^{2d}$ of $\mathbb R^{2d}$ such that $P_{\text{Ran}(\phi^*)}Ah_{2n-1} = \lambda_n h_{2n}$ and $P_{\text{Ran}(\phi^*)}Ah_{2n} = -\lambda_n h_{2n-1}$ for all $n=1,\ldots,L$, and $P_{\text{Ran}(\phi^*)}Ah_{n} =0$ for all $2L<n\leq d$. Moreover, the condition $\|A\|\leq 1$ implies that $|\lambda_1|,\ldots,|\lambda _L|\leq 1$, and since $(\text{Ran}(\phi^*))^{\perp}$ is a zero eigenspace of $P_{\text{Ran}(\phi^*)}A$ (see the construction of $A$ in the proof of Proposition \ref{prop:existofaskewsymmopvalpr}), we conclude that $h_{n} \in \text{Ran}(\phi^*)$ for $n=1,\,2,\,\ldots,2L$. By a usual orthogonalization procedure, we may assume that there exists $K\geq 2L$ such that $h_n\in  \text{Ran}(\phi^*)$ for $2L<n\leq K$ and $h_n\bot \text{Ran}(\phi^*)$ for $K<n\leq 2d$ (then $K$ is the dimension of $\text{Ran}(\phi^*)$). Notice that $X^*$ is $d$-dimensional, so $\text{Ran}(\phi^*)$ is at most $d$-dimensional and hence obviously $K\leq d$. Due to Lemma \ref{lemma:traceindep}, the expression \eqref{eq:derofIintimetimes2} does not depend on the basis $(e_n)_{n=1}^d$ (and the corresponding dual basis $(e^*_n)_{n=1}^d$), so we can choose a basis $(e_n)_{n=1}^d$ such that $\phi^*e_n^* = h_n$ for all $n=1,\ldots,K$ and $\phi^*e_n^* = 0$ for all $K<n\leq d$ (such a basis exists since $\text{span}\{h_1,\ldots,h_K\}= \text{Ran}(\phi^*)$). Then \eqref{eq:derofIintimetimes2} becomes
\begin{equation}\label{eq:derofIintimetimes2ver23}
 \begin{split}
    \sum_{i,j=1}^K& \frac{\partial^2 U_{\Phi, \Psi}(M_{s-}+iN_{s-})}{\partial e_i\partial e_j}\langle h_i, h_j\rangle \\
    &+\sum_{i,j=1}^K \frac{\partial^2 U_{\Phi, \Psi}(M_{s-}+iN_{s-})}{\partial ie_i\partial ie_j}\langle \psi^* e_i^*, \psi^* e_j^*\rangle\\
     &\quad+2\sum_{i,j=1}^K \frac{\partial^2 U_{\Phi, \Psi}(M_{s-}+iN_{s-})}{\partial e_i\partial ie_j}\langle h_i, P_{\text{Ran}(\phi^*)}Ah_j\rangle\geq 0
    \end{split}
\end{equation}
(The second sum is up to $K$ due to the fact that $\phi^*x^*=0$ implies $\psi^*x^*=0$ for any $x^*\in X^*$, see \eqref{eq:ineqweakonphiandpsi}).
Notice that the bilinear form $V:X\times X \to \mathbb R$ defined by
$$
V(x,y) := -\frac{\partial^2 U_{\Phi, \Psi}(M_{s-}+iN_{s-})}{\partial ix\partial iy},\;\;\; x,y\in X,
$$
is nonnegative by Theorem \ref{thm:existenceofU^H_Phi,Psi} and symmetric by the definition. Moreover, by \eqref{eq:ineqweakonphiandpsi},
$$
\langle \psi^* x^*, \psi^* x^*\rangle = \|\psi^* x^*\|^2 \leq \|\phi^* x^*\|^2=\langle \phi^* x^*, \phi^* x^*\rangle,\;\;\; \mbox{for }x^*\in X^*.
$$
Therefore Corollary \ref{cor:lemmatraceindep} yields
\begin{equation*}
 \begin{split}
   \sum_{i,j=1}^K \frac{\partial^2 U_{\Phi, \Psi}(M_{s-}+iN_{s-})}{\partial ie_i\partial ie_j}\langle \psi^* e_i^*, \psi^* e_j^*\rangle &\geq \sum_{i,j=1}^K \frac{\partial^2 U_{\Phi, \Psi}(M_{s-}+iN_{s-})}{\partial ie_i\partial ie_j}\langle \phi^* e_i^*, \phi^* e_j^*\rangle\\
   &\quad =\sum_{i,j=1}^K \frac{\partial^2 U_{\Phi, \Psi}(M_{s-}+iN_{s-})}{\partial ie_i\partial ie_j}\langle h_i, h_j\rangle,
    \end{split}
\end{equation*}
so \eqref{eq:derofIintimetimes2ver23} is not less than
\begin{align*}
 \sum_{i,j=1}^K \frac{\partial^2 U_{\Phi, \Psi}(M_{s-}+iN_{s-})}{\partial e_i\partial e_j}\langle h_i, h_j\rangle &+\sum_{i,j=1}^K \frac{\partial^2 U_{\Phi, \Psi}(M_{s-}+iN_{s-})}{\partial ie_i\partial ie_j}\langle h_i, h_j\rangle\\
    & +2\sum_{i,j=1}^K \frac{\partial^2 U_{\Phi, \Psi}(M_{s-}+iN_{s-})}{\partial e_i\partial ie_j}\langle h_i, P_{\text{Ran}(\phi^*)}Ah_j\rangle\\
    =\sum_{i=1}^K \frac{\partial^2 U_{\Phi, \Psi}(M_{s-}+iN_{s-})}{\partial e_i\partial e_i}\langle h_i, h_i\rangle &+\sum_{i=1}^K \frac{\partial^2 U_{\Phi, \Psi}(M_{s-}+iN_{s-})}{\partial ie_i\partial ie_i}\langle h_i, h_i\rangle\\
    & +2\sum_{i,j=1}^K \frac{\partial^2 U_{\Phi, \Psi}(M_{s-}+iN_{s-})}{\partial e_i\partial ie_j}\langle h_i, P_{\text{Ran}(\phi^*)}Ah_j\rangle.
\end{align*}
The latter expression consists of two parts:
    \begin{equation}\label{eq:derofIintimetimes2ver2}
 \begin{split}
  &\sum_{i=1}^{2L} \frac{\partial^2 U_{\Phi, \Psi}(M_{s-}+iN_{s-})}{\partial e_i\partial e_i} +\sum_{i=1}^{2L}\frac{\partial^2 U_{\Phi, \Psi}(M_{s-}+iN_{s-})}{\partial ie_i\partial ie_i}\\
 & +2\sum_{n=1}^L \lambda_n \Bigl(\frac{\partial^2 U_{\Phi, \Psi}(M_{s-}+iN_{s-})}{\partial e_{2n-1}\partial ie_{2n}}- \frac{\partial^2 U_{\Phi, \Psi}(M_{s-}+iN_{s-})}{\partial e_{2n}\partial ie_{2n-1}}\Bigr) \\
 =&\,\sum_{n=1}^L\biggl\{\frac{\partial^2 U_{\Phi, \Psi}(M_{s-}+iN_{s-})}{\partial e_{2n-1}\partial e_{2n-1}} + \frac{\partial^2 U_{\Phi, \Psi}(M_{s-}+iN_{s-})}{\partial e_{2n}\partial e_{2n}}\\
 &\quad +2 \lambda_n \Bigl(\frac{\partial^2 U_{\Phi, \Psi}(M_{s-}+iN_{s-})}{\partial e_{2n-1}\partial ie_{2n}} - \frac{\partial^2 U_{\Phi, \Psi}(M_{s-}+iN_{s-})}{\partial e_{2n}\partial ie_{2n-1}}\Bigr) \\
 &\quad +\Bigl(\frac{\partial^2 U_{\Phi, \Psi}(M_{s-}+iN_{s-})}{\partial ie_{2n-1}\partial ie_{2n-1}} + \frac{\partial^2 U_{\Phi, \Psi}(M_{s-}+iN_{s-})}{\partial ie_{2n}\partial ie_{2n}}\Bigr)\biggr\}
 \end{split}
\end{equation}
and
    \begin{equation}\label{eq:derofIintimetimes2ver2018}
 \begin{split}
  \sum_{i=2L+1}^{K} \frac{\partial^2 U_{\Phi, \Psi}(M_{s-}+iN_{s-})}{\partial e_i\partial e_i} +\sum_{i=2L+1}^{K}\frac{\partial^2 U_{\Phi, \Psi}(M_{s-}+iN_{s-})}{\partial ie_i\partial ie_i}\\
  = \sum_{i=2L+1}^{K}\Bigl(\frac{\partial^2 U_{\Phi, \Psi}(M_{s-}+iN_{s-})}{\partial e_i\partial e_i} + \frac{\partial^2 U_{\Phi, \Psi}(M_{s-}+iN_{s-})}{\partial ie_i\partial ie_i} \Bigr).
 \end{split}
\end{equation}
Now, the expression \eqref{eq:derofIintimetimes2ver2} is nonnegative by Corollary \ref{cor:propnicesecderofU_PhiPsi} and \eqref{eq:derofIintimetimes2ver2018} is nonnegative by Remark \ref{rem:xzplurisubhstaff}. This gives $I_2 \geq 0$. Putting all the above facts together, we obtain
\[
 \mathbb E U_{\Phi, \Psi}(M_t + iN_t) \geq \mathbb EU_{\Phi, \Psi}(M_0 + iN_0).
\]
However, by Remark \ref{rem:orth+wds}, we have $N_0=0$ almost surely, so Theorem~\ref{thm:existenceofU^H_Phi,Psi} implies
\[
 \mathbb EU_{\Phi, \Psi}(M_0 + iN_0) = \mathbb EU_{\Phi, \Psi}(M_0) \geq 0,
\]
which completes the proof.

\smallskip

{\em Step 4.} Now we assume that $U_{\Phi,\Psi}$ is general (i.e., not necessarily twice integrable). We will use a standard mollification argument. Let $\phi:X+iX \to \mathbb R_+$ be a $C^{\infty}$ radial function with compact support such that $\int_{X+iX}\phi(s)\ud s = 1$. For each $\eps>0$, define $U^{\eps}_{\Phi,\Psi}:X+iX\to \mathbb R$ via the convolution
\[
 U^{\eps}_{\Phi,\Psi}(x+iy) :=\int_{X+iX} U_{\Phi,\Psi}(x+iy - \eps s)\phi(s)\ud s,\;\;\; x,y\in X.
\]
Then $U^{\eps}_{\Phi,\Psi}$ is of class $C^{\infty}$ and for any $x\in X$ we have
\begin{equation}\label{eq:U^epsisposonX}
  U^{\eps}_{\Phi,\Psi}(x) =\int_{X+iX} U_{\Phi,\Psi}(x - \eps s)\phi(s)\ud s \geq  U_{\Phi,\Psi}(x)\geq 0,
\end{equation}
since $U_{\Phi,\Psi}$ is subharmonic (see Remark \ref{rem:PSHeither-inftyorL^1loc}).
Therefore, repeating the arguments from the above steps, we get
\begin{equation}\label{eq:epspertofUPhiPsi}
 \begin{split}
   \mathbb E \int_{X+iX}\Bigl[|\mathcal H^{\mathbb T}_X|_{\Phi, \Psi}&\Phi(M_t-\eps r) - \Psi(N_t-\eps u)\Bigr] \phi(r+iu)\ud {r + iu}\geq \mathbb EU^{\eps}_{\Phi,\Psi}(M_t + iN_t)\geq \mathbb E U^{\eps}_{\Phi,\Psi}(M_0) \geq 0,
 \end{split}
\end{equation}
where the latter bound follows from \eqref{eq:U^epsisposonX}. Note that $\Psi(N_t + \eps u)$ is uniformly bounded (when $r+iu$ runs over the support of $\phi$) and notice that for any $x$,  $\eps \mapsto \frac{\Phi(x-\eps)+\Phi(x+\eps)}{2}$ is an increasing function of $\e>0$. Furthermore, we have $\phi(r+iu)=\phi(-r+iu)\geq 0$ and hence
\begin{equation}\label{eq:approxofconvexPhi}
\begin{split}
   \eps \mapsto \int_{X+iX}\Phi(M_t-\eps r) \phi(r+iu)\ud(r+iu)  = \int_{X+iX}\frac{\Phi(M_t-\eps r) + \Phi(M_t+\eps r)}{2} \phi(r+iu)\ud(r+iu),
\end{split}
\end{equation}
decreases as $\e\downarrow 0$. Combining these observations with standard limiting theorems, we deduce the desired claim.
\end{proof}

Now we prove our main result in full generality. Of course, we will exploit an appropriate limiting procedure, which enables us to deduce the claim from its finite-dimensional version just established above.

\begin{proof}[Proof of \eqref{eq:HTnormboudfororcontmart} for infinite-dimensional $X$] We may assume that $\E \Phi(M_t)<\infty$, since otherwise the claim is obvious.
Suppose that $(Y_n)_{n\geq 1}$ is a se\-que\-n\-ce of finite-dimensional subspaces of $X^*$ such that $Y_{n}\subset Y_{n+1}$ for any $n\geq 1$ and $\overline{\cup_{n\geq 1} Y_n}=X^*$. For each $n\geq 1$ define $X_n := Y_n^*$, let $P_n:Y_n \hookrightarrow X^*$ be the corresponding embedding operator and let $P_n^*:X\to X_n$ be its adjoint  (recall that $X$ is reflexive). Finally, define $ {\Phi}_n, {\Psi}_n:X_n \to \mathbb R_+$ by the formulae
\begin{align*}
  {\Phi}_n(\widetilde x) = \inf\{\Phi(x)\,:\,{x\in X,\,P_n^* x=\tilde x}\}, \quad
  {\Psi}_n(\widetilde x) = \inf\{\Psi(x)\,:\,{x\in X,\,P_n^* x=\tilde x}\},
\end{align*}
for $\tilde x\in X_n$.
In the light of  Lemma \ref{lem:Phi_YPsi_Yareconvex}, both $  {\Phi}_n$ and $ {\Psi}_n$ are convex functions. Moreover, by Proposition \ref{prop:normPhi_YPsi_YlessnormPhiPsi},
\begin{equation}\label{eq:normofHforPhinPsinisless}
  |\mathcal H^{\mathbb T}_{X_n}|_{ {\Phi}_n, {\Psi}_n} \leq |\mathcal H^{\mathbb T}_{X}|_{{\Phi}, {\Psi}}.
\end{equation}
Let us show that the processes $P_n^*M$ and $P_n^*N$ are orthogonal for each $n\geq 1$. By the very definition, we must prove that for a fixed functional $x^*\in X_n^*$, the local martingales $\langle P_n^*M, x^*\rangle$ and $\langle P_n^*N, x^*\rangle$ are orthogonal. This follows at once from orthogonality of $M$, $N$  and the identities
\begin{equation}\label{eq:P_n^*MandP_n^*NareorthandWDS}
\begin{split}
  \langle P_n^*M, x^*\rangle &= \langle M, P_nx^*\rangle, \qquad  \langle P_n^*N, x^*\rangle = \langle N, P_nx^*\rangle.
\end{split}
\end{equation}
These identities also immediately give the weak differential subordination $P_n^*N \stackrel{w}\ll P_n^*M$, since $M$, $N$ enjoy this condition.
Finally, observe that  by Lemma \ref{lem:Phi_YPsi_Yareconvex}, we have $ \mathbb E \Phi_n(P^*_n M_t)\leq \mathbb E \Phi(M_t)<\infty$.
Therefore, applying the finite-dimensional version of \eqref{eq:HTnormboudfororcontmart}, we see that for each $n\geq 1$,
\begin{equation}\label{eq:ineqforPhi_n(P_n^*M)andPsi_n(P_n^*N)}
  \mathbb E \Psi_n(P^*_n N_t) \leq |\mathcal H^{\mathbb T}_{X_n}|_{\Phi_n, \Psi_n}\mathbb E \Phi_n(P^*_n M_t) \leq |\mathcal H^{\mathbb T}_{X}|_{\Phi, \Psi}\mathbb E \Phi_n(P^*_n M_t),
\end{equation}
where the second passage is due to \eqref{eq:normofHforPhinPsinisless}. Note that with probability $1$ we have $\Phi_n(P_n^* M_t)\nearrow \Phi(M_t)$ and  $\Psi_n(P_n^* N_t)\nearrow \Psi(N_t)$ monotonically as $n\to \infty$ by Lemma \ref{lem:Phi_nPsi_nconvergestoPhiPsi}. This establishes the desired estimate, by Lebesgue's monotone convergence theorem.
\end{proof}

It remains to handle the sharpness of \eqref{eq:HTnormboudfororcontmart}.

\begin{proof}[Proof of the estimate $|\mathcal H_{X}^{\mathbb T}|_{\Phi, \Psi}\leq C_{\Phi,\Psi, X}$]
 This follows immediately from the reasoning presented in Section \ref{sec:orthodiff}: indeed, \eqref{eq:HTnormboudfororcontmart} implies the corresponding bound
$$ \int_\mathbb{T}\Psi\big(\mathcal{H}^\mathbb{T}_X f\big)\mbox{d}x\leq C_{\Phi,\Psi, X}\int_\mathbb{T}\Phi( f)\mbox{d}x$$
for any step function $f:\mathbb{T}\to X$.
\end{proof}

\begin{remark}\label{rem:PhiisnotconvexifXisfd}
It is easy to see that if $X$ is finite dimensional, then there is no need for $\Phi$ to be convex. The limiting argument presented in the above proof does not need this requirement. {(The only place where the convexity of $\Phi$ is used is \eqref{eq:approxofconvexPhi}; we leave to the reader the question how to avoid this issue).}
\end{remark}

\section{Applications}\label{sec:applications}

\subsection{Hilbert transforms on $\mathbb T$, $\mathbb R$, and $\mathbb Z$}

Let $X$ be a Banach space and let $\Phi,\Psi:X\to \mathbb R_+$ be continuous functions. Let $(S, \Sigma,\mu)$ be a measure space, with $S$ equal to $\mathbb T$, $\mathbb R$, or $\mathbb Z$. A~function $f:S\to X$ is called a {\em step function}, if it is of the form
$$
f(t)=\sum_{k=1}^N x_k\mathbf 1_{A_k}(t),\;\;\; t\in S,
$$
where $N$ is finite, $x_k\in X$ and $A_k$ are intervals in $S$ of a finite measure.

\begin{definition}\label{def:defofRHT}
 The {\em Hilbert transform} $\mathcal H_{X}^{\mathbb R}$ is a linear operator that maps a step function $f:\mathbb R\to X$ to the function
 \begin{equation}\label{eq:defdefofRHT}
   (\mathcal H^{\mathbb R}_X f)(t):= \frac{1}{\pi}{\textnormal{p.v.}}\int_{\mathbb R}\frac{f(s)}{t-s}\ud s,\;\;\; t\in \mathbb R.
 \end{equation}
The associated $\Phi,\Psi$-norms $|\mathcal H_{X}^{\mathbb R}|_{\Phi,\Psi}$ are given by a formula similar to that used previously:
\begin{align*}
|\mathcal H_{X}^{\mathbb R}|_{\Phi,\Psi}:=\inf\Biggl\{c\in [0,\infty]: \int_{\mathbb R} \Psi(\mathcal H_X^{\mathbb R}f(s))\ud s\leq c\int_{\mathbb R} \Phi(f(s))\ud s\mbox{ for all step functions }f:\mathbb R\to X\Biggr\}.
\end{align*}
 \end{definition}

\begin{definition}\label{def:defofdisHT}
 The {\em discrete Hilbert transform} $\mathcal H_{X}^{\rm dis}$ is a linear operator that maps a step function $f:\mathbb Z \to X$ to the function
\[
  (\mathcal H^{\rm dis}_X f)(t):= \frac{1}{\pi}\sum_{s\in \mathbb Z \setminus \{t\}}\frac{f(s)}{t-s},\;\;\; t\in \mathbb Z.
\]
The associated $\Phi,\Psi$-norms $|\mathcal H_{X}^{\rm dis}|_{\Phi,\Psi}$ are given by
\begin{align*}
|\mathcal H_{X}^{\rm dis}|_{\Phi,\Psi}:=\inf\Biggl\{c\in [0,\infty]: \sum_{s\in \mathbb Z} \Psi(\mathcal H_X^{\rm dis}f(s))&\leq c\sum_{s\in \mathbb Z} \Phi(f(s))\mbox{ for all step functions }f:\mathbb Z\to X\Biggr\}.
\end{align*}
\end{definition}

We will also need a certain variant of $\Phi,\Psi$-norm in the periodic setting. Namely, define $|\mathcal{H}_X^{\mathbb{T},0}|_{\Phi,\Psi}$ by
\begin{align*}
|\mathcal{H}_X^{\mathbb{T},0}|_{\Phi,\Psi}:=\inf\Bigg\{c\in [0,\infty]&: \int_\mathbb{T} \Psi(\mathcal{H}^\mathbb{T}_Xf(s))\mbox{d}s\leq c\int_\mathbb{T} \Phi(f(s))\mbox{d}s\\
&\mbox{for all step functions }f:\mathbb{T}\to X\mbox{ with }\int_\mathbb{T}f(s)\mbox{d}s=0\Bigg\}.
\end{align*}

The following theorem demonstrates that the norm of the Hilbert transform does not depend whether it is defined on $\mathbb T$, $\mathbb R$, or $\mathbb Z$.

\begin{theorem}\label{thm:HTnormthesameforTSZT,0}
 Let $X$ be a Banach space and let $\Phi, \Psi:X\to \mathbb R$ be continuous convex functions such that $\Phi(0)=0$. Then
 \[
  |\mathcal H_X^{\mathbb T,0}|_{\Phi,\Psi} = |\mathcal H_X^{\mathbb R}|_{\Phi,\Psi}\leq|\mathcal H_X^{\rm dis}|_{\Phi,\Psi}\leq | \mathcal H_X^{\mathbb T}|_{\Phi,\Psi}.
 \]
 Moreover, if $\Phi$ is symmetric, then
  \[
 |\mathcal H_X^{\mathbb T,0}|_{\Phi,\Psi} = |\mathcal H_X^{\mathbb R}|_{\Phi,\Psi}=|\mathcal H_X^{\rm dis}|_{\Phi,\Psi}= | \mathcal H_X^{\mathbb T}|_{\Phi,\Psi}.
 \]
\end{theorem}

The proof will consist of several steps.

\begin{proposition}\label{prop:ineqforHTonRdisandT}
 Let $X$ be a Banach space and let $\Phi,\Psi:X\to \mathbb R_+$ be convex functions. Then we have
 \[
  |\mathcal H^{\mathbb R}_X|_{\Phi, \Psi}\leq |\mathcal H^{\rm dis}_X|_{\Phi, \Psi} \leq |\mathcal H^{\mathbb T}_X|_{\Phi, \Psi}.
 \]
\end{proposition}

\begin{proof}
Introduce yet another Hilbert-type operator acting on step functions $f:\mathbb{R}\to \R$ by
\[
  (\mathcal H^{\mathbb R, \rm dis}_Xf)(t) := \frac{1}{\pi}\sum_{s\in \mathbb Z\setminus\{0\}}\frac{f(t-s)}{s},\;\;\; t\in \mathbb R,
 \]
and define its $\Phi,\Psi$-norm analogously.
We will first prove that $|\mathcal H^{\mathbb R}_X|_{\Phi, \Psi}\leq|\mathcal H^{\mathbb R, \rm dis}_X|_{\Phi, \Psi}$. To this end,  fix a step function $f$ on $\R$ and define its $\e$-dilation by $f_{\eps}(\cdot):= f(\eps \cdot)$. Then similarly to \cite[Theorem 4.3]{Lae07}, we have
\begin{align*}
 \frac{\int_{\mathbb R} \Psi((\mathcal H^{\mathbb R, \rm dis}_Xf_{\eps})(s))\ud s}{\int_{\mathbb R}\Phi(f_{\eps}(s))\ud s} &= \frac{\int_{\mathbb R} \Psi(\pi^{-1}\sum_{k\in \mathbb Z\setminus\{0\}}{f_{\eps}(s-k)}/{k})\ud s}{\int_{\mathbb R}\Phi(f_{\eps}(s))\ud s}\\
 &=\frac{\int_{\mathbb R} \Psi(\pi^{-1}\sum_{k\in \mathbb Z\setminus\{0\}}{\eps f(\eps s-\eps k)}/({\eps k}))\ud (\eps s)}{\int_{\mathbb R}\Phi(f(\eps s))\ud (\eps s)}\\
  &=\frac{\int_{\mathbb R} \Psi(\pi^{-1}\sum_{k\in \mathbb Z\setminus\{0\}}{\eps f(s-\eps k)}/({\eps k}))\ud s}{\int_{\mathbb R}\Phi(f(s))\ud s}.
\end{align*}
Since $\frac{1}{\pi}\sum_{k\in \mathbb Z\setminus\{0\}}\frac{f(s-\eps k)}{\eps k}\eps\to \mathcal H^{\mathbb R}_X f(s)$ for a.e.\ $s\in \mathbb R$, Fatou's lemma yields
\begin{align*}
 |\mathcal H^{\mathbb R}_X|_{\Phi, \Psi} = \sup_{f\in F^{\rm step}_X}\frac{\int_{\mathbb R} \Psi(\mathcal H^{\mathbb R}_X f(s))\ud s}{\int_{\mathbb R}\Phi(f(s))\ud s}&\leq \sup_{f\in F^{\rm step}_X} \liminf_{\eps\to 0}\frac{\int_{\mathbb R} \Psi((\mathcal H^{\mathbb R, \rm dis}_Xf_{\eps})(s))\ud s}{\int_{\mathbb R}\Phi(f_{\eps}(s))\ud s}\leq |\mathcal H^{\mathbb R,\rm dis}_X|_{\Phi, \Psi} = |\mathcal H^{\rm dis}_X|_{\Phi, \Psi}.
\end{align*}
where the latter equality follows from the direct repetition of the arguments from \cite[Theorem 4.2]{Lae07}. This gives us the first inequality of the assertion.
The proof of the fact that $|\mathcal H^{\rm dis}_X|_{\Phi, \Psi} \leq |\mathcal H^{\mathbb T}_X|_{\Phi, \Psi}$ follows word-by-word from the infinite-di\-men\-sio\-nal analogue of the recent approach of Ba{\~n}uelos and Kwa{\'s}nicki \cite{BanKw17} combined with the estimate \eqref{eq:HTnormboudfororcontmart}.
\end{proof}

\begin{theorem}\label{thm:HXRleqHXT,0}
Let $X$ be a Banach space and let $\Phi,\,\Psi:X\to \R_+$ be continuous functions. Then $|\mathcal{H}_X^\R|_{\Phi,\Psi}\leq |\mathcal{H}_X^{\mathbb{T},0}|_{\Phi,\Psi}$.
\end{theorem}
\begin{proof}
Fix a step function $f:\R\to X$. It takes only a finite number of values, so we may assume that $X$ is finite dimensional (which will guarantee the validity of the reasoning below). For any $n\geq 1$, introduce the function $g_n:\R\to X$ by
$$
g_n(x)=\frac{1}{2\pi n}\int_{-\pi n}^{\pi n}f(t)\cot\frac{x-t}{2n}\mbox{d}t,\;\;\; x\in \mathbb R.
$$
It follows from the observation of Zygmund \cite[p.\ 256]{ZygTS02} that  $g_n\to \mathcal{H}^\R_X f$ a.e.
as $n\to \infty$. On the other hand, the function $x\mapsto g_n(nx)$, $|x|\leq \pi$, is precisely the periodic Hilbert transform of the function $x\mapsto f(nx)$, $|x|\leq \pi$ (see \eqref{eq:periodic}). Therefore, it is also the periodic Hilbert transform of the \emph{centered} function
$$ x\mapsto f(nx)-\frac{1}{2\pi}\int_{-\pi}^\pi f(ns)\mbox{d}s, \qquad |x|\leq \pi.$$
Clearly, the latter is a step function. Consequently, by Fatou's lemma and the definition of $|\mathcal{H}_X^{\mathbb{T},0}|_{\Phi,\Psi}$,
\begin{align*}
 \int_\R \Psi(\mathcal{H}^\R_X f)\mbox{d}x&\leq \liminf_{n\to \infty}\int_{-\pi n}^{\pi n} \Psi(g_n)\mbox{d}x= \liminf_{n\to \infty}\int_{-\pi }^{\pi } \Psi(g_n(nx))n\mbox{d}x\\
 &\leq |\mathcal{H}_X^{\mathbb{T},0}|_{\Phi,\Psi} \liminf_{n\to \infty} \int_{-\pi}^\pi \Phi\left(f(nx)-\frac{1}{2\pi}\int_{-\pi}^\pi f(ns)\mbox{d}s\right)n\mbox{d}x\\
 &= |\mathcal{H}_X^{\mathbb{T},0}|_{\Phi,\Psi} \liminf_{n\to \infty} \int_{-\pi n}^{\pi n} \Phi\left(f(x)-\frac{1}{2\pi n}\int_{-\pi n}^{\pi n} f(s)\mbox{d}s\right)\mbox{d}x.
\end{align*}
However, $\frac{1}{2\pi n}\int_{-\pi n}^{\pi n} f(s)\mbox{d}s \to 0$ by the fact that $f$ is a step function. Therefore, again using this property of $f$ and the continuity of $\Phi$, the last expression of the above chain equals
$|\mathcal{H}_X^{\mathbb{T},0}|_{\Phi,\Psi}\int_\R \Phi(f)\mbox{d}x$.
Since $f$ was arbitrary, the result follows.
\end{proof}

Now we turn our attention to the estimate in the reverse direction. We start from the observation that it does not hold true if $\Phi(0)>0$ and $\Psi\neq 0$. Indeed, if $\Phi(0)>0$, then $\int_\R \Phi(f)\mbox{d}x=\infty$ for any step function  and hence $|\mathcal{H}^\R|_{\Phi,\Psi}=0$.  On the other hand, the condition $\Psi\neq 0$ implies that $|\mathcal{H}_X^{\mathbb{T},0}|_{\Phi,\Psi}>0$: it is easy to construct a step function $f:\mathbb{T}\to X$ of mean zero for which $\int_\R \Psi(\mathcal{H}^\mathbb{T}f)\mbox{d}x>0$.

In other words, the inequality $|\mathcal{H}_X^{\mathbb{T},0}|_{\Phi,\Psi}\leq |\mathcal{H}_X^\R|_{\Phi,\Psi}$ fails, because of obvious reasons, if $\Phi(0)>0$ and $\Psi\neq 0$. If $\Psi$ is identically $0$, then the estimate holds true: the reason is even more trivial -- both sides are zero. It remains to study the key possibility when $\Phi(0)=0$ and $\Psi\neq 0$.

\begin{theorem}\label{thm:HXTleqHXR}
Let $X$ be a Banach space and let $\Phi,\,\Psi:X\to \R_+$ be arbitrary continuous functions such that $\Phi(0)=0$ and $\Psi\neq 0$. Then $|\mathcal{H}_X^{\mathbb{T},0}|_{\Phi,\Psi}\leq |\mathcal{H}_X^\R|_{\Phi,\Psi}$.
\end{theorem}

\begin{proof} {As was mentioned above,} the assumption $\Psi\neq 0$ implies $|\mathcal{H}_X^{\mathbb{T},0}|_{\Phi,\Psi}>0$. For the sake of clarity, we split the reasoning into a few separate parts.

\smallskip

\emph{Step 1. Auxiliary analytic maps.}
Let $D $ denote the open unit disc of $\C$ and let $H=R\times (0,\infty)$ be the upper halfplane. Define $K:D \cap H\to H$ by the formula $K(z)=-(1-z)^2/(4z)$. It is not difficult to verify that $K$ is conformal and hence so is its inverse $L$. Let us extend $L$ to the continuous function on $\overline{H}$. It is easy to see that $L(z)\to 0$ as $z\to \infty$. Furthermore, $L$ maps the interval $[0,1]$ onto $\{e^{i\theta}:0\leq \theta\leq \pi\}$. More precisely, we have the following formula: if $x\in [0,1]$, then
\begin{equation}\label{soul}
L(x)=e^{i\theta}, \mbox{ where $\theta \in [0,\pi]$ is uniquely determined by $x=\sin^2(\theta/2)$.}
\end{equation}
In addition, $L$ maps the set $\R\setminus [0,1]$ onto the open interval $(-1,1)$; precisely, we have the identity
\begin{equation}\label{soul2}
L(x)=\left\{\begin{array}{ll}
1-2x-2\sqrt{x^2-x} & \mbox{if }x<0,\\
1-2x+2\sqrt{x^2-x} & \mbox{if }x>1.
\end{array}\right.
\end{equation}
In particular, we easily check that for any $\delta>0$, the function $L$ is bounded away from $1$ outside any interval of the form $[-\delta,1+\delta]$ and $|L(x)|=O(|x|^{-1})$ as $x\to \pm \infty$.

\smallskip

\emph{Step 2. A function on $\mathbb{T}$ and its extension to a disc}. Fix a positive number $\e$ and pick a step function $f:\mathbb{T}\to X$ of integral $0$ such that
  $$
  \int_\mathbb{T} \Psi(\mathcal{H}^\mathbb{T}_Xf)\mbox{d}x>(|\mathcal{H}_X^{\mathbb{T},0}|_{\Phi,\Psi}-\e)\cdot \int_\mathbb{T} \Phi(f)\mbox{d}x.
  $$
We may assume that $X$ is finite-dimensional, restricting to the range of $f$ if necessary. Given a big number $R>0$, consider a continuous function $\kappa^R:X\to [0,1]$ equal to $1$ on $B(0,R)$ and equal to $0$ outside $B(0,2R)$.
Set $\Psi^R(x)=\Psi(x)\cdot \kappa^R(x)$ for $x\in X$. Note that $\Psi^R$ is uniformly continuous, since it is continuous and supported on the compact ball $B(0,2R)$ (recall that $X$ is finite dimensional).
 By Lebesgue's monotone convergence theorem, if $R$ is sufficiently big, we also have
\begin{equation}\label{lower}
\int_\mathbb{T} \Psi^R(\mathcal{H}^\mathbb{T}_Xf)\mbox{d}x>(|\mathcal{H}_X^{\mathbb{T},0}|_{\Phi,\Psi}-\e)\cdot \int_\mathbb{T} \Phi(f)\mbox{d}x.
\end{equation}
There is an analytic function $F:D \to X+iX$ with the property that the radial limit $\lim_{r\to 1-} F(re^{i\theta})$ is equal to $f(e^{i\theta})+i\mathcal{H}^\mathbb{T}_Xf(e^{i\theta})$ for almost all $|\theta|\leq \pi$. Note that we have
\begin{equation}\label{zeroF}
 F(0)=\frac{1}{2\pi}\int_\mathbb{T} f\mbox{d}x+i\cdot 0=0
\end{equation}
and that the ``real part'' of $F$ is bounded (by the supremum norm of $f$).
Consider the analytic function $M_n:H\to X+iX$ given by the composition
$$
M_n(z)=F(L^{2n}(z))
$$
and decompose it as $M_n(z)=\Re{M}_n(z)+i \Im{M}_n(z)$, with $\Re{M}_n$ and $\Im{M}_n$ taking values in $X$. Observe that for each $n$ the function $\Re M_n$ is bounded by the supremum norm of $f$ (which is directly inherited from the ``real part'' of the function $F$). In addition, $h=\mathbf 1_{[0,1]}\Re M_n$, considered as a function on $\R$, is a step function (with the number of steps depending on $n$ and going to infinity). Since $\lim_{z\to\infty}L(z)=0$, we have $\lim_{z\to\infty}M_n(z)=0$ and therefore $\mathcal{H}^\mathbb{T}_X \Re M_n(x)=\Im M_n(x)$ for $x\in \R$.

\smallskip

\emph{Step 3. Calculations.} We compute that

\begin{equation}\label{chain}
\begin{split}
\int_\R \Phi\left(h(x)\right)\mbox{d}x&= \int_0^1 \Phi \left(\,\Re M_n(x))\right)\mbox{d}x=\int_0^1 \Psi \left(\,f(L^{2n}(x))\right)\mbox{d}x=\frac{1}{2}\int_0^\pi \Phi\left(\,f(e^{2in\theta})\right)\sin\theta\mbox{d}\theta\\
&=\frac{1}{2}\int_0^{2n\pi}\Phi\left(f(e^{i\theta})\right)\sin\left(\frac{\theta}{2n}\right)\frac{\mbox{d}\theta}{2n}=\frac{1}{2}\int_0^{2\pi}\Phi\left(f(e^{i\theta})\right)\sum_{k=0}^{n-1}\sin\left(\frac{k\pi}{n}+\frac{\theta}{2n}\right)\frac{\mbox{d}\theta}{2n}\\
&=\frac{1}{2}\int_0^{2\pi}\Phi\left(f(e^{i\theta})\right)\frac{\cos\left(\frac{\theta-\pi}{{2n}}\right)}{2n\sin\left(\frac{\pi}{2n}\right)}\,\mbox{d}\theta\xrightarrow{n\to\infty}\frac{1}{2\pi}\int_0^{2\pi}\Phi\left(f(e^{i\theta})\right)\mbox{d}\theta.
\end{split}
\end{equation}
Now, let us similarly handle the integral $\int_\R \Psi^K(\mathcal{H}^\R h)\mbox{d}x$. We have
\begin{equation}\label{chain_two}
\begin{split}
&\int_\R \Psi^R\left(\mathcal{H}^\R_X h(x)\right)\mbox{d}x
\geq  \int_0^1 \Psi^R \left(\mathcal{H}^\R_X h(x)\right)\mbox{d}x=\int_0^1 \Psi^R(\mathcal{H}^\R_X\Re M_n-\mathcal{H}^\R_X (\mathbf 1_{\R\setminus [0,1]}\Re M_n))\mbox{d}x\\
&=\int_0^1 \Psi^R(\mathcal{H}^\R_X\Re M_n)\mbox{d}x +\int_0^1 \bigg[\Psi^R(\mathcal{H}^\R_X\Re M_n-\mathcal{H}^\R_X (\mathbf 1_{\R\setminus [0,1]}\Re M_n))-\Psi^R(\mathcal{H}^\R_X\Re M_n)\bigg]\mbox{d}x.
\end{split}
\end{equation}
Now, we have $\mathcal{H}^\R_X\Re M_n(x)=\Im M_n(x)=\mathcal{H}^\mathbb{T}_Xf(L^{2n}(x))$, so a calculation similar to that in \eqref{chain} gives
$$
\int_0^1 \Psi^R(\mathcal{H}^\R_X\Re M_n)\mbox{d}x\xrightarrow{n\to\infty}\frac{1}{2\pi}\int_0^{2\pi}\Psi^R\left(\mathcal{H}^\mathbb{T}_Xf(e^{i\theta})\right)\mbox{d}\theta.
$$
To deal with the last integral in \eqref{chain_two} we will first show that $\mathcal{H}^\R_X (\mathbf 1_{\R\setminus [0,1]}\Re M_n)$ converges to $0$ in $L^2$, as $n\to \infty$. To this end, recall that $X$ is finite-dimensional and hence it has the UMD property. Consequently, by \cite[Corollary 5.2.11]{HNVW1}
\begin{equation}\label{term1}
 \int_\R |\mathcal{H}^\R_X (\mathbf 1_{\R\setminus [0,1]}\Re M_n)|^2\mbox{d}x\leq C_X \int_{\R\setminus [0,1]} |\Re M_n|^2\mbox{d}x
\end{equation}
for some constant $C_X$ depending only on $X$. Fix an arbitrary $\eta>0$. As we have already noted above, $\Re M_n$ is bounded by the supremum norm of $f$. Setting $\delta=\eta/(C_X \sup_X ||f||^2)$, we see that
\begin{equation}\label{term2}
  \int_{(-\delta,0)}|\Re M_n(x)|^2\mbox{d}x+\int_{(1,1+\delta)}|\Re M_n(x)|^2\mbox{d}x\leq 2\eta C_X^{-1}.
 \end{equation}
Furthermore, recall that $L$ maps $\R\setminus [0,1]$ onto $(-1,1)$, it is bounded away from $1$ outside $[-\delta,1+\delta]$ and $|L(x)|=O(|x|^{-1})$ as $x\to \pm \infty$.
 Since $F$ is analytic and vanishes at $0$, we conclude that $M_n(x)=F(L^{2n}(x))=O(|x|^{-2n})$ and hence
\begin{equation}\label{term3}
 \lim_{n\to \infty} \int_{\R \setminus [-\delta,1+\delta]} |\Re M_n(x)|^2\mbox{d}x=0.
\end{equation}
Putting \eqref{term1}, \eqref{term2} and \eqref{term3} together, we see that if $n$ is sufficiently large, then $\int_\R |\mathcal{H}^\R_X (\mathbf 1_{\R\setminus [0,1]}\Re M_n)|^2\mbox{d}x\leq 3\eta$ and the aforementioned convergence in $L^2$ holds. In particular, passing to a subsequence if necessary, we see that $\mathcal{H}^\R_X (\mathbf 1_{\R\setminus [0,1]}\Re M_n) \to 0$ almost everywhere. However, as we have already mentioned above, the function $\Psi^R$ is uniformly continuous, so the expression in the square brackets in the last term in \eqref{chain_two} converges to zero almost everywhere. In addition, this expression is bounded in absolute value by $\sup \Psi^R$. Consequently, by Lebesgue's dominated convergence theorem, the last integral in \eqref{chain_two} converges to $0$ as $n\to \infty$. Putting all the above facts together, we see that if $n$ is sufficiently large, then
$$
\int_\R \Psi^R\left(\mathcal{H}^\R_X h(x)\right)\mbox{d}x\geq (1-\e)\cdot \frac{1}{2\pi}\int_0^{2\pi}\Psi^R\left(\mathcal{H}^\mathbb{T}_Xf(e^{i\theta})\right)\mbox{d}\theta.
$$
Combining this with \eqref{lower} and \eqref{chain}, we obtain that for $n$ large enough we have
$$
\int_\R \Psi(\mathcal{H}^\R_X h(x))\mbox{d}x\geq \int_\R \Psi^R\left(\mathcal{H}^\R_X h(x)\right)\mbox{d}x\geq (1-\e)(|\mathcal{H}_X^{\mathbb{T},0}|_{\Phi,\Psi}-\e)\int_\R \Phi(h)\mbox{d}x.
$$
Since $h$ is a step function and $\e$ was arbitrary, the claim follows.
\end{proof}

\begin{remark}
 Note that if $\Psi(0)\neq 0$ then Theorem \ref{thm:HXRleqHXT,0} and \ref{thm:HXTleqHXR} do not make any sense. Indeed, if this is the case, then there exists $\eps>0$ and $R$ such that $\Psi(x)\geq \eps$ for any $x\in X$ with $\|x\|\leq R$. Since for any step function $f:\mathbb R \to X$ the function $\mathcal H^{\mathbb R}_X f$ is in $L^2(\mathbb R;X)$, the set $\{\|\mathcal H^{\mathbb R}_X f\|\leq R\}\subset \mathbb R$ is of infinite measure, so
 \[
  \int_{\mathbb R} \Psi(\mathcal H^{\mathbb R}_X f(s))\ud s \geq \int_{\mathbb R} \mathbf 1_{\|\mathcal H^{\mathbb R}_X f\|\leq R}(s) \eps \ud s = \infty,
 \]
so $|\mathcal H^{\mathbb T}_X|_{\Phi,\Psi} \geq  |\mathcal H^{\mathbb T,0}_X|_{\Phi,\Psi} = |\mathcal H^{\mathbb R}_X|_{\Phi,\Psi} =\infty$.
\end{remark}

\begin{rem}\label{remark_2}
 The finiteness of $|\mathcal{H}_X^{\mathbb{T},0}|_{\Phi,\Psi}$ implies the existence of  a plurisubharmonic function $U_{\Phi,\Psi}:X+iX\to \R$ such that $U_{\Phi,\Psi}(0)\geq 0$. Hence, modifying the proof of Theorem \ref{thm:orthmartPhiPsi}, we see that the inequality \eqref{eq:HTnormboudfororcontmart} holds, with $|\mathcal{H}_X^{\mathbb{T}}|_{\Phi,\Psi}$ replaced with $|\mathcal{H}_X^{\mathbb{T},0}|_{\Phi,\Psi}$, if the dominating martingale $M$ is additionally assumed to start from $0$.
\end{rem}

\begin{theorem}\label{thm:symmPhiHT0=HT}
Let $\Phi, \Psi:X\to \mathbb R_+$ be continuous such that $\Phi$ is symmetric (i.e., $\Phi(x)=\Phi(-x)$ for all $x\in X$) and $\Psi$ is convex. Then $|\mathcal{H}^{\mathbb{T},0}_X|_{\Phi,\Psi}=|\mathcal{H}^\mathbb{T}_X|_{\Phi,\Psi}$.
\end{theorem}
\begin{proof}
It suffices to show the estimate $|\mathcal{H}^{\mathbb{T},0}_X|_{\Phi,\Psi}\geq |\mathcal{H}^\mathbb{T}_X|_{\Phi,\Psi}$. Fix $\e>0$. By the definition of $|\mathcal{H}^\mathbb{T}_X|_{\Phi,\Psi}$, there is a step function $f:\mathbb{T}\to X$ such that
\begin{equation}\label{boundH}
 \int_\mathbb{T} \Psi(\mathcal{H}^\mathbb{T}_Xf)\mbox{d}x>(|\mathcal{H}^\mathbb{T}_X|_{\Phi,\Psi}-\e)\int_\mathbb{T} \Phi(f)\mbox{d}x.
\end{equation}
Let $F=F_1+iF_2$ be the analytic extension of $f+i\mathcal{H}^\mathbb{T}_Xf:\mathbb{T}\to X+iX$ to the unit disc and suppose that
 $B=(B^1,B^2)$ is the planar Brownian motion started at $0$ and stopped upon hitting $\mathbb{T}$. Let $\tau=\inf\{t\geq 0:|B_t|=1\}$ be the lifetime of $B$. The processes $M_t=F_1(B_t)$, $N_t=F_2(B_t)$ are orthogonal martingales such that $N$ is weakly differentially subordinate to $M$. By Fatou's lemma and Lebesgue's monotone convergence theorem (observe that $f$, being a step function, is bounded) we see that if $t$ is sufficiently large, then
$$ \E \Psi(N_t)>(|\mathcal{H}^\mathbb{T}_X|_{\Phi,\Psi}-\e)\E \Phi(M_t).$$
If the expectation of $M$ is zero, then by {Remark \ref{remark_2}} we know that
$$ \E \Psi(N_t)\leq |\mathcal{H}^{\mathbb{T},0}_X|_{\Phi,\Psi}\E \Phi(M_t)$$
and hence we obtain that
\begin{equation}\label{desired_bound}
|\mathcal{H}^{\mathbb{T},0}_X|_{\Phi,\Psi}\geq |\mathcal{H}^\mathbb{T}_X|_{\Phi,\Psi}-\e.
\end{equation}
We will show that this is also true if the expectation $\mathrm{x}=\mathbb{E}M_t$ does not vanish. To this end, consider another Brownian motion $W=(W^1,W^2)$ in $\R^2$ started at $0$ and stopped upon reaching the boundary of the strip $S=\{(x,y):|x|\leq 1\}$. Let $\sigma=\inf\{t:|W^1_t|=1\}$ denote its lifetime. We may assume that $W$ is constructed on the same probability space as $B$ and that both processes are independent. We splice these processes as follows: set
$$ \widetilde{M}_s=\begin{cases}
\mathrm{x}W^1_s & \mbox{if }s\leq \sigma,\\
\operatorname*{sgn}(W_\sigma^1)M_{s-\sigma} & \mbox{if }s>\sigma
\end{cases}$$
and
$$ \widetilde{N}_s=\begin{cases}
\mathrm{x}W^2_s & \mbox{if }s\leq \sigma,\\
\mathrm{x}W^2_\sigma+N_{s-\sigma} & \mbox{if }s>\sigma.
\end{cases}$$
In other words, the pair $(\widetilde{M},\widetilde{N})$ behaves like a Brownian motion evolving in the strip $S\mathrm{x}$ until its first coordinate reaches $\mathrm{x}$ or $-\mathrm{x}$, and then it starts behaving like the pair $(M,\widetilde{N}_\sigma+N)$ or $(-M,\widetilde{N}_\sigma+N)$, depending on which the side of the boundary of $S\mathrm{x}$ the process $\widetilde{M}$ reaches. Note that $\widetilde{M}$, $\widetilde{N}$ are orthogonal martingales such that $\widetilde{N}$ is weakly differentially subordinate to $\widetilde{M}$ and $\widetilde{M}_0=0$. Consequently, by Remark \ref{remark_2} for any $t$,
\begin{equation}\label{4.20}
 \E \Psi(\widetilde{N}_t)\leq |\mathcal{H}^{\mathbb{T},0}_X|_{\Phi,\Psi}\E \Phi(\widetilde{M}_t).
\end{equation}
Now,
$$ \E \Psi(\widetilde{N}_t)\geq \E \Psi(\widetilde{N}_t)1_{\{t\geq \sigma\}}=\E \Psi(\mathrm{x}W^2_\sigma+N_{t-\sigma})1_{\{t\geq \sigma\}}.$$
However, $W$ and $B$ are independent, and the random variable $\mathrm{x}W^2_\sigma$ is symmetric. Therefore, using the fact that $\Psi$ is convex, we see that
$$ \E \Psi(\widetilde{N}_t)\geq \E \Psi(N_{t-\sigma})1_{\{t\geq \sigma\}}.$$
Furthermore, using the symmetry of $\Phi$, we have
$$ \E \Phi(\widetilde{M}_t)1_{\{t\geq \sigma\}}= \E \Phi(\operatorname*{sgn}(W_\sigma^1)M_{t-\sigma})1_{\{t\geq \sigma\}}=\E \Phi(M_{t-\sigma})1_{\{t\geq \sigma\}}.$$
As previously, combining \eqref{boundH} with Fatou's lemma and Lebesgue's dominated convergence theorem, if $t$ is sufficiently large, then
$$ \E \Psi(N_{t-\sigma})1_{\{t\geq \sigma\}}>(|\mathcal{H}^\mathbb{T}_X|_{\Phi,\Psi}-\e)\E \Phi(M_{t-\sigma})1_{\{t\geq \sigma\}}$$
and hence also
$$ \E \Psi(\widetilde{N}_t)>(|\mathcal{H}^\mathbb{T}_X|_{\Phi,\Psi}-\e)\E \Phi(\widetilde{M}_{t})1_{\{t\geq \sigma\}}.$$
But $\lim_{t\to \infty}\E \Phi(\widetilde{M}_{t})1_{\{t< \sigma\}}=0$, by Lebesgue's dominated convergence theorem (we have $1_{\{t<\sigma\}}\to 0$ and the norm of $\widetilde{M}_t$ is bounded by $\|x\|$ for $t\in [0,\sigma]$). Therefore, the preceding estimate gives
$$ \E \Psi(\widetilde{N}_t)>(|\mathcal{H}^\mathbb{T}_X|_{\Phi,\Psi}-\e)\E \Phi(\widetilde{M}_{t})$$
if $t$ is sufficiently big. By \eqref{4.20}, this gives \eqref{desired_bound} and completes the proof of the theorem, since $\e$ was arbitrary.
\end{proof}

\begin{remark}
 Assume that $|\mathcal H^{\mathbb T}_X|_{\Phi, \Psi} = |\mathcal H^{\mathbb T,0}_{X}|_{\Phi, \Psi}$ (this holds true under some additional assumptions on $\Phi$ and $\Psi$, see Theorem \ref{thm:symmPhiHT0=HT}). Then the plurisubharmonic function $U_{\Phi,\Psi}$ considered in Remark \ref{remark_2} coincides with the one considered in Theorem \ref{thm:existenceofU^H_Phi,Psi}, and hence we automatically have that $U_{\Phi, \Psi}(x)\geq 0$ for all $x\in X$.
\end{remark}

\begin{proof}[Proof of Theorem \ref{thm:HTnormthesameforTSZT,0}]
 The theorem follows from Proposition \ref{prop:ineqforHTonRdisandT}, Theorem \ref{thm:HXRleqHXT,0}, \ref{thm:HXTleqHXR}, \ref{thm:symmPhiHT0=HT}, and the fact that $|\mathcal H_X^{\mathbb T, 0}|_{\Phi,\Psi}\leq |\mathcal H_X^{\mathbb T}|_{\Phi,\Psi}$.
\end{proof}

\begin{remark}
Notice that Theorem \ref{thm:HTnormthesameforTSZT,0} can not be applied to more general norms. For example, if $X$ is a UMD Banach space, $1<q<p<\infty$, then
$$
\|\mathcal H^{\mathbb T}_X\|_{\mathcal L(L^p(\mathbb T;X), L^q(\mathbb T; X))}<\infty \;\;
\text{and}\;\;
\|\mathcal H^{\mathbb R}_X\|_{\mathcal L(L^p(\mathbb R;X), L^q(\mathbb R; X))}=\infty.
$$
\end{remark}

\subsection{Decoupling constants}\label{sebsec:decconstants}

We turn our attention to the next important application. We need some additional notation. Consider the probability space $([0,1],\mathcal{B}(0,1),|\cdot|)$, equipped with the dyadic filtration $(\F_n)_{n\geq 0}$ (i.e., generated by the Haar system $(h_n)_{n=0}^\infty$, {see e.g.\ \cite{HNVW1}}). A martingale $f$ adapted to this filtration is called a \emph{Paley-Walsh martingale}.

\begin{definition}
 Let $X$ be a Banach space and let $1<p<\infty$ be a fixed parameter. Then we define $\beta_{p,X}^{\Delta ,+}$ and $\beta_{p,X}^{\Delta ,-}$ to be the smallest $\beta^+$ and $\beta^-$ such that
  \[
  \frac{1}{(\beta^-)^p}\mathbb E\Bigl\|\sum_{n=0}^{\infty}df_n\Bigr\|^p \leq \mathbb E\Bigl\|\sum_{n=0}^{\infty}r_ndf_n\Bigr\|^p  \leq (\beta^+)^p \mathbb E\Bigl\|\sum_{n=0}^{\infty}df_n\Bigr\|^p
  \]
for any finite Paley-Walsh martingale $(f_n)_{n\geq 0}$ and any independent Rademacher sequence $(r_n)_{n\geq 0}$. Furthermore, we define  $\beta_{p,X}^{\gamma ,+}$ and $\beta_{p,X}^{\gamma ,-}$ to be the least possible values of $\beta^+$ and $\beta^-$ for which
 \[
  \frac{1}{(\beta^-)^p}\mathbb E\Bigl\|\int_0^\infty \phi \ud W\Bigr\|^p \leq \mathbb E\Bigl\|\int_0^\infty \phi \ud \widetilde W\Bigr\|^p  \leq (\beta^+)^p \mathbb E\Bigl\|\int_0^\infty \phi \ud W\Bigr\|^p,
  \]
where $W$ is a standard Brownian motion, $\phi:\mathbb R_+ \times \Omega \to X$ is an elementary progressive process, and $\widetilde{W}$ is another Brownian motion independent of $\phi$ and $W$.
\end{definition}

Decoupling constants appear naturally while working with UMD Banach spaces (see e.g.\ \cite{Gar85,CV,HNVW1,Ver07,MC,Geiss99,CG}).
The following result, a natural corollary of Theorem \ref{thm:orthmartPhiPsi} for $\Phi(x)=\Psi(x)=\|x\|^p$, exhibits the direct connection between decoupling constants and $\hbar_{p,X}:= \|\mathcal H_X^{\mathbb T}\|_{\mathcal L(L^p(\mathbb T;X))}$ (see Corollary~\ref{thm:existenceofU^H_p,X}).

\begin{corollary}
 Let $X$ be a Banach space and let $1<p< \infty$ be a fixed parameter. Then we have
 \begin{equation}\label{eq:gammabeta+beta-estforHilbTrans}
  \hbar_{p,X} \geq \max\{\beta_{p, X}^{\gamma,+}, \beta_{p,X}^{\gamma,-}\}
 \end{equation}
and hence
  \begin{equation}\label{eq:deltabeta+beta-estforHilbTrans}
  \hbar_{p,X} \geq C\max\{\beta_{p, X}^{\Delta,+}, \beta_{p,X}^{\Delta,-}\}.
 \end{equation}
 Here $C=\mathbb E |\gamma| \mathbb E\sqrt{\tau}$, where $\gamma$ is a standard normal random variable and $\tau= \inf\{t\geq 0: |W_t|=1\}$ for a standard Brownian motion $W$.
\end{corollary}

Note that $\mathbb E{\tau} \leq (\mathbb E \sqrt{\tau})^{\frac 23}(\mathbb E \tau^2)^{\frac 13}$ by H\"older's inequality, so $C$ in \eqref{eq:deltabeta+beta-estforHilbTrans} is bounded from below by $\frac{(\mathbb E \tau)^{\frac{3}{2}}}{(\mathbb E \tau^2)^{\frac 12}}\mathbb E |\gamma| = \frac{\sqrt{6}}{\sqrt {5\pi}}\approx 0.618$ (since $\mathbb E \tau = 1$ and $\mathbb E \tau^2 = \frac{5}{3}$).

\begin{proof}
 The inequality \eqref{eq:gammabeta+beta-estforHilbTrans} follows directly from the definition of $\beta_{p, X}^{\gamma,+}$ and $\beta_{p,X}^{\gamma,-}$. Indeed, for any Brownian motion $W$, elementary progressive process $\phi$, and a Brownian motion $\widetilde W$ independent of $\phi$ and $W$ we have, for any $x^*\in X^*$,
 \[
 \Bigl [\Bigl\langle \int_0^{\cdot} \phi \ud W, x^*\Bigr\rangle\Bigr]_t = \Bigl[\int_0^{\cdot}\langle  \phi, x^*\rangle \ud W\Bigr]_t = \int_0^t |\langle \phi(s), x^*\rangle|^2\ud s,
 \]
  \[
 \Bigl [\Bigl\langle \int_0^{\cdot} \phi \ud \widetilde W, x^*\Bigr\rangle\Bigr]_t = \Bigl[\int_0^{\cdot}\langle  \phi, x^*\rangle \ud \widetilde W\Bigr]_t = \int_0^t |\langle \phi(s), x^*\rangle|^2\ud s,
 \]
so $\int \phi \ud W \stackrel{w}\ll  \int \phi \ud \widetilde  W \stackrel{w}\ll \int \phi \ud W$. Moreover, by \cite[Lemma 17.10]{Kal},
\begin{align*}
  \Bigl [\Bigl\langle \int_0^{\cdot} \phi \ud W, x^*\Bigr\rangle,\Bigl\langle \int_0^{\cdot} \phi \ud \widetilde W, x^*\Bigr\rangle\Bigr]_t &=  \Bigl [\int_0^{\cdot}\langle  \phi , x^*\rangle\ud W, \int_0^{\cdot} \langle\phi, x^* \rangle\ud \widetilde W \Bigr]_t=\int_0^{t}|\langle  \phi(s) , x^*\rangle|^2 \ud [W,\widetilde W]_s=0,
\end{align*}
where the latter holds since $W$ and $\widetilde W$ are independent. Therefore $\int \phi \ud W$ and $\int \phi \ud\widetilde W$ are orthogonal local martingales satisfying the differential subordination (``in both directions''), so by Theorem \ref{thm:orthmartPhiPsi},
 \[
  \frac{1}{(\hbar_{p,X})^p}\mathbb E\Bigl\|\int_{\mathbb R_+} \phi \ud W\Bigr\|^p \leq \mathbb E\Bigl\|\int_{\mathbb R_+} \phi \ud \widetilde W\Bigr\|^p  \leq (\hbar_{p,X})^p \mathbb E\Bigl\|\int_{\mathbb R_+} \phi \ud W\Bigr\|^p.
  \]
Let us now turn to the second part. First notice that $\beta_{p, X}^{\gamma, +}\geq C \beta_{p, X}^{\Delta, +}$ (see \cite[(2.5)]{Ver07} and the discussion thereafter), so $\hbar_{p,X} \geq \beta_{p, X}^{\gamma, +}\geq C \beta_{p, X}^{\Delta, +}$. On the other hand, $X$ can be assumed UMD (and hence reflexive), so by the discussion above we have $\hbar_{p',X^*}\geq C \beta_{p', X^*}^{\Delta, +}$. But $\hbar_{p',X^*} = \hbar_{p,X}$ (since $(\mathcal H^{\mathbb T}_X)^* = \mathcal H^{\mathbb T}_{X^*}$), and $\beta_{p', X^*}^{\Delta, +} \geq \beta_{p, X}^{\Delta, -}$ analogously to \cite[Theorem 1]{Gar90}, so $\hbar_{p,X} \geq C \beta_{p, X}^{\Delta, -}$.
\end{proof}

\begin{remark}
 Notice that \eqref{eq:gammabeta+beta-estforHilbTrans} together with \cite[Theorem 3]{Gar85} yields the related estimate $
 \max\{\beta_{p, X}^{\gamma,+}, \beta_{p,X}^{\gamma,-}\} \leq \hbar_{p,X} \leq \beta_{p, X}^{\gamma,+}\beta_{p,X}^{\gamma,-}.
 $
\end{remark}

\begin{remark}
Let $X$ be a UMD Banach function space. Then inequality \eqref{eq:deltabeta+beta-estforHilbTrans} together with \cite{KLW18} provide the lower bound for $\hbar_{p,X}$ in terms of $\beta_{p,X}$ of the same order as \eqref{eq:sqrtbetaleqhleqbeta^2}. Indeed, by \cite{KLW18} thanks to Banach function space techniques one can show that
\[
\beta_{p, X} \lesssim_{p} q(c_{q,X} \beta_{p, X}^{\Delta, +})^2,
\]
where $q$ is the cotype of $X$ and $c_{q, X}$ is the corresponding cotype constant. Therefore by applying \eqref{eq:deltabeta+beta-estforHilbTrans} we get the following square root dependence:
\[
\sqrt {\beta_{p, X}}\lesssim_{p} \sqrt q c_{q, X} \hbar_{p,X}.
\]
\end{remark}

\subsection{Necessity of the UMD property}\label{sebsec:NecofUMD}

Our next result answers a very natural question about the link of the number $|\mathcal{H}_X^{\mathbb{T},0}|_{\Phi,\Psi}$ to the UMD property.
\begin{theorem}\label{UMD}
Let $\Phi$, $\Psi:X\to \R_+$ be continuous convex functions such that $\Psi(0)=0$. Assume in addition that there is a positive number $C$ such that the sets $\{x\in X:\Psi(x)<C\}$ and $\Phi(B(0,C))$ are bounded. If $|\mathcal{H}_X^{\mathbb{T},0}|_{\Phi,\Psi}<\infty$, then $X$ is UMD.
\end{theorem}
\begin{rem}\label{remark_1}
It is easy to see that the assumption $\Psi(0)=0$ combined with the boundedness of $\{\Psi<C\}$ enforces the function $\Psi$ to explode ``uniformly'' in the whole space. That is, if $B(0,R)$ is the ball containing $\{\Psi<C\}$, then the convexity of $\Psi$ implies $\Psi(x)\geq C\|x\|/R$ for all $x\notin B(0,R)$. Some condition of this type is necessary, as the following simple example indicates. Take $X=\ell_\infty$ and set $\Phi(x)=|x_1|^2=\Psi(x)$. Then $|\mathcal{H}_X^{\mathbb{T},0}|_{\Phi,\Psi}=1<\infty$, while $X$ is not UMD. The reason is that the function $\Psi$ controls only the subspace generated by the first coordinate.
\end{rem}

\begin{remark}
 Note that $X$ being UMD does not imply $|\mathcal{H}_X^{\mathbb{T},0}|_{\Phi,\Psi}<\infty$. Indeed, if $\Phi$ and $\Psi$ are of different homogeneity (i.e.\ $\Phi(ax) = a^{\alpha}\Phi(x)$, $\Phi(ax) = a^{\beta}\Phi(x)$ for any $x\in X$, $a\geq 0$, and for some fixed positive $\alpha \neq \beta$), then for any nonzero step function $f:\mathbb T \to X$ such that $\int_{\mathbb T} f(s)\ud s=0$ and for any $a\geq 0$ we have that
 \begin{align*}
  \int_{\mathbb T}\Psi(\mathcal H^{\mathbb T}_X f(s))\ud s &= \frac{1}{a^{\beta}}\int_{\mathbb T}\Psi(\mathcal H^{\mathbb T}_X (af)(s))\ud s \leq \frac{1}{a^{\beta}}|\mathcal H^{\mathbb T,0}_X|_{\Phi,\Psi}\int_{\mathbb T}\Phi(af(s))\ud s=a^{\alpha-\beta}|\mathcal H^{\mathbb T,0}_X|_{\Phi,\Psi}\int_{\mathbb T}\Phi(f(s))\ud s,
 \end{align*}
 so $a^{\alpha-\beta}|\mathcal H^{\mathbb T,0}_X|_{\Phi,\Psi}\leq |\mathcal H^{\mathbb T,0}_X|_{\Phi,\Psi}$ for any $a> 0$, and since $\alpha \neq \beta$, $|\mathcal H^{\mathbb T,0}_X|_{\Phi,\Psi} = \infty$. The classical examples of such $\Phi$ and $\Psi$ are $\Phi(x) = \|x\|^p$, $\Psi(x)=\|x\|^q$, $x\in X$ for different $p$ and $q$.
\end{remark}

The proof of Theorem \ref{UMD} will exploit the following four lemmas. In what follows, $N^*=\sup_{t\geq 0}\|N_t\|$ is the maximal function of $N$.

\begin{lemma}\label{lemma1}
Under the assumptions of Theorem \ref{UMD}, there exists a constant $c_1$ depending on $\Phi$, $\Psi$ and $X$, such that if $M$, $N$ are orthogonal martingales such that $N$ is weakly differentially subordinate to $M$, $M_0=0$ and $\|M\|_\infty\leq c_1$, then $\mathbb{P}(N^*\geq 1)<1$.
\end{lemma}
\begin{proof} Let $R$ be as in Remark \ref{remark_1} and suppose that $\Phi(B(0,C))\subseteq [-R',R']$. Then for any $\lambda\geq 1$ we have, in the light of Remark \ref{remark_2},
\begin{align*}
 \mathbb{P}(\|N_t\|\geq 1)&=\mathbb{P}(R\lambda \|N_t\|\geq R\lambda)\leq \frac{\mathbb{E}\Psi(R\lambda N_t)}{C\lambda}\leq \frac{|\mathcal{H}_X^{\mathbb{T},0}|_{\Phi,\Psi}\E \Phi(R\lambda M_t)}{C\lambda}.
 \end{align*}
It suffices to take $\lambda=\tfrac{2R'|\mathcal{H}_X^{\mathbb{T},0}|_{\Phi,\Psi}}{C}$ and $c_1=C/(R\lambda)$.
\end{proof}

\begin{lemma}\label{lemma2}
Suppose that the assumptions of Theorem \ref{UMD} are satisfied. Let $M$ and $N$ be \emph{continuous-path} orthogonal martingales such that $N$ is weakly differentially subordinate to $M$, $M_0=0$ and $\mathbb{P}(N^*> 1)=1$. Then there exist continuous-path martingales $\widetilde {M}$, $\widetilde {N}$ such that $\widetilde {N}$ is weakly differentially subordinate to $\widetilde {M}$, $\widetilde {M}_0=0$, $\mathbb{P}(\widetilde {N}^*> 1)\geq 1/2$ and $\|\widetilde {M}\|_\infty\leq 2\|M\|_1.$
\end{lemma}
\begin{proof}
Define $\tau=\inf\{t\geq 0:\|M_t\|\geq 2\|M\|_1\}$ (as usual, $\inf\varnothing=+\infty$) and put $\widetilde {M}=M^\tau$, $\widetilde {N}=N^\tau$. Since $M$ has continuous paths and starts from $0$, we have $\|\widetilde {M}\|_\infty\leq 2\|M\|_1$. Furthermore, $\mathbb{P}(\widetilde {N}^*> 1)\geq \mathbb{P}(\widetilde {N}=N)\geq 1/2$,
since
$$
\mathbb{P}(\widetilde {N}\neq N)=\mathbb{P}(\tau<\infty)=\mathbb{P}(M^*\geq 2\|M\|_1)\leq 1/2
$$
by \cite[Theorem 1.3.8(i)]{KS}.
\end{proof}

\begin{lemma}\label{lemma3}
Suppose that the assumptions of Theorem \ref{UMD} are satisfied. Then there exists a constant $c>0$ such that if $M$, $N$ are continuous-path orthogonal martingales such that $N$ is weakly differentially subordinated to $M$, $M_0=0$ and $N^*> 1$ almost surely, then $\|M\|_1\geq c$.
\end{lemma}
\begin{proof} Let $c_1$ be the number guaranteed by Lemma \ref{lemma1}.
Suppose that such a $c$ does not exist. Then for any positive integer $j$ there exist a pair $(M^j,N^j)$ of orthogonal martingales such that $N^j$ is weakly differentially subordinate to $M^j$, $M^j_0=0$, $\mathbb{P}((N^j)^*> 2)=1$ and $\|M^j\|_1\leq 2^{-j-1}c_1$. By Lemma \ref{lemma2}, for each $j$ there is a pair $(\widetilde {M}^j,\widetilde {N}^j)$ of orthogonal, weakly differentially subordinate martingales satisfying $\widetilde {M}^j_0=0$, $\mathbb{P}((\widetilde {N}^j)^*> 2)\geq 1/2$ and $\|\widetilde {M}^j\|_\infty\leq 2^{-j}c_1$. We may assume that the underlying probability space is the same for all pairs and that all the pairs are independent. For each $j$ there is a positive number $t_j$ such that the event
$$ A_j=\{\|\widetilde {N}^j_t\|>2\mbox{ for some }t\leq t_j\}$$
has probability greater than $1/3$. Set $t_0=0$ and consider the martingale pair $(M,N)$ defined as follows: if $t\in [t_0+t_1+\ldots+t_n,t_0+t_1+\ldots+t_{n+1})$ for some $n$, then
\begin{equation}\label{splice}
 M_t=\widetilde {M}_{t_1}^1+\widetilde {M}_{t_2}^2+\ldots+\widetilde {M}_{t_n}^n+\widetilde {M}_{t-t_1-t_2-\ldots-t_n}^{n+1},
\end{equation}
and analogously for $N$. Then $M$ and $N$ are orthogonal, $N$ is weakly differentially subordinate to $M$, $M_0=0$ and
$$ \|M\|_\infty\leq \sum_{j=1}^\infty \|\widetilde {M}^j\|_\infty\leq \sum_{j=1}^\infty 2^{-j}c_1=c_1.$$
Furthermore, by Borel-Cantelli lemma,
$$ \mathbb{P}(N^*\geq 1)\geq \mathbb{P}\left(\limsup_{j\to\infty} A_j\right)=1,$$
since the events $A_j$ are independent and $\sum_{j=1}^\infty \mathbb{P}(A_j)=\infty$. Therefore we have that $\|M\|_{\infty} \leq c_1$, $\mathbb P(N^*\geq 1)=1$, $N \stackrel{w}\ll M$, and $M$ and $N$ are orthogonal, which contradicts the assertion of Lemma \ref{lemma1}.
\end{proof}

\begin{lemma}
Suppose that the assumptions of Theorem \ref{UMD} are satisfied. Then there exists a positive constant $C$ such that if $M$, $N$ are continuous-path orthogonal martingales such that $N$ is weakly differentially subordinate to $M$ and $M_0=0$, then
\begin{equation}\label{weak_type}
 \mathbb{P}(N^*>1)\leq C\|M\|_1.
\end{equation}
\end{lemma}
\begin{proof}
Let $c$ be the constant guaranteed by the previous lemma.
Suppose that the assertion is not true. Then for any positive integer $j$ there is a martingale pair $(M^j,N^j)$ satisfying the usual structural properties such that
\begin{equation}\label{reverse}
 \mathbb{P}((N^j)^*>2)> 2^{j+1}c^{-1}\|M^j\|_1.
\end{equation}
We splice these martingale pairs into one pair $(M,N)$ as previously, however, this time we allow pairs to appear several times. More precisely, denote $a_j=\mathbb{P}((N^*)^j> 2)$. Consider $\lceil 1/a_1\rceil$ copies of $(M^1,N^1)$, $\lceil  1/a_2\rceil$ copies of $(M^2,N^2)$, and so on (all the pairs are assumed to be independent). Let $t_j$ be positive numbers such that the events $A_j=\{\|N^j_t\|>2\mbox{ for some }t\leq t_j\}$ have probability greater than $a_j/2$. Splice the aforementioned independent martingale pairs (with multiplicities) into one pair $(M,N)$ using a formula analogous to \eqref{splice}. Then, by \eqref{reverse},
$$
\|M\|_1\leq \sum \|M^j\|_1\leq \sum_{j=1}^\infty \left\lceil \frac{1}{a_j}\right\rceil \|M^j\|_1\leq \sum_{j=1}^\infty \frac{2}{a_j}\cdot a_j c2^{-j-1}=c
$$
and, again by Borel-Cantelli lemma, $ \mathbb{P}(N^*>1)=1$. Here we use the independence of the events $A_j$ and
$$
\sum \mathbb{P}(A_j)\geq \sum_{j=1}^\infty \frac{1}{a_j}\cdot \frac{a_j}{2}=\infty.
$$
This contradicts Lemma \ref{lemma3}.
\end{proof}

\begin{proof}[Proof of Theorem \ref{UMD}]
We will prove that theorem using the well-known extrapolation technique (\mbox{good-$\lambda$} inequalities) of Burkholder \cite{Bur73}.

{\em Step 1.} First we show that for any fixed $0<\delta<1$ and $\beta>1$ there exists $\eps>0$ depending only on $\delta$, $\beta$, and $X$ such that for any orthogonal \emph{continuous-path} martingales $M,N:\mathbb R_+ \times \Omega \to X$ with $M_0=N_0=0$ and $N \stackrel{w}\ll M$,
\begin{equation}\label{eqLgoodlambdaorthcontmart}
  \mathbb P(N^*>\beta \lambda, M^* \leq \delta\lambda) \leq \eps \mathbb P(N^* >\lambda)
\end{equation}
for any $\lambda >0$. Without loss of generality assume that both martingales take their values in a finite-dimensional subspace of $X$. Define three stopping times
\begin{equation}\label{eq:goodlambdastoptimes}
 \begin{split}
   \mu&:= \inf\{t\geq 0: \|N_t\| > \lambda\},\\
 \nu&:= \inf\{t\geq 0: \|M_t\| > \delta\lambda\},\\
 \sigma&:= \inf\{t\geq 0: \|N_t\| > \beta\lambda\}.
 \end{split}
\end{equation}
All the stopping times are predictable since $M$ and $N$ are continuous. Therefore, the equation $U(t)=\mathbf 1_{[\mu, \nu\wedge \sigma ]}(t)$ defines a predictable process {(where $U=0$ on $\mathbb R_+$ if $\nu\wedge \sigma < \mu$)}, which in turn gives rise to  the martingales
{
\begin{equation}\label{eq:goodlambdatildeMN}
 \begin{split}
   \widetilde M &:= \int U\ud M= M^{\nu\wedge\delta} - M^{\mu\wedge \nu\wedge\delta},\\
 \widetilde N &:= \int U \ud N= N^{\nu\wedge\delta} - N^{\mu \wedge \nu\wedge\delta}.
 \end{split}
\end{equation}
}
Notice that {by \eqref{eq:goodlambdastoptimes} and \eqref{eq:goodlambdatildeMN}}, $\widetilde M^* \leq 2\delta\lambda$ on $\{\mu <\infty\}$ and $\widetilde M^* = 0$ on $\{\mu=\infty\}$, so
\begin{equation}\label{eq:goodlambdawidetildeMismajbywidetildeN}
 \|\widetilde M\|_1 \leq  2\delta \lambda \mathbb P(N^* > \lambda).
\end{equation}
Since $\widetilde N \stackrel{w}\ll \widetilde M$, $\widetilde M_0=\widetilde N_0=0$ and $\widetilde M$ and $\widetilde N$ are orthogonal,
\[
  \mathbb P(N^*>\beta \lambda, M^* \leq \delta\lambda) \leq \mathbb P(\widetilde N^* > (\beta-1)\lambda) \stackrel{(i)}\leq \frac{C}{(\beta-1)\lambda} \|\widetilde M\|_1 \stackrel{(ii)}\leq \frac{2\delta  C}{(\beta-1)}\mathbb P(N^* > \lambda),
\]
where $(i)$ follows from \eqref{weak_type} with the same constant $C$ depending only on $X$, and $(ii)$ follows from \eqref{eq:goodlambdawidetildeMismajbywidetildeN}. Therefore \eqref{eqLgoodlambdaorthcontmart} holds with $\eps = {2\delta C}/{(\beta-1)}$.

\smallskip

{\em Step 2.}
Now a straightforward integration argument (cf. \cite[Lemma 7.1]{Bur73}), together with Doob's maximal inequality, yield the $L^p$ estimate
$$
\sup_{t\geq 0}\|N_t\|_p\leq \|N^*\|_p\leq C_{p, X}\|M^*\|_p\leq \frac{pC_{p,X}}{p-1}\sup_{t\geq 0}\|M_t\|_p,\qquad 1<p<\infty,
$$
for any pair of continuous, orthogonal, differentially subordinated martingales such that $M_0=0$. Here
\begin{equation}\label{eq:defCpX}
C_{p,X}^p=\frac{\delta^{-p}\beta^p}{1-\beta^p\cdot 2\delta C/{(\beta-1)}},
\end{equation}
which, if we let $\beta=1+p^{-1}$ and $\delta=(10Cp)^{-1}$,
depends only on $p$ and the constant in \eqref{weak_type}. This in turn yields the corresponding $L^p$ inequality for the periodic Hilbert transform for functions of integral $0$. By Theorem \ref{thm:HTnormthesameforTSZT,0} the assumption on the zero-average can be omitted, and hence $X$ is UMD by \cite[Corollary 5.2.11]{HNVW1}.
\end{proof}

Now we will take a closer look at the classical ``LlogL'' estimates of Zygmund \cite{ZygTS02}.
For a Banach space $X$ and a step function $f:\mathbb T \to X$, we define
\[
 \|f\|_{L\log L(\mathbb T;X)} := \int_{\mathbb T} (\|f(s)\|+1)\log(\|f(s)\|+1)\ud s
\]
and denote
{
$$
\hbar_{L\log L,X}=|\mathcal H^{\mathbb T}_X|_{L\log L(\mathbb T;X)\to L^1(\mathbb T;X)} := \sup_{f:\mathbb T \to X \; \text{step}}\frac{\|\mathcal H^{\mathbb T}_Xf\|_{L^1(\mathbb T;X)}}{\|f\|_{L\log L(\mathbb T;X)}}.
$$

\begin{remark}
In the light of Theorem \ref{thm:HTnormthesameforTSZT,0},  we have
 \begin{align*}
  \hbar_{L\log L,X} = |\mathcal H_X^{\mathbb T,0}|_{L\log L(\mathbb T;X)\to L^1(\mathbb T;X)} &= |\mathcal H_X^{\mathbb R}|_{L\log L(\mathbb R;X)\to L^1(\mathbb R;X)}=|\mathcal H_X^{\rm dis}|_{L\log L(\mathbb Z;X)\to L^1(\mathbb Z;X)}
 \end{align*}
for any Banach space $X$.
\end{remark}

We will establish the following statement.

\begin{theorem}\label{cor:LlogLiffUMD}
  Let $X$ be a Banach space. Then $X$ has the UMD property if and only if $\hbar_{L\log L,X}<\infty$.
\end{theorem}

For the proof we will need the following lemma.

\begin{lemma}\label{lem:hbarleq1/p-1}
 Let $X$ be a UMD Banach space. Then there exists a constant $C_X$ depending only on $X$ such that $\hbar_{p,X}\leq C_X\frac{ p}{p-1}$ for all $1<p<2$.
\end{lemma}

\begin{proof}
Let $M, N:\mathbb R_+\times \Omega \to X$ be continuous orthogonal martingales such that $N \stackrel{w}\ll M$ and $N_0 =0$. As we have already seen above,
\[
 \sup_{t\geq 0}(\mathbb E \|N_t\|^p)^{\frac{1}{p}} \leq \frac{p}{p-1}C_{p,X} \sup_{t\geq 0}(\mathbb E \|M_t\|^p)^{\frac{1}{p}},
\]
where $ C_{p,X}\leq 10Cpe(1-e/5)^{-1/p}$ (see \eqref{eq:defCpX} and the discussion following it). Therefore, if $1<p<2$, we may assume that this constant depends only on {$C$ (which essentially depends only on $X$)}. The claim follows from the sharpness part of Theorem \ref{thm:orthmartPhiPsi}.
\end{proof}

\begin{proof}[Proof of Theorem \ref{cor:LlogLiffUMD}]
The inequality $\hbar_{L\log L,X}<\infty$ implies UMD by Theorem \ref{UMD} applied to $\Phi(x) =(\|x\|+1)\log(\|x\|+1)$ and $\Psi(x) = \|x\|$, $x\in X$. The converse holds true by Lemma \ref{lem:hbarleq1/p-1} and Yano's extrapolation argument (see e.g.\ \cite{Yano51,EK05}).
\end{proof}

\subsection{Weak differential subordination of martingales: sharper $L^p$-in\-equ\-a\-li\-ties}\label{subsec:sharperLpforWDS}

As it was noticed in \eqref{eq:WDSbeta^2(beta+1)est}, for a UMD Banach space $X$, any $1<p<\infty$ and any $X$-valued local martingales $M$ and $N$ such that $N\stackrel{w}\ll M$, we have
\[
 \mathbb E \|N_t\|^p \leq c_{p, X}^p \mathbb E \|M_t\|^p,\;\;\; t\geq 0,
\]
with $c_{p, X}\leq \beta_{p, X}^2(\beta_{p, X}+1)$. The purpose of this subsection is to show that the upper bound can be substantially improved.

\begin{theorem}\label{thm:WDSofmartbeta+h}
 Let $X$ be a Banach space, let $1<p<\infty$ and assume that $M,\,N$ are local martingales satisfying $N\stackrel{w}\ll M$. Then
 \begin{equation}\label{eq:thmWDSofmartbeta+h}
   \mathbb E\|N_t\|^p \leq (\beta_{p, X} + \hbar_{p,X})^p \mathbb E \|M_t\|^p\;\;\; \mbox{for any }t\geq 0.
 \end{equation}
\end{theorem}

\begin{remark}\label{rem:wdsbeta(beta+1)est}
 Note that $\hbar_{p,X} \leq \beta_{p, X}^2$ (see \eqref{eq:sqrtbetaleqhleqbeta^2}), so \eqref{eq:thmWDSofmartbeta+h} gives
 \[
   (\mathbb E\|N_t\|^p)^{\frac {1}{p}} \leq \beta_{p, X}(\beta_{p, X} + 1) (\mathbb E \|M_t\|^p)^{\frac {1}{p}}\;\;\; t\geq 0,
 \]
 which is better than \eqref{eq:WDSbeta^2(beta+1)est}.
\end{remark}

For the proof of Theorem \ref{thm:WDSofmartbeta+h} we will need the notion of the {\em Burkholder function}.

\begin{definition}\label{def:concbiconczigzagconcave}
 Let $E$ be a linear space. A function $f:E \to \mathbb R$ is called concave if for any $x, y \in E$ and any $\lambda \in [0,1]$ we have $f(\lambda x+(1-\lambda)y)\geq \lambda f(x)+(1-\lambda)f(y)$.
  A function $f:E\times E \to \mathbb R$ is called zigzag-concave if for each $x, y\in E$ and $\eps \in[-1,1]$ the function $z\mapsto f(x+z, y+\eps z)$ is concave.
\end{definition}

The following theorem can be found in \cite{HNVW1,Y17FourUMD,Burk86}.

\begin{theorem}[Burkholder]\label{thm:Burkholder}
 For a Banach space $X$ the following conditions are equivalent:
 \begin{enumerate}
  \item $X$ is a UMD Banach space;
  \item for each $p\in (1,\infty)$ there exists a constant $\beta$ and a zigzag-concave function $U:X\times X \to \mathbb R$, convex in the second variable, such that
  \begin{equation}\label{eq:ineqonU}
     U(x,y)\geq \|y\|^p - \beta^p\|x\|^p,\;\;\;x,y\in X.
  \end{equation}
 \end{enumerate}
 The smallest admissible $\beta$, for which such $U$ exists, is equal to $\beta_{p, X}$.
\end{theorem}

Any function $U$ as in the above theorem will be called a {\em Burkholder function}.

\begin{remark}\label{rem:UisCinfty}
 Suppose that the Banach space $X$ is finite-dimensional and let $U:X\times X \to \mathbb R$ be a zigzag-concave function. Let $\rho:X\times X \to \mathbb R_+$ be a compactly supported nonnegative function of class $C^{\infty}$. Then the convolution $U_{\rho} := U * \rho:X\times X \to \mathbb R$ is zigzag-concave and of class $C^\infty$ (see e.g.\ \cite{BO12}).
\end{remark}

{
While working with the Burkholder function $U:X\times X \to \mathbb R$ we will use the following notation: for given vectors $x,y\in X$ instead of writing
\[
 \frac{\partial^2 U}{\partial (x,0)^2},\frac{\partial^2 U}{\partial (0,y)^2},\frac{\partial^2 U}{\partial (x,0)\partial (0,y)}
\]
we will write
\[
 \frac{\partial^2 U}{\partial x^2},\frac{\partial^2 U}{\partial y^2},\frac{\partial^2 U}{\partial x\partial y}.
\]
Therefore for the convenience of the reader throughout this subsection we always assume that the first coordinate of any vector in $X\times X$ is $x$ (perhaps with a subscript), while the second coordinate is $y$ (perhaps with a subscript). The same holds for partial derivatives.
}

We also will need the following lemma.

\begin{lemma}\label{lem:twdiffzigzagfnction}
 Let $X$ be a finite-dimensional Banach space, let $F:X\times X\to \mathbb R$ be a zigzag-concave function and let $(x_0,y_0)\in X\times X$ be such that $F$ is twice Fr\'echet differentiable at $(x_0,y_0)$. Let $(x,y)\in X\times X$ be such that $y=x$. Then for each $\lambda\in [-1,1]$,
 {
 \[
  \frac{\partial^2F(x_0, y_0)}{\partial x^2} + 2\lambda \frac{\partial^2F(x_0, y_0)}{\partial x\partial y} + \frac{\partial^2F(x_0, y_0)}{\partial y^2}\leq 0.
 \]
 }
\end{lemma}
\begin{proof}
 Since the function
 {
 \[
  \lambda \mapsto   \frac{\partial^2F(x_0, y_0)}{\partial x^2} + 2\lambda \frac{\partial^2F(x_0, y_0)}{\partial x\partial y} + \frac{\partial^2F(x_0, y_0)}{\partial y^2}
 \]
 }
is linear in $\lambda \in [-1,1]$, it is sufficient to check the cases $\lambda = \pm 1$. To this end notice that
{
\begin{align*}
    \frac{\partial^2F(x_0, y_0)}{\partial x^2} \pm 2 \frac{\partial^2F(x_0, y_0)}{\partial x\partial y} + \frac{\partial^2F(x_0, y_0)}{\partial y^2}= \frac{\partial^2}{\partial t^2}F(x_0 + tx, y_0 \pm tx)\Big|_{t=0} \leq 0,
\end{align*}
}
where the latter follows from Definition \ref{def:concbiconczigzagconcave}.
\end{proof}

\begin{proof}[Proof of Theorem \ref{thm:WDSofmartbeta+h}] We begin with similar reductions as in the proof of Theorem \ref{thm:orthmartPhiPsi}. First, we may assume that $X$ is a finite-dimensional Banach space. Let $d\geq 1$ be the dimension of $X$. Let $M = M^c + M^d$ and $N= N^{c} + N^{d}$ be the Meyer-Yoeurp decompositions (see Subsection \ref{subsec:MYdec}). Then by Proposition \ref{prop:MYdecconttocontpdtopd} $N^c \stackrel{w}\ll M^c$ and $N^d \stackrel{w}\ll M^d$. Let $\tau = (\tau_s)_{s\geq 0}$ be the time-change constructed in Step 1 of the proof of Theorem \ref{thm:orthmartPhiPsi} (see \cite[Section 4]{Y17MartDec}). So,  there exists a $2d$-dimensional standard Brownian motion $W$ on an extended probability space $(\widetilde {\Omega}, \widetilde {\mathcal F}, \widetilde {\mathbb P})$ equipped with an extended filtration $\widetilde {\mathbb F}=(\widetilde {\mathcal F}_t)_{t\geq 0}$, and there exist two progressively measurable processes $\phi, \psi:\mathbb R_+ \times\Omega \to \mathcal L(\mathbb R^{2d},X)$ such that $M^c \circ \tau = \phi \cdot W$ and $N^c\circ \tau=\psi \cdot W$. Let us redefine $M:= M\circ \tau$ and $N:= N\circ \tau$ (hence $M^c := M^c \circ \tau$, $M^d := M^d \circ \tau$, $N^c := N^c\circ \tau$, and $N^d := N^d \circ \tau$, see \cite[Subsection 2.6]{Y17MartDec}). Without loss of generality we may further assume that $M$ and $N$ terminate after some deterministic time: $M_t = M_{t\wedge T}$ and $N_{t} = N_{t\wedge T}$ for some fixed parameter $T\geq 0$. Analogously to Proposition \ref{prop:existofaskewsymmopvalpr} there exists a progressively measurable $A:\mathbb R_+ \times \Omega \to \mathcal L(\mathbb R^{2d})$ which satisfies $\|A\|\leq 1$ on $\mathbb R_+ \times \Omega$ and $\psi = \phi A$.
Let us define $A^{\rm sym} := \frac{A + A^T}{2}$, $A^{\rm asym}:= \frac{A-A^T}{2}$. If we set
\[
 N^{\rm sym} := N^d + (\phi A^{\rm sym}) \cdot W, \qquad
 N^{\rm asym}:= (\phi A^{\rm asym})\cdot W,
\]
then $N^{\rm sym} \stackrel{w}\ll M$ and $N^{\rm asym} \stackrel{w}\ll M$. Indeed, if $N^{\rm sym} = N^{{\rm sym},c} + N^{{\rm sym},d}$ and $N^{\rm asym} = N^{{\rm asym},c} + N^{{\rm asym},d}$ are the corresponding Meyer-Yoeurp decompositions, then $N^{{\rm sym},d} = N^d \stackrel{w}\ll M^d$, $N^{{\rm asym},d} = 0\stackrel{w}\ll M^d$, and for any $x^* \in X^*$ and $t\geq 0$, we have
\begin{align*}
 [\langle N^{{\rm sym},c}, x^*\rangle]_t = \int_0^t \bigl\|\tfrac{A(s)+A^T(s)}{2} \phi^*(s)x^*\bigr\|^2 \ud s &\leq \int_0^t \bigl\|\tfrac{A(s)+A^T(s)}{2}\bigr\|^2\| \phi^*(s)x^*\|^2 \ud s\leq \int_0^t\| \phi^*(s)x^*\|^2\ud s = [\langle M^c, x^*\rangle]_t.
\end{align*}
Here $\bigl\|\tfrac{A(s)+A^T(s)}{2}\bigr\|\leq 1$ by the triangle inequality. Therefore $N^{{\rm sym},c}\stackrel{w}\ll M^c$ and, analogously, $N^{{\rm asym},c}\stackrel{w}\ll M^c$, so the weak differential subordination holds by virtue of Proposition \ref{prop:MYdecconttocontpdtopd}.

Let us now show that
 \begin{equation}\label{eq:asympartLpineq}
    \mathbb E \|N^{\rm asym}_t\|^p \leq \hbar_{p, X}^p \mathbb E \|M_t\|^p\;\;\;\mbox{ for } t\geq 0.
 \end{equation}
We have $N^{\rm asym}_0=0$ and $N^{\rm asym} \stackrel{w}\ll M$; we will prove in addition that $M$ and $N^{\rm asym}$ are orthogonal. For fixed $x^*\in X^*$ and $t\geq 0$ we may write
\begin{align*}
 [\langle M, x^*\rangle,\langle N^{\rm asym}, x^*\rangle]_t &=  [\langle M^c, x^*\rangle,\langle N^{\rm asym}, x^*\rangle]_t +  [\langle M^d, x^*\rangle,\langle N^{\rm asym}, x^*\rangle]_t\\
&= [\langle M^c, x^*\rangle,\langle N^{\rm asym}, x^*\rangle]_t = [\langle \phi\cdot W, x^*\rangle,\langle (\phi A^{\rm asym})\cdot W, x^*\rangle]_t\\
&= [\langle \phi, x^*\rangle\cdot W,\langle (\phi A^{\rm asym}), x^*\rangle\cdot W]_t=\int_0^t \langle \phi^*(s)x^*, A^{{\rm asym}*}(s) \phi^*(s)x^*\rangle\ud s = 0,
\end{align*}
where the second equality is a consequence of pure discontinuity of $M^d$ and continuity of $N^{\rm asym}$, while the last equality follows from the fact that $A^{\rm asym}$ is antisymmetric. This gives the orthogonality of the processes and \eqref{eq:asympartLpineq} follows from \eqref{eq:HTnormboudfororcontmart}.

The next step is to show that
\begin{equation}\label{eq:sympartLpineq}
    \mathbb E \|N^{\rm sym}_t\|^p \leq \beta_{p, X}^p \mathbb E \|M_t\|^p\;\;\; \mbox{ for } t\geq 0.
 \end{equation}
Let $U:X\times X\to \mathbb R$ be the Burkholder function guaranteed  by Theorem \ref{thm:Burkholder}. Using the same argument as in \cite{BO12}, we may assume that $U$ is of class $C^\infty$ (see also Remark \ref{rem:UisCinfty}). Applying It\^o's formula \eqref{eq:itoformula} for a fixed basis $(x_i)_{i=1}^d$ of $X$ with the dual basis $(x_i^*)_{i=1}^d$ of $X^*$, we get
\begin{align*}
 \mathbb E U(M_t, N^{\rm sym}_t) &= \mathbb E U(M_0, N^{\rm sym}_0)+ \frac{1}{2}\mathbb E I_1 +  \mathbb E I_2,
\end{align*}
where

\begin{equation}\label{eq:defofI_1forbeta+hbar}
 \begin{split}
    I_1 &:= \int_0^t \sum_{i,j=1}^d \frac{\partial^2 U(M_{s-}, N^{\rm sym}_{s-})}{\partial x_i \partial x_j}  \ud [\langle M, x_i^*\rangle,\langle M, x_j^*\rangle]^c_s\\
  &\quad+ \int_0^t \sum_{i,j=1}^d \frac{\partial^2 U(M_{s-}, N^{\rm sym}_{s-})}{\partial y_i \partial y_j}  \ud [\langle N^{\rm sym}, x_i^*\rangle,\langle N^{\rm sym}, x_j^*\rangle]^c_s\\
  &\quad+ 2\int_0^t \sum_{i,j=1}^d \frac{\partial^2 U(M_{s-}, N^{\rm sym}_{s-})}{\partial x_i \partial y_j}  \ud [\langle M, x_i^*\rangle,\langle N^{\rm sym}, x_j^*\rangle]^c_s
 \end{split}
\end{equation}
and
\begin{align*}
  I_2 :=\sum_{0\leq s\leq t}(\Delta U(M_s, N^{\rm sym}_s)&- \langle \partial_x U(M_{s-}, N^{\rm sym}_{s-}), \Delta M_s \rangle- \langle \partial_y U(M_{s-}, N^{\rm sym}_{s-}), \Delta N^{\rm sym}_s \rangle).
\end{align*}
{Here $\partial_x U(\cdot),\partial_y U(\cdot)\in X^*$ are the corresponding Fr\'echet derivatives of $U$ in the first and the second $X$-subspace of the product space $X\times X$.}
Let us first show that $\mathbb E I_1 \leq 0$.
{Indeed, note that}
\begin{equation}\label{eq:Itoform2ndtermBurkfun}
 \begin{split}
   &\sum_{i,j=1}^d\frac{\partial^2 U(M_{s-}, N^{\rm sym}_{s-})}{\partial x_i \partial x_j} \langle \phi^* x_i^*, \phi^* x_j^*\rangle \\
  &\quad+ \sum_{i,j=1}^d\frac{\partial^2 U(M_{s-}, N^{\rm sym}_{s-})}{\partial y_i \partial y_j}  \langle A^{{\rm sym}*} \phi^* x_i^*, A^{{\rm sym}*} \phi^* x_j^*\rangle \\
  &\quad+ 2 \sum_{i,j=1}^d \frac{\partial^2 U(M_{s-}, N^{\rm sym}_{s-})}{\partial x_i \partial y_j}\langle \phi^* x_i^*,A^{{\rm sym}*} \phi^* x_j^*\rangle \leq 0.
 \end{split}
\end{equation}
Note {also} that by Corollary \ref{cor:lemmatraceindep} and convexity of $U$ in the second variable,
\begin{equation}\label{eq:withoutABurkfunraioni}
 \begin{split}
  \sum_{i,j=1}^d \frac{\partial^2 U(M_{s-}, N^{\rm sym}_{s-})}{\partial y_i \partial y_j}  \langle A^{{\rm sym}*} \phi^* x_i^*, A^{{\rm sym}*}\phi^* x_j^*\rangle\leq \sum_{i,j=1}^d \frac{\partial^2 U(M_{s-}, N^{\rm sym}_{s-})}{\partial y_i \partial y_j}  \langle  \phi^* x_i^*,  \phi^* x_j^*\rangle.
 \end{split}
\end{equation}
The operator $P_{\text{Ran}(\phi^*)}A^{{\rm sym}*}P_{\text{Ran}(\phi^*)}$ is symmetric and
$$
\|P_{\text{Ran}(\phi^*)}A^{{\rm sym}*}P_{\text{Ran}(\phi^*)}\|\leq 1.
$$
Therefore by the spectral theorem there exist a $[-1,1]$-valued sequence $(\lambda_i)_{i=1}^{2d}$ and an orthonormal basis $(\tilde h_i)_{i=1}^{2d}$ of $(\mathbb R^{2d})^*$ such that $P_{\text{Ran}(\phi^*)}A^{{\rm sym}*}P_{\text{Ran}(\phi^*)} \tilde h_i = \lambda_i \tilde h_i$. Moreover, since $\text{Ran}(P_{\text{Ran}(\phi^*)}A^{{\rm sym}*}P_{\text{Ran}(\phi^*)})\subset \text{Ran}(\phi^*)$, $\tilde h_i \in \text{Ran}(\phi^*)$ if $\lambda_i \neq 0$, so we may assume that there exists a basis
$(\tilde x_i)_{i=1}^d$ of $X$ with the dual basis $(\tilde x^*_i)_{i=1}^d$ such that $ \phi^*\tilde x^*_i = \tilde h_i$ for $1\leq i\leq m$ and $ \phi^*\tilde x^*_i = 0$ for $m<i\leq d$, where $m\in \{0,\ldots,d\}$ is the dimension of $\phi^*$. By Lemma \ref{lemma:traceindep} the expression on the left-hand side of \eqref{eq:Itoform2ndtermBurkfun} does not depend on the choice of the basis of $X$ and the corresponding dual basis. Therefore, using \eqref{eq:withoutABurkfunraioni}, it is not bigger than
\begin{equation*}
   \sum_{i=1}^m\frac{\partial^2 U(M_{s-}, N^{\rm sym}_{s-})}{\partial x_i \partial x_i}
 + \sum_{i=1}^m \frac{\partial^2 U(M_{s-}, N^{\rm sym}_{s-})}{\partial y_i \partial y_i}
 + 2 \sum_{i=1}^m \lambda_i \frac{\partial^2 U(M_{s-}, N^{\rm sym}_{s-})}{\partial x_i \partial y_i},
\end{equation*}
which is bounded from above by $0$ (see Lemma \ref{lem:twdiffzigzagfnction}). Thus, \eqref{eq:Itoform2ndtermBurkfun} follows.
Therefore by \eqref{eq:defofI_1forbeta+hbar} and \eqref{eq:Itoform2ndtermBurkfun}, we see that
\begin{align*}
 I_1 &= \int_0^t \sum_{i,j=1}^d \frac{\partial^2 U(M_{s-}, N^{\rm sym}_{s-})}{\partial x_i \partial x_j}\langle \phi^* x_i^*, \phi^* x_j^*\rangle \ud s\\
  &\quad+ \int_0^t \sum_{i,j=1}^d \frac{\partial^2 U(M_{s-}, N^{\rm sym}_{s-})}{\partial y_i \partial y_j}  \langle A^{{\rm sym}*} \phi^* x_i^*, A^{{\rm sym}*} \phi^* x_j^*\rangle \ud s\\
  &\quad+ 2\int_0^t \sum_{i,j=1}^d \frac{\partial^2 U(M_{s-}, N^{\rm sym}_{s-})}{\partial x_i \partial y_j} \langle \phi^* x_i^*,A^{{\rm sym}*} \phi^* x_j^*\rangle \ud s \leq 0,
\end{align*}
and hence the expectation of $I_1$ is nonpositive. The inequality $I_2 \leq 0$ can be proved by repeating the arguments from \cite[Proof of Theorem 3.18]{Y17FourUMD}, while for the estimate
$ U(M_0,N^{\rm sym}_0) \leq 0$, consult \cite[Remark 3.10]{Y17FourUMD}. Therefore, we have
\[
 \mathbb E \|N^{\rm sym}_t\|^p - \beta_{p, X}^p \mathbb E \|M_t\|^p \leq \mathbb E U(M_t, N^{\rm sym}_t) \leq \mathbb E U(M_0,N^{\rm sym}_0) \leq 0,
\]
so \eqref{eq:sympartLpineq} holds. The general inequality \eqref{eq:thmWDSofmartbeta+h} follows from \eqref{eq:asympartLpineq}, \eqref{eq:sympartLpineq}, and the triangle inequality.
\end{proof}

{
\begin{remark}
 It is an open problem whether there exists a Burkholder function $U$ such that $-U$ is plurisubharmonic (note that $X\times X\simeq X+ iX$, so the plurisubharmonicity condition is well-defined). If it exists, then $\hbar_{p, X}\leq \beta_{p, X}$ by Theorem \ref{thm:existenceofU^H_Phi,Psi}, and so the open problem outlined in Remark \ref{rem:h_p,Xvsbeta_p,X} is solved. Unfortunately, plurisubharmonicity of $-U$ is discovered only in the Hilbert space case (see \cite{Wang} and \cite[Remark 5.6]{Y17FourUMD}).
\end{remark}
}

\subsection{Weak differential subordination of harmonic functions}\label{subsec:WDSharmfunc}

Let $X$ be a Banach space, let $d\geq 1$ be a fixed dimension and let $\mathcal O$ be an open subset of $\mathbb R^d$. A function $f:\mathcal O \to X$ is called {\em harmonic} if it takes its values in a finite-dimensional subspace of $X$, is twice-differentiable, and
$$
\Delta f(s) := \sum_{i=1}^d \partial_i^2 f(s) = 0,\;\;\; s\in \mathcal O.
$$


For each $s\in \mathcal O$, we define $\nabla f(s)\in \mathcal L(\mathbb R^d, X)$ by
\[
 \nabla f(s) (a_1e_1 + \cdots a_de_d) = \sum_{i=1}^d a_i \partial_i f(s),\;\;\;a_1,\ldots,a_d \in \mathbb R,
\]
where $(e_i)_{i=1}^d$ is the basis of $\mathbb R^d$.

\begin{definition}
 Let $X$, $d$, $\mathcal O$ be as above and assume that $f,g:\mathcal O\to X$ are harmonic functions. Then
 \begin{enumerate}
  \item $g$ is said to be {\em weakly differentially subordinate} to $f$ (which will be denoted by $g \stackrel{w}\ll f$) if
  \begin{equation}\label{eq:defWDSofharmfunc}
   |\langle \nabla g(s), x^* \rangle| \leq |\langle \nabla f(s), x^* \rangle|,\;\;\; s\in \mathcal O, x^* \in X^*;
  \end{equation}
\item $f$ and $g$ are said to be {\em orthogonal} if
\begin{equation}\label{eq:deforthofharmfunc}
 \Bigl\langle \langle \nabla f(s), x^* \rangle,\langle \nabla g(s), x^* \rangle \Bigl\rangle=0,\;\;\;s\in \mathcal O, x^* \in X^*.
\end{equation}
 \end{enumerate}
\end{definition}
Here $|\cdot|$ in \eqref{eq:defWDSofharmfunc} is assumed to be the usual Euclidean norm in $(\mathbb R^d)^* \simeq \mathbb R^d$, and $\langle \cdot, \cdot\rangle$ in \eqref{eq:deforthofharmfunc} is the usual scalar product in $(\mathbb R^d)^* \simeq \mathbb R^d$.

The notion of weak differential subordination of vector-valued harmonic functions extends the concept originally formulated in the one-dimensional case by Burkholder \cite{Burk89}. As shown in that paper, the differential subordination of harmonic functions lead to the corresponding $L^p$-inequalities for $1<p<\infty$. The aim of this subsection is to show the extension of that result to general weakly differentially subordinated harmonic functions and to show more general $\Phi, \Psi$-type estimates under the orthogonality assumption. We start with recalling the  definition of a harmonic measure.

\begin{definition}
 Let $\mathcal O\subset \mathbb R^d$ be an open set containing the origin and let $\partial \mathcal O$ be the {\em boundary} of $\mathcal O$. The probability measure $\mu$ on $\partial \mathcal O$ is called a {\em harmonic measure with respect to the origin}, if for any Borel subset $A \subset \partial \mathcal O$ we have
 \[
  \mu(A):= \mathbb P\{W_{\tau}\in A\}.
 \]
 Here $W:\mathbb R_+ \times \Omega \to \mathbb R^d$ is a standard Brownian motion starting from $0$ and $\tau$ is the exit-time of $W$ from $\mathcal{O}$.
\end{definition}

\begin{theorem}\label{thm:WDSoforthharmfunc}
 Let $X$ be a Banach space, let $d\geq 1$ be a fixed dimension and let $\mathcal O$ be an open, bounded subset of $\R^d$ containing the origin. Assume further that $\Phi,\Psi:X \to \mathbb R_+$ are continuous functions such that $\Psi$ is convex and $\Psi(0)=0$. Then for any continuous functions $f, g:\overline{\mathcal O}\to X$ harmonic and orthogonal on $\mathcal O$ satisfying  $g \stackrel{w}\ll f$ and $g(0)=0$ we have
 \[
  \int_{\partial \mathcal O} \Psi(g(s))\ud \mu(s) \leq C_{\Phi,\Psi,X}\int_{\partial \mathcal O} \Phi(f(s))\ud \mu(s).
 \]
Here $\mu$ is the harmonic measure on $\partial \mathcal O$ with respect to the origin and the least admissible $C_{\Phi,\Psi,X}$ equals $|\mathcal H^{\mathbb T}_{X}|_{\Phi,\Psi}$.
\end{theorem}

\begin{remark}
 We do not assume that $\Phi$ is convex because both $f$ and $g$ take their values in a finite-dimensional subspace of $X$, see Remark \ref{rem:PhiisnotconvexifXisfd}.
\end{remark}

\begin{proof}[Proof of Theorem \ref{thm:WDSoforthharmfunc}]
 Let $W:\mathbb R_+\times \Omega \to \mathbb R^d$ be a standard Brownian motion and let $
 \tau:= \inf\{t\geq 0: W_t\notin \mathcal O\}.
 $
 Then both $M:=f(W^{\tau})$ and $N:=g(W^{\tau})$ are martingales since both $f$ and $g$ are harmonic on $\mathcal O$ (see e.g.\ \cite[Theorem 18.5]{Kal}). By It\^o's formula and the fact that both $f$ and $g$ are harmonic we have
 \begin{align*}
  M_t&=f(W^{\tau}_t) = f(0) + \int_0^t \nabla f(W^{\tau}_s) \ud W^{\tau}_s,\;\;\; t\geq 0,\\
   N_t&=g(W^{\tau}_t) = \int_0^t \nabla g(W^{\tau}_s) \ud W^{\tau}_s,\;\;\; t\geq 0,
 \end{align*}
where in the second line we have used the equality $g(0)=0$. Therefore for any $x^*\in X^*$ and any $0\leq u\leq t$ we have
 \begin{align*}
  [\langle N, x^*\rangle]_t-[\langle N, x^*\rangle]_u = \int_u^t\|\langle\nabla g(W^{\tau}_s), x^*\rangle\|^2 \ud s\leq \int_u^t\|\langle\nabla f(W^{\tau}_s), x^*\rangle\|^2 \ud s = [\langle M, x^*\rangle]_t-[\langle M, x^*\rangle]_u,
 \end{align*}
and
 \begin{align*}
  [\langle M, x^*\rangle,\langle N, x^*\rangle]_t &= \int_0^t\Bigl\langle \langle\nabla g(W^{\tau}_s), x^*\rangle ,\langle\nabla f(W^{\tau}_s), x^*\rangle\Bigr\rangle\ud s =0.
 \end{align*}
 Consequently, $M$ and $N$ are orthogonal and $N \stackrel{w}\ll M$, so
 \begin{align*}
  \int_{\partial \mathcal O} \Psi(g(s))\ud \mu(s) = \lim_{t\to \infty}\mathbb E \Psi(g(W^\tau_t))\leq \lim_{t\to \infty}|\mathcal H_X^{\mathbb T}|_{\Phi, \Psi}\mathbb E \Phi(f(W^\tau_t))  = |\mathcal H_X^{\mathbb T}|_{\Phi, \Psi}\int_{\partial \mathcal O} \Phi(f(s))\ud \mu(s).
 \end{align*}
Here the first and the last equality follow from the dominated convergence theorem and the definition of $\mu$, while the middle one is due to Theorem \ref{thm:orthmartPhiPsi}.

The sharpness of the constant $C_{\Phi,\Psi,X}=|\mathcal H_X^{\mathbb T}|_{\Phi,\Psi}$ follows from the case $d=2$, $\mathcal O\subset \mathbb R^2$ being the unit disc, $f$ and $g$ being such that $g|_{\partial \mathcal O} = \mathcal H_X^{\mathbb T} (f|_{\partial \mathcal O})$ (in this case $\mu$ becomes the probability Lebesgue measure on the unit circle $\partial \mathcal O$).
\end{proof}

{
\begin{remark}
 Sharpness of the estimate
  \[
  \int_{\partial \mathcal O} \Psi(g(s))\ud \mu(s) \leq |\mathcal H^{\mathbb T}_{X}|_{\Phi,\Psi}\int_{\partial \mathcal O} \Phi(f(s))\ud \mu(s)
 \]
 for a fixed domain $\mathcal O$ remains open. Nevertheless, in the case $d=2$ and $\mathcal O$ being bounded with a Jordan boundary (e.g.\ polygon-shaped) the sharpness follows immediately from the Carath\'eodory's theorem (see e.g.\ \cite[Subsection I.3 and Appendix F]{GM08}).
\end{remark}

}

Let us turn to the corresponding result for $L^p$-estimates for differentially subordinate harmonic functions (i.e., not necessarily orthogonal).

\begin{theorem}\label{thm:WDSHFpestimate}
 Let $X$, $d$ and $\mathcal{O}$ be as in the previous statement. Assume further that $f,g:\overline {\mathcal O} \to X$ are continuous functions harmonic on $\mathcal O$ satisfying $g \stackrel{w}\ll f$ and $g(0) = a_0f(0)$ for some $a_0\in [-1,1]$. Then for any $1<p<\infty$ we have
 \begin{equation}\label{eq:thmWDSHFpestimate}
    \Bigl(\int_{\partial \mathcal O}\|g(s)\|^p \ud \mu(s)\Bigr)^{\frac{1}{p}} \leq C_{p, X}  \Bigl(\int_{\partial \mathcal O}\|f(s)\|^p \ud \mu(s)\Bigr)^{\frac{1}{p}},
 \end{equation}
where $\mu$ is the harmonic measure of $\partial \mathcal O$, and the least admissible constant $C_{p, X}$ is within the segment $[\hbar_{p,X}, \beta_{p, X} + \hbar_{p,X}]$.
\end{theorem}

\begin{remark}\label{rem:WDSHFscalarcaseopenproblem}
 In the scalar-valued setting it is known that the optimal $C_{p, \mathbb R}$ is within the range $[\cot(\tfrac{\pi}{2p^*}), p^*-1]$. The precise identification of $C_{p,\mathbb R}$ is an open problem formulated by Burkholder in \cite{Burk89}.
\end{remark}

\begin{proof}[Proof of Theorem \ref{thm:WDSHFpestimate}]
This is quite similar to the proof of the latter statement, so we will be brief and only indicate the necessary changes which need to be implemented. For the lower bound $C_{p, X}\geq \hbar_{p,X}$, modify appropriately the last sentence of the proof of Theorem \ref{thm:WDSoforthharmfunc}. To show the upper bound for $C_{p,X}$,
consider the martingales $M:= f(W^{\tau})$ and $N:= g(W^{\tau})$, where $W$ and $\tau$ are as previously. Arguing
as in the proof of Theorem \ref{thm:WDSoforthharmfunc}, we show that $N \stackrel{w}\ll M$ and hence {by Theorem \ref{thm:WDSofmartbeta+h}},
\begin{align*}
 \Bigl(\int_{\partial \mathcal O}\|g(s)\|^p \ud \mu(s)\Bigr)^{\frac{1}{p}} &= \lim_{t\to \infty} (\mathbb E \|N_t\|^p)^{\frac{1}{p}} \leq \limsup_{t\to \infty} (\beta_{p,X} +\hbar_{p,X}) (\mathbb E \|M_t\|^p)^{\frac{1}{p}}\\
 &\leq \lim_{t\to \infty} (\beta_{p,X} +\hbar_{p,X}) (\mathbb E \|M_t\|^p)^{\frac{1}{p}}=(\beta_{p,X} + \hbar_{p,X}) \Bigl(\int_{\partial \mathcal O}\|f(s)\|^p \ud \mu(s)\Bigr)^{\frac{1}{p}}.
\end{align*}
This completes the proof.
\end{proof}

\begin{remark}
 Note that any significant improvement for the upper bound of $C_{p, X}$ in \eqref{eq:thmWDSHFpestimate} could automatically solve an open problem. Let us outline two remarkable examples. If one could show that $C_{p, X} \leq C \beta_{p, X}$ for some universal constant $C>0$, then the open problem outlined in Remark \ref{rem:h_p,Xvsbeta_p,X} will be solved. On the other hand, if one could show that $C_{p, X} = \hbar_{p, X}$, then the question of Burkholder concerning the optimal constant $C_{p,\mathbb R}$ in the real-valued case would be answered (see Remark \ref{rem:WDSHFscalarcaseopenproblem}).
\end{remark}

\subsection{Inequalities for singular integral operators}\label{subsec:singint}
Our final application concerns the extension of $\Phi,\,\Psi$-estimates from the setting of nonperiodic Hilbert transform to the case of odd-kernel singular integral operators on $\R^d$. We start with the notion of a \emph{directional Hilbert transform}: given a unit vector $\theta\in \R^d$, we define the operator $\mathcal{H}_\theta$ by
$$
 \mathcal{H}_\theta f(x)=\frac{1}{\pi}\mbox{p.v.}\int_\R f(x-t\theta)\frac{\mbox{d}t}{t}, {\;\;\;x\in \mathbb R^d,}
 $$
where $f$ is a sufficiently regular real-valued function on $\R^d$,
and call it the Hilbert transform of $f$ in the direction $\theta$. For example, if $e_1$ stands for the  unit vector $(1,0,0,\ldots,0)\in \R^d$, then $\mathcal{H}_{e_1}$ is obtained by applying the Hilbert transform in the first variable followed by the identity operator in the remaining variables. Consequently, by Fubini's theorem, we see that for any functions $\Phi$, $\Psi:X\to [0,\infty)$ and any step function $f:\R^d\to X$ (finite linear combination of characteristic functions of rectangles) we have
$$
 \int_{\R^d}\Psi(\mathcal{H}_{e_1}f)\mbox{d}x\leq |\mathcal{H}^\mathbb{R}_X|_{\Phi,\Psi}\int_{\R^d} \Phi(f)\mbox{d}x.
 $$
Now, if $A$ is an arbitrary orthogonal matrix, we have
$$
\mathcal{H}_{Ae_1}(f)(x)=\mathcal{H}_{e_1}(f\circ A)(A^{-1}x), { \;\;\;x\in \mathbb R^d,}
$$
so the above inequality holds true for any directional Hilbert transform $\mathcal{H}_\theta$.

Suppose that $\Omega:S^{d-1}\to \R$ is an odd function satisfying $||\Omega||_{L^1(S^{d-1})}=1$ and define the associated operator
$$
T_\Omega f(x)=\frac{2}{\pi}\mbox{p.v.}\int_{\R^d} \frac{\Omega(y/|y|)}{|y|^d}f(x-y)\mbox{d}y, { \;\;\;x\in \mathbb R^d.}
$$
Then $T_\Omega$ can be expressed as an average of directional Hilbert transforms:
$$
 T_\Omega f(x)=\int_{S^{d-1}} \Omega(\theta)\mathcal{H}_\theta f(x)\mbox{d}\theta,  { \;\;\;x\in \mathbb R^d}.
 $$
(Sometimes this identity is referred to as the method of rotations.)
Consequently, if $\Psi$ is convex and even, we get
\begin{align*}
 \int_{\R^d}\Psi(T_\Omega f)\mbox{d}x=\int_{\R^d} \Psi\left(\int_{S^{d-1}} \Omega(\theta)\mathcal{H}_\theta f(x)\mbox{d}\theta\right)\mbox{d}x\leq \int_{S^{d-1}}|\Omega(\theta)|\int_{\R^d} \Psi(\mathcal{H}_\theta f(x))\mbox{d}x\mbox{d}\theta\leq |\mathcal{H}^\mathbb{R}_{\Phi,\Psi}|\int_{\R^d} \Phi(f)\mbox{d}x.
\end{align*}
In particular, if we fix $d$ and $j\in \{1,2,\ldots,d\}$, then the kernel
$$
 \Omega_{j,d}(\theta)=\frac{\pi\Gamma\left(\frac{d+1}{2}\right)}{2\pi^{(d+1)/2}}{\theta_j},\qquad \theta\in S^{d-1},
 $$
gives rise to the Riesz transform $R_j$. Therefore, we see that any $\Phi,\Psi$-estimate for the nonperiodic Hilbert transform (where $\Psi$ is assumed to be a convex and even function on $X$) holds true, with an unchanged constant, also in the context of Riesz transforms. In particular, the $L^p$ norms of Riesz transforms are dominated by the $L^p$ norms of Hilbert transform.

\smallskip
The following theorem connects the $\Phi,\Psi$-norm of an odd power of a Riesz transform with the $\Phi,\Psi$-norm of the Hilbert transform.

\begin{theorem}
Let $X$ be a Banach space, $d\geq 1$, $j\in \{1,\ldots,d\}$, $m\geq 1$ be odd. Let $R_{j,X}$ be the corresponding Riesz transform acting on $X$-valued step functions, $\Phi, \Psi:X \to \mathbb R_+$ be convex continuous such that $\Psi$ is even. Then
\[
|R_{j,X}^m|_{\Phi, \Psi} \leq \Bigl|\frac{2\Gamma(\tfrac{m+d}{2})}{\Gamma(\tfrac{d}{2})\Gamma(\tfrac{m}{2})}\mathcal H^{\mathbb R}_X\Bigr|_{\Phi, \Psi}.
\]
\end{theorem}

\begin{proof}
The proof follows from the discussion above, the fact that $R_{j,X}^m$ is a singular integral of the following form (see e.g.\ \cite[p.\ 33]{IM96}):
\[
R_{j,X}^m f(x) = \frac{\Gamma(\tfrac{m+d}{2})}{\pi^{\tfrac d2}\Gamma(\tfrac{m}{2})}\int_{\mathbb R^d}\frac{f(x-y)y_j^m}{|y|^{m+d}}\ud y,\;\;\; x\in \mathbb R^d,
\]
where $f:\mathbb R^d \to X$ is a step function, and the fact that the volume of $S^{d-1}$ equals ${2\pi^{\tfrac d2}}/{\Gamma(\tfrac d2)}$.
\end{proof}

Notice that if $d$ is fixed, then $\tfrac{2\Gamma(\tfrac{m+d}{2})}{\Gamma(\tfrac{d}{2})\Gamma(\tfrac{m}{2})}$ is of the order $m^{d/2}$, so in particular we have that for all $1<p<\infty$,
\[
\|R_{j,X}^m\|_{L^p(\mathbb R^d;X) \to L^p(\mathbb R^d;X)} \lesssim_d m^{d/2} \|\mathcal H^{\mathbb R}_X\|_{L^p(\mathbb R; X) \to L^p(\mathbb R; X)}.
\]

\subsection{Hilbert operators}\label{sec:Hilbert operators}

Let $X$ be a Banach space, let $d$ be a positive integer and pick $j\in\{1,\ldots, d\}$. Let $f:\mathbb R_{j+}^d \to X$ be locally integrable function, where $\R^d_{j+}=\{x\in \R^d:x_j>0\}$.  We define $T_j f:\mathbb R_{j+}^d \to X$ by the formula
\[
 T_j f(x) := \frac{\Gamma(\frac{d+1}{2})}{\pi^{(d+1)/2}}\int_{\mathbb R_{j+}^d}\frac{f(y)(x_j+y_j)}{|x+y|^{d+1}}\ud y,\;\;\; x\in \mathbb R_{j+}^d.
\]
This type of operators resembles Riesz transforms, but due to the domain restrictions the use of principal value is not necessary.
Note that if $d=1$, then $T_{j}$ is the {\em Hilbert operator} $T$  given by
\[
 T f(x) := \frac{1}{\pi}\int_{\mathbb R_{+}}\frac{f(y)}{x+y}\ud y,\;\;\; x\in \mathbb R_{+}.
\]

We have the following statement.

\begin{theorem}
 Let $X$ be a Banach space, $\Phi, \Psi:X \to \mathbb R_+$ be convex continuous such that $\Psi$ is even, $d\geq 1$, $j\in\{1,\ldots, d\}$, $1<p<\infty$. Then
 \begin{equation}\label{eq:ineqonT_j}
|T_j|_{\Phi,\Psi} \leq |\mathcal H^{\mathbb R}_X|_{\Phi, \Psi}.
 \end{equation}
\end{theorem}

\begin{proof}
By the discussion in Subsection \ref{subsec:singint} it is sufficient to show that
$$
 |T_j|_{\Phi,\Psi} \leq |R_{j,X}|_{\Phi, \Psi}.
 $$
 Fix a step function $f:\mathbb R^d_{j+}\to X$. Let $\tilde f:\mathbb R^d \to X$ be such that $\tilde f(x_1,\ldots,x_d) = 0$ if $x_j<0$ and $\tilde f|_{\mathbb R^d_{j+}} = f$. Then $T_j f(x) = R_{j, X} \tilde f(-x)$ for any $x\in \mathbb R^d_{j+}$, and therefore
 \begin{align*}
\int_{\mathbb R_{j+}^d} \Psi(T_j f(x)) \ud x &= \int_{\mathbb R^d} \Psi(R_{j, X} \tilde f(-x))\mathbf 1_{x_j>0} \ud x \leq \int_{\mathbb R^d} \Psi(R_{j, X} \tilde f(-x)) \ud x\\
&=\int_{\mathbb R^d} \Psi(R_{j, X} \tilde f(x)) \ud x\leq |R_{j,X}|_{\Phi, \Psi}\int_{\mathbb R^d} \Phi(\tilde f(x))\ud x\\
& = |R_{j,X}|_{\Phi, \Psi}\int_{\mathbb R^d_{j+}} \Phi(f(x))\ud x.
 \end{align*}
\end{proof}

\begin{remark}\label{rem:T_jnicephipsi}
Notice that if $\Phi$ and $\Psi$ are of the form $\Phi(x) = \phi(\|x\|)$, $\Psi(x) = \psi(\|x\|)$ for some convex symmetric functions $\phi, \psi:\mathbb R \to \mathbb R_+$, then one can improve \eqref{eq:ineqonT_j}. Indeed, one can show that $|T_j|_{\Phi, \Psi} = |T_j|_{\phi, \psi}$, which does not depend on the Banach space $X$: for any step function $f:\mathbb R^d_{j+} \to X$ one has that
\begin{align*}
\int_{\mathbb R_{j+}^d}\Psi(T_j f(x)) \ud x &= \int_{\mathbb R_{j+}^d}\psi(\|T_j f(x)\|) \ud x\\
& = \int_{\mathbb R_{j+}^d}\psi\Bigl(\Bigl\| \frac{\Gamma(\frac{d+1}{2})}{\pi^{(d+1)/2}}\int_{\mathbb R_{j+}^d}\frac{f(y)(x_j+y_j)}{|x+y|^{d+1}}\ud y\Bigr\|\Bigr) \ud x\\
&\leq \int_{\mathbb R_{j+}^d}\psi\Bigl( \frac{\Gamma(\frac{d+1}{2})}{\pi^{(d+1)/2}}\int_{\mathbb R_{j+}^d}\frac{g(y)(x_j+y_j)}{|x+y|^{d+1}}\ud y\Bigr) \ud x\\
&= \int_{\mathbb R_{j+}^d}\psi(T_j g(x)) \ud x \leq |T_j|_{\phi, \psi}\int_{\mathbb R_{j+}^d}\phi(g(x))\ud x \\
&=  |T_j|_{\phi, \psi}\int_{\mathbb R_{j+}^d}\Phi(f(x))\ud x,
\end{align*}
where $g:\mathbb R^d_{j+}\to \mathbb R_+$ is a step function such that $g(\cdot) = \|f(\cdot)\|$. In particular, if $\Phi(x) = \Psi(x) = \|x\|^p$ for some $1<p<\infty$, then by \cite[Theorem 1.1]{Os17a}
\[
\|T_j\|_{L^p(\mathbb R^d_{j+};X) \to L^p(\mathbb R^d_{j+};X)} = \sin^{-1}(\pi/p).
\]
\end{remark}

\section*{Acknowledgment}
The authors would like to thank Mark Veraar for pointing out the estimate \eqref{eq:deltabeta+beta-estforHilbTrans} and Remark \ref{rem:T_jnicephipsi} to them. The authors would also like to express their gratitude to the referee for many helpful comments and suggestions which greatly improved the paper. A. Os\k ekowski was supported by Narodowe Centrum Nauki (Poland), grant no. DEC-2014/14/E/ST1/00532.

\def\cprime{$'$} \def\polhk#1{\setbox0=\hbox{#1}{\ooalign{\hidewidth
  \lower1.5ex\hbox{`}\hidewidth\crcr\unhbox0}}}
  \def\polhk#1{\setbox0=\hbox{#1}{\ooalign{\hidewidth
  \lower1.5ex\hbox{`}\hidewidth\crcr\unhbox0}}} \def\cprime{$'$}

\end{document}